\documentclass[a4paper,10pt]{amsart}
\usepackage[utf8]{inputenc} 
\usepackage[T1]{fontenc}
\usepackage{mathtools,amssymb,amsthm} 
\usepackage{dsfont}
\usepackage[all]{xy}
\usepackage{xparse} 
\usepackage{microtype}
\usepackage{mathrsfs}
\usepackage{varioref}
\usepackage[english,french]{babel} 
\usepackage{hyperref}

\input{macros.tex}

\title[Sur les $\ell$-blocs de niveau zéro des groupes $p$-adiques II]{Sur les $\ell$-blocs de niveau zéro des groupes $p$-adiques II}
\author{Thomas Lanard}
\email{thomas.lanard@univie.ac.at}

\begin{document}

\begin{abstracts}
\abstractin{french}
Soient $G$ un groupe $p$-adique se déployant sur une extension non-ramifiée et $\rep[\ld][0]{\G}$ la catégorie abélienne des représentations lisses de $\G$ de niveau $0$ à coefficients dans $\ld=\Ql$ ou $\Zl$. Nous étudions la plus fine décomposition de $\rep[\ld][0]{\G}$ en produit de sous-catégories que l'on peut obtenir par la méthode introduite dans \cite{lanard}, la seule méthode connue à ce jour lorsque $\ld=\Zl$ et $G$ n'est pas forme intérieure de $\gl{n}$. Nous en donnons deux descriptions, une première du côté du groupe à la Deligne-Lusztig, puis une deuxième du côté dual à la Langlands. Nous prouvons plusieurs propriétés fondamentales comme la compatibilité à l'induction et la restriction parabolique ou à la correspondance de Langlands locale. Les facteurs de cette décomposition ne sont pas des blocs, mais on montre comment les regrouper pour obtenir les blocs "stables".

\abstractin{english}
Let $\G$ be a $p$-adic group which splits over an unramified extension and $\rep[\ld][0]{\G}$ the abelian category of smooth level $0$ representations of $\G$ with coefficients in $\ld=\Ql$ or $\Zl$. We study the finest decomposition of $\rep[\ld][0]{\G}$ into a product of subcategories that can be obtained by the method introduced in \cite{lanard}, which is currently the only one available when $\ld=\Zl$ and $G$ is not an inner form of $\gl{n}$. We give two descriptions of it, a first one on the group side à la Deligne-Lusztig, and a second one on the dual side à la Langlands. We prove several fundamental properties, like for example the compatibility with parabolic induction and restriction or the compatibility with the local Langlands correspondence. The factors of this decomposition are not blocks, but we show how to group them to obtain "stable" blocks.
\end{abstracts}

\maketitle

\section{Introduction}

Soient $\kk$ un corps $p$-adique et $\gpalg{\G}$ un groupe réductif connexe défini sur $\kk$. Notons $\G:=\gpalg{\G}[\kk]$. Soit $\lprime$ un nombre premier, $\lprime \neq p$, et posons $\ld=\Ql$ ou $\Zl$. On appelle $\rep[\ld]{\G}$ la catégorie abélienne des représentations lisses de $\G$ à coefficients dans $\ld$ et $\rep[\ld][0]{\G}$ la sous-catégorie pleine des représentations de niveau 0.
\bigskip

Nous nous intéressons à la décomposition de $\rep[\ld]{\G}$ en un produit de sous-catégories. Le cas $\ld=\Ql$ est bien connu puisque le théorème de décomposition de Bernstein fournit une décomposition de $\rep[\Ql]{\G}$ en blocs (c'est à dire en sous-catégories indécomposables). Le cas réellement intéressant est donc $\Zl$. Les travaux de Vignéras \cite{vigneras} et Helm \cite{helm} (voir également les travaux de Sécherre et Stevens \cite{SecherreStevens}) ont permis d'obtenir une décomposition en blocs de $\rep[\Zl]{\gl{n}(\kk)}$. L'inconvénient de cette méthode est qu'elle s'appuie sur "l'unicité du support supercuspidal" qui n'est pas vraie en général (un contre exemple a été trouvé par Dudas dans $\sp{8}$ sur un corps fini puis relevé en $p$-adique par Dat dans \cite{datParabolic}). Pour pallier ce problème, Dat propose dans \cite{dat_equivalences_2014} une nouvelle méthode, basée sur la théorie de Deligne-Lusztig et des systèmes d'idempotents sur l'immeuble de Bruhat-Tits semi-simple (comme dans \cite{meyer_resolutions_2010}), permettant de reconstruire les blocs de $\gl{n}(\kk)$ dans le cas du niveau 0. Nous avons généralisé cette méthode dans \cite{lanard} au cas d'un groupe réductif connexe défini sur $\kk$ et déployé sur une extension non-ramifiée de $\kk$. Néanmoins cette décomposition n'est en général pas la décomposition en blocs, ni même la décomposition la plus fine que puisse donner la méthode. Dans cet article, nous explicitons cette décomposition la plus fine, et nous en donnons deux interprétations, une première du côté du groupe à la Deligne-Lusztig grâce à des paires $(\gpalg{S},\theta)$ (que l'on définit ci-dessous) puis une deuxième du côté dual à la Langlands.

\subsection{Décompositions grâce aux paires \texorpdfstring{$(\gpalg{S},\theta)$}{(S,theta)}}

\label{secintrost}

Supposons que $\gpalg{G}$ se déploie sur $\knr$ l'extension non-ramifiée maximale de $\kk$. Définissons $\pairesSt=\{(\gpalg{S},\theta)\}$ comme l'ensemble des paires $(\gpalg{S},\theta)$, où $\mathbf{S}$ est un $\kk$-tore maximal non-ramifié de $\gpalg{\G}$ et $\theta : {}^{0}\mathbf{S}(\kk)/\mathbf{S}(\kk)^{+} \to \Ql^{*}$ est un caractère d'ordre inversible dans $\ld$ (${}^{0}\mathbf{S}(\kk)$ est le sous-groupe borné maximal de $\gpalg{S}(\kk)$ et $\mathbf{S}(\kk)^{+}$ est son pro-$p$ radical).

Nous définissons sur $\pairesSt$ trois relations d'équivalence $\sim_\infty$, $\sim_r$ et $\sim_e$ de la façon suivante. On a $(\gpalg{S},\theta) \sim_\infty (\gpalg{S}',\theta')$ s'il existe $g\in \gpalg{G}[\knr]$ et $m \in \mathbb{N}^{*}$ tels que ${}^{g} \mathbf{S}(\kk_m)=\mathbf{S}'(\kk_{m})$ et $g( \theta \circ N_{\kk_m/\kk}) = \theta'\circ N_{\kk_m/\kk}$ (où $\kk_m$ est l'extension non-ramifiée de degré $m$ de $\kk$ et $N_{\kk_m/\kk}$ est la norme). On définit $\sim_r$ (resp. $\sim_e$) en rajoutant de plus la condition $g^{-1}\fr(g) \in \NSte$ (resp. et $g^{-1}\fr(g) \in \NSta$) (où $\NSta \subseteq \NSte$ sont certains sous-groupes du normalisateur de $\gpalg{S}$ définis en \ref{secreleqpairesst}). Notons que $\sim_e \Rightarrow \sim_r \Rightarrow \sim_\infty$.

Soit $(\gpalg{S},\theta) \in \pairesSt$, notons $\rep[\ld][(\gpalg{S},\theta)]{G}$ la sous-catégorie pleine de $\rep[\ld][0]{G}$ des objets dont chacun des sous-quotients irréductibles $\pi$ vérifie qu'il existe $(\gpalg{S}',\theta') \sim_e (\gpalg{S},\theta)$ et $\sigma$ une facette $\fr$-stable de l'appartement de $\gpalg{S}'$ tels que  $\langle {}^{*}\mathcal{R}_{\gpfinialg{S}'}^{\quotredalg{G}{\sigma}}(\pi^{\radpara{G}{\sigma}}),\theta'\rangle_{\gpfini{S}'_{\ld}} \neq 0$, où $\radpara{G}{\sigma}$ est le pro-$p$ radical du sous-groupe parahorique de $G$ en $\sigma$, $ {}^{*}\mathcal{R}$ est la restriction de Deligne-Lusztig et $\gpfini{S}'_{\ld}$ est le sous-groupe maximal de $\gpfini{S}'$ d'ordre inversible dans $\ld$. Notons que $\rep[\ld][(\gpalg{S},\theta)]{G}$ ne dépend que de la $\sim_e$-classe d'équivalence de $(\gpalg{S},\theta)$.

On obtient alors le théorème suivant, démontré dans la section \ref{secPairesStheta},

\begin{The}
Les catégories $\rep[\ld][(\gpalg{S},\theta)]{G}$ sont non nulles et fournissent une décomposition
\[\rep[\ld][0]{G} = \prod_{[\gpalg{S},\theta]_e \in \pairesSt/{\sim_e}} \rep[\ld][(\gpalg{S},\theta)]{G}.\]
\end{The}

\begin{Rem}
\begin{enumerate}
\item $\rep[\ld][(\gpalg{S},\theta)]{G}$ est "minimale pour la méthode utilisée", c'est à dire en utilisant Deligne-Lusztig et des systèmes d'idempotents. Cependant, ce n'est pas un bloc en général.
\item Lorsque $G=\gl{n}[\kk]$ nous avons $\sim_e=\sim_r=\sim_\infty$ et $\rep[\ld][(\gpalg{S},\theta)]{G}$ est un bloc.
\end{enumerate}
\end{Rem}

Notons que si $\sim$ est une relation d'équivalence sur $\pairesSt$ plus faible que $\sim_e$ (en particulier $\sim_r$ et $\sim_\infty$) alors en regroupant les catégories $\rep[\ld][(\gpalg{S},\theta)]{G}$ selon la $\sim$-équivalence, nous obtenons une décomposition de $\rep[\ld][0]{G}$. Comme dans la théorie de Deligne-Lusztig, nous pouvons associer à une paire $(\gpalg{S},\theta)$ une classe de conjugaison semi-simple dans le dual sur $\res$, le corps résiduel de $\kk$, et nous montrons au paragraphe \ref{secpairestconjss} qu'on obtient ainsi un diagramme commutatif
\[ \xymatrix{
\pairesSt/{\sim_r} \ar@{^{(}->}[r] \ar@{->}[d] & \ss*{\gpfini*{\G}} \ar@{->}[d]  \\
\pairesSt/{\sim_\infty} \ar@{^{(}->}[r]  & (\ss*{\gpfinialg*{\G}})^{\fr}  }\]
où $\gpfinialg*{\G}$ désigne le dual de $\gpalg{G}$ sur $\res$, le corps résiduel de $\kk$, $\gpfini*{\G} = \gpfinialg*{\G}[\res]$ et la notation $\ss*{(\cdot)}$ désigne l'ensemble des classes de conjugaison semi-simples d'ordre inversible dans $\ld$. Notons que lorsque le groupe $\gpalg{G}$ est quasi-déployé, les injections horizontales sont des bijections. Ce diagramme nous fournit en particulier des décompositions de $\rep[\ld][0]{G}$ indexées par $(\ss*{\gpfinialg*{\G}})^{\fr}$ et $\ss*{\gpfini*{\G}}$ dont nous énoncerons les propriétés dans la section suivante. Pour étudier la relation de $\sim_e$-équivalence, nous ne pouvons pas rester du côté des groupes finis et nous aurons besoin de rajouter des données cohomologiques provenant du groupe $p$-adique.

\subsection{Interprétation duale de la décomposition associée à \texorpdfstring{$\sim_\infty$}{infini}}

\label{secintrosiminf}

Rappelons la définition de $\Lpm{\Ild}$ l'ensemble des paramètres inertiels modérés de \cite{lanard}. Notons $\Ldual{\G}:=\gpalg*{\G}[\Ql]  \rtimes \langle\widehat{\vartheta}\rangle$ le groupe de Langlands dual de $\gpalg{G}$, où $\widehat{\vartheta}$ est un automorphisme de $\gpalg*{\G}$ (le dual de $\gpalg{\G}$ sur $\Ql$) induit par l'action d'un Frobenius inverse. Posons $\Ild$ le sous-groupe fermé maximal de l'inertie de pro-ordre inversible dans $\ld$. Alors $\Lpm{\Ild}$ est l'ensemble des classes de $\gpalg*{\G}$-conjugaison de morphismes continus $\Ild \rightarrow \Ldual{\G}$ triviaux sur l'inertie sauvage qui admettent une extension à un $L$-morphisme admissible $\varphi : \weil' \rightarrow \Ldual{\G}$. 

Construisons à partir de $(\gpalg{S},\theta) \in \pairesSt$ un paramètre inertiel $\phi \in \Lpm{\Ild}$. Notons $\tilde{\theta}$ un caractère de $\gpalg{S}(\kk)$ qui relève $\theta$ et $\varphi : \weil \to \Ldual{S}$ le paramètre de Langlands associé à $\tilde{\theta}$ via la correspondance de Langlands locale pour les tores. Choisissons $\iota$ un plongement (non-canonique) $\iota : \Ldual{S} \hookrightarrow \Ldual{G}$ qui nous permet d'obtenir un paramètre de Langlands pour $\gpalg{G}$ : $\iota \circ \varphi : \weil \to \Ldual{G}$. On pose alors $\phi:=(\iota \circ \varphi)_{|\Ild}$ ce qui nous donne une application $\pairesSt \to \Lpm{\Ild}$. Nous démontrons dans la partie \ref{secSimr} que cette application induit une injection
\[ \pairesSt/{\sim_\infty} \hookrightarrow \Lpm{\Ild},\]
d'image les paramètres "relevants", donnant à son tour une décomposition
\[\rep[\ld][0]{\G} = \prod_{\phi \in \Lpm{\Ild}} \rep[\ld][\phi]{\G}\]
qui est celle de \cite{lanard} Théorème 3.4.5.

\subsection{Interprétation duale de la décomposition associée à \texorpdfstring{$\sim_r$}{~r}}

Introduisons l'ensemble $\Lpbm{\Ild}$ des paramètres inertiels "augmentés". Soit $\phi \in \Lpm{\Ild}$. Notons $C_{\gpalg*{\G}}(\phi)$ le centralisateur de $\phi(\Ild)$ dans $\gpalg*{\G}$ qui est un groupe réductif, possiblement non connexe. Par définition $\phi$ peut s'étendre en un paramètre $\varphi : \weil \longrightarrow \Ldual{\G}$ qui permet de définir $\widetilde{C}_{\gpalg*{\G}}(\phi):=C_{\gpalg*{\G}}(\phi) \varphi(\weil)$, un sous-groupe de $\Ldual{\G}$ indépendant du choix de $\varphi$. Notons maintenant $\pi_{0}(\phi):=\pi_{0}(C_{\gpalg*{\G}}(\phi))$ et $\widetilde{\pi}_{0}(\phi):=\widetilde{C}_{\gpalg*{\G}}(\phi)/C_{\gpalg*{\G}}(\phi)^{\circ}$. On obtient alors une suite exacte scindée
\[ 0 \longrightarrow \pi_{0}(\phi) \longrightarrow \widetilde{\pi}_{0}(\phi) \longrightarrow  \langle \widehat{\vartheta} \rangle \longrightarrow 0.\]
Appelons $\Lpbm{\Ild}$ l'ensemble des couples $(\phi,\sigma)$ à $\gpalg*{\G}$-conjugaison près, où $\phi : \Ild \to \Ldual{\G}$ est un paramètre inertiel modéré et $\sigma : \langle \widehat{\vartheta} \rangle \to \widetilde{\pi}_{0}(\phi)$ est une section à la suite exacte précédente.

Avec les mêmes notations que dans la section \ref{secintrosiminf}, on construit une application $\pairesSt \to \Lpbm{\Ild}$ en envoyant  $(\gpalg{S},\theta)$ sur $(\phi,\sigma)$ où $\phi:=(\iota \circ \varphi)_{|\Ild}$ et $\sigma(\widehat{\vartheta}):=(\iota \circ \varphi)(\text{Frob})$. Nous démontrons alors dans la partie \ref{secSimr} le théorème suivant.

\begin{The}
\label{thedecompps}
L'application $\pairesSt \to \Lpbm{\Ild}$ induit une injection
\[ \pairesSt/{\sim_r} \hookrightarrow \Lpbm{\Ild}\]
dont l'image est le sous-ensemble des paramètres "relevants" (en particulier elle est bijective lorsque $\gpalg{G}$ est quasi-déployé). La décomposition obtenue
\[ \rep[\ld][0]{\G}  =\prod_{(\phi,\sigma) \in \Lpbm{\Ild}} \rep[\ld][(\phi,\sigma)]{\G}\]
possède les propriétés suivantes
\begin{enumerate}
\item Elle est compatible à l'induction et la restriction parabolique (voir le théorème \ref{theInductionParaboliquephisigma} pour un énoncé plus précis).
\item Compatibilité entre les décompositions sur $\Zl$ et $\Ql$ :
\[\rep[\Zl][(\phi,\sigma)]{G} \cap \rep[\Ql]{G} = \prod_{(\phi',\sigma')} \rep[\Ql][(\phi',\sigma')]{G}\]
où le produit est pris sur les $(\phi',\sigma') \in \Lpbm{\iner}$ s'envoyant sur $(\phi,\sigma)$ par l'application naturelle $\Lpbm{\iner} \to \Lpbm{\inerl}$ (obtenue en restreignant à $\inerl$, voir section \ref{seclienzlqlr} pour plus de détails).
\end{enumerate}

Par ailleurs, soit $\mathcal{B}^{st}_{m,\ld}$ l'ensemble des paramètres modérés $\varphi : \weil \rightarrow \Ldual{\G}$ à équivalence inertielle près comme dans \cite{haines} définition 5.3.3 (voir aussi section \ref{secEqinert}). Alors l'application qui à $\varphi$ associe $\phi=\varphi_{|\Ild}$ et $\sigma(\widehat{\vartheta})=\varphi(\text{Frob})$ induit une bijection $\mathcal{B}^{st}_{m,\ld} \tosim \Lpbm{\Ild}$ (voir section \ref{secEqinert}), et nous avons alors une nouvelle interprétation de la décomposition précédente 
\[\rep[\ld][0]{\G}  =\prod_{(\phi,\sigma) \in \Lpbm{\Ild}} \rep[\ld][(\phi,\sigma)]{\G}=\prod_{ [\varphi] \in \mathcal{B}_{m,\ld}^{st}} \rep[\ld][[\varphi]]{\G}.\]
\end{The}

Cette décomposition est indexée par des paramètres à la Langlands mais est construite indépendamment de la correspondance de Langlands locale. Il est alors souhaitable de vérifier la compatibilité à cette dernière. Lorsque $\G$ est un groupe classique non-ramifié, c'est à dire $\gl{n}$, $\sp{2n}$, $\so{2n+1}$, $\so{2n}$, $\so*{2n}$ (groupe spécial orthogonal quasi-déployé associé à une extension quadratique non-ramifiée $\kk'/\kk$) ou $\un{n}[\kk'][\kk]$, et $\ld=\Ql$ alors on a la correspondance de Langlands locale (\cite{HarrisTaylor} \cite{henniart} \cite{arthur} \cite{mok} \cite{KMSW}). Nous montrons alors dans la partie \ref{secSimr} le théorème suivant.

\begin{The}
\label{thelanglnadsps}
Soit $\G$ un groupe classique non-ramifié, $\ld=\Ql$ et $p \neq 2$. 

\begin{enumerate}
\item Alors la décomposition $\rep[\Ql][0]{\G}=\prod_{[\varphi] \in \mathcal{B}^{st}_{\Ql}} \rep[\Ql][[\varphi]]{\G}$ est compatible à la correspondance de Langlands locale. C'est-à-dire que si $\pi$ est une représentation irréductible, alors $\pi \in \rep[\Ql][[\varphi]]{\G}$ si et seulement si $\varphi_{\pi|\weil} \in [\varphi]$, où $\varphi_{\pi}$ est le paramètre de Langlands associé à $\pi$.

\item Cette décomposition est la décomposition de $\rep[\Ql][0]{\G}$ en "blocs stables". C'est à dire que ces facteurs correspondent à des idempotents primitifs du centre de Bernstein stable.

\end{enumerate}
\end{The}

Remarquons que le lien avec les classes de conjugaisons semi-simples données dans la partie \ref{secintrost} est donné par un diagramme commutatif
\[ \xymatrix{
\ss*{\gpfini*{\G}} \ar@{<->}[r]^-{\sim} \ar@{->}[d] &  \Lpbm{\Ild} \ar@{->}[d]\\
 (\ss*{\gpfinialg*{\G}})^{\fr} \ar@{<->}[r]^-{\sim}  &  \Lpm{\Ild} }.\]

\subsection{Interprétation duale de la décomposition associée à \texorpdfstring{$\sim_e$}{~e}}

Pour paramétrer $\pairesSt/{\sim_e}$ nous avons besoin de rajouter des donnés cohomologiques. Nous ne le faisons que pour les formes intérieures pures des groupes non-ramifiés. Supposons donc que $\gpalg{G}$ est non-ramifié (c'est à dire on rajoute l'hypothèse quasi-déployé). Soit $\omega \in H^1(\kk,\gpalg{G})$, à qui correspond $\gpalg{G}_{\omega}$, une forme intérieure pure de $\gpalg{G}$. L'isomorphisme de Kottwitz identifie $\omega$ à un élément de $\Irr[\pi_{0}(Z(\gpalg*{\G})^{\widehat{\vartheta}})]$. À partir de $(\phi,\sigma) \in \Lpbm{\Ild}$, nous définissons une application
\[\hps : \Irr[\pi_{0}(Z(C_{\gpalg*{\G}}(\phi)^{\circ})^{\sigma(\widehat{\vartheta})})]/{\pi_{0}(\phi)^{\sigma(\widehat{\vartheta})}} \to \Irr[\pi_{0}(Z(\gpalg*{\G})^{\widehat{\vartheta}})].\]
La partie \ref{secsime} nous donne alors le théorème suivant.

\begin{The}
\label{thedecompsc}
Nous avons une bijection (qui dépend du choix d'un sommet hyperspécial dans l'immeuble de Bruhat-Tits de $G$)
\[\mathcal{C}^{r}(\phi,\sigma,\omega)/{\sim_{e}} \overset{\sim}{\longrightarrow} \hps^{-1}(\omega),\]
où $\mathcal{C}^{r}(\phi,\sigma,\omega)$ est la classe de $\sim_r$-équivalence image réciproque de $(\phi,\sigma)$ par l'injection $\pairesSt(G_{\omega})/{\sim_r} \hookrightarrow \Lpbm{\Ild}$, pour $G_{\omega}$ la forme intérieure pure de $G$ associée à $\omega$. Celle-ci nous fournit une décomposition
\[\rep[\ld][(\phi,\sigma)]{G_{\omega}} = \prod_{\alpha \in \hps^{-1}(\omega)} \rep[\ld][(\phi,\sigma,\alpha)]{G_{\omega}}.\]
De plus
\begin{enumerate}
\item Cette décomposition est compatible à l'induction et la restriction parabolique.
\item Nous avons $\rep[\Zl][(\phi,\sigma,\alpha)]{G_{\omega}} \cap \rep[\Ql]{G_{\omega}} = \prod_{(\phi',\sigma',\alpha')} \rep[\Ql][(\phi',\sigma',\alpha')]{G_{\omega}}$, où le produit est pris sur les $(\phi',\sigma') \in \Lpbm{\iner}$ s'envoyant sur $(\phi,\sigma)$ par $\Lpbm{\iner} \to \Lpbm{\inerl}$ et $\alpha' \in \hps^{-1}(\omega)$ s'envoyant sur $\alpha$ par l'application naturelle $\Irr[\pi_{0}(Z(C_{\gpalg*{\G}}(\phi')^{\circ})^{\sigma'(\widehat{\vartheta})})] \to \Irr[\pi_{0}(Z(C_{\gpalg*{\G}}(\phi)^{\circ})^{\sigma(\widehat{\vartheta})})]$.
\end{enumerate}

\end{The}

\begin{Rem}
\begin{enumerate}
\item Cette décomposition est la plus fine que l'on puisse obtenir avec "les méthodes de cet article" (c'est à dire en utilisant Deligne-Lusztig).
\item Elle démontre, dans le cas du niveau zéro, une décomposition prédite par Dat dans \cite{datFunctoriality} 2.1.2 et 2.1.4.
\end{enumerate}
\end{Rem}

Lorsque $\ld=\Ql$, nous nous attendons à une compatibilité de ce théorème à la correspondance de Langlands locale enrichie (à un paramètre de Langlands enrichi $(\varphi,\eta)$ où $\eta \in \Irr(\pi_{0}(C_{\gpalg*{G}}(\varphi)))$, on associe un $\alpha$ en restreignant $\eta$ à $\pi_{0}(Z(C_{\gpalg*{\G}}(\phi)^{\circ})^{\sigma(\widehat{\vartheta})})$). Nous le vérifions partiellement pour $\varphi$ un paramètre de Langlands modéré elliptique en position générale (elliptique signifie que l'image de $\varphi$ n'est pas contenue dans un sous-groupe de Levi propre de $\Ldual{\G}$ et en position générale que le centralisateur de $\varphi(\iner)$ dans $\gpalg*{G}$ est un tore) et la correspondance de Langlands locale de DeBacker-Reeder (\cite{DebackerReeder}).

\bigskip

Malgré le fait que les sous-catégories du théorème \ref{thedecompsc} soient minimales pour la méthode utilisée, celles-ci ne sont pas des blocs. En effet la représentation supercuspidale $\Theta_{10}$ de $\sp{4}[\mathbb{Q}_{p}]$, induite depuis l'inflation à $\sp{4}[\mathbb{Z}_{p}]$ de la représentation cuspidale unipotente $\theta_{10}$ de $\sp{4}[\mathbb{F}_{p}]$, est dans $\rep[\ld][1]{\G}$, le facteur direct associé au paramètre inertiel trivial, qui ne se décompose pas davantage par le théorème \ref{thedecompsc}.

\subsection{Plan de l'article}

L'exposition des résultats principaux de cet article faite dans l'introduction, ne reflète pas la méthode de démonstration utilisée. Notamment, nous commencerons avec une définition très différente de $\rep[\ld][(\gpalg{S},\theta)]{G}$. La partie \ref{secSystCoh} traite des systèmes cohérents d'idempotents et des systèmes de classes de conjugaison 0-cohérents (comme introduit dans \cite{lanard}). Ce sont ces systèmes qui permettent de construire des sous-catégories de Serre de $\rep[\ld][0]{G}$. Nous nous intéressons plus particulièrement à la notion de système minimal (les plus petits systèmes cohérents que l'on puisse construire avec Deligne-Lusztig). Dans la partie \ref{secPairesStheta}, nous construisons un procédé qui à une paire $(\gpalg{S},\theta)$ associe un système d'idempotents cohérent minimal et donc une sous-catégorie $\rep[\ld][(\gpalg{S},\theta)]{G}$. Nous caractérisons également la relation d'équivalence sur les paires $(\gpalg{S},\theta)$ induite par ce procédé. Ceci nous permet d'obtenir les résultats énoncés en \ref{secintrost}. Nous réinterprétons ces résultats en des termes duaux à la Langlands dans les parties \ref{secSimr} et \ref{secsime}. La partie \ref{secSimr} a pour but l'obtention des théorèmes \ref{thedecompps} et \ref{thelanglnadsps}. La partie \ref{secsime}, quand à elle, se concentre sur le théorème \ref{thedecompsc}.

\bigskip

\paragraph{\textbf{Remerciements}}
Je tiens à remercier Jean-François Dat pour son aide précieuse concernant la rédaction de cet article. Je remercie également le rapporteur qui a permis de simplifier les preuves de l'annexe \ref{secCohomologie}.

\tableofcontents

\section*{Notations}

  Soit $\kk$ un corps $p$-adique et $\res$ son corps résiduel. Notons $q=|\res|$. On fixe une clôture algébrique $\kalg$ de $\kk$ et on note $\knr$ l'extension non-ramifiée maximale de $\kk$ dans $\kalg$. On appellera $\mathfrak{o}_{\kk}$ (resp. $\mathfrak{o}_{\knr}$) l'anneau des entiers de $\kk$ (resp. $\knr$). Notons également $\resalg$ le corps résiduel de $\knr$ qui est alors une clôture algébrique de $\res$.
  
 \medskip

  On adopte les conventions d'écriture suivantes. $\gpalg{\G}$ désignera un groupe réductif connexe défini sur $\kk$ que l'on identifiera avec $\gpalg{\G}[\kalg]$ et l'on notera $\G:=\gpalg{\G}[\kk]$ et $G^{nr}:=\gpalg{\G}[\knr]$. On appellera $\gpalg*{\G}$ son groupe dual sur $\Ql$. Pour les groupes réductifs connexes sur $\res$, nous utiliserons la police d'écriture $\gpfinialg{\G}$ et l'on identifiera $\gpfinialg{\G}$ avec $\gpfinialg{\G}[\resalg]$ et de même on note $\gpfini{\G}:=\gpfinialg{\G}[\res]$. Le groupe dual de $\gpalg{\G}$ sur $\resalg$ sera noté $\gpfinialg*{\G}$.
  
\medskip

On prend $\lprime$ un nombre premier différent de $p$, et on pose $\ld= \Ql$ ou $\Zl$. On note $\rep[\ld]{\G}$ la catégorie abélienne des représentations lisses de $\G$ à coefficients dans $\ld$ et $\rep[\ld][0]{\G}$ la sous-catégorie pleine des représentations de niveau 0.

\medskip

 On note $\bt[L]$ l'immeuble de Bruhat-Tits semi-simple associé à $\G$ sur $L$ une extension non-ramifiée de $\kk$. Nous voyons ce dernier comme un complexe polysimplicial et on note $\bts[L]$ pour l'ensemble des polysimplexes de dimension 0, c'est à dire les sommets.  Nous désignerons généralement les sommets par des lettres latines $x$,$y$,$\cdots$ et les polysimplexes par des lettres grecques $\sigma$,$\tau$, $\cdots$. On munit $\bt$ de la relation d'ordre $\sigma \leq \tau$ si $\sigma$ est une facette de $\tau$. Deux sommets $x$ et $y$ sont dit adjacent s'il existe un polysimplexe $\sigma$ tel que $x \leq \sigma$ et $y \leq \sigma$. Dans ce cas, on notera $[x,y]$ le plus petit polysimplexe contenant $x\cup y$. Notons également, pour $\sigma$,$\tau$ deux polysimplexes, $H(\sigma,\tau)$ l'enveloppe polysimpliciale de $\sigma$ et $\tau$, c'est à dire l'intersection de tous les appartements contenant $\sigma \cup \tau$.
 
 Soit $\sigma \in \bt$. On note $\para{\G}{\sigma}$ le sous-groupe parahorique en $\sigma$ ("fixateur connexe" de $\sigma$) et $\radpara{\G}{\sigma}$ son pro-$p$-radical. Le quotient, $\quotred{\G}{\sigma}$, est le groupe des $\res$-points d'un groupe réductif connexe $\quotredalg{\G}{\sigma}$ défini sur $\res$. 

\medskip

  Si $H$ désigne l'ensemble des points d'un groupe algébrique à valeur dans un corps alors $\ss{H}$ désigne l'ensemble des classes de conjugaison semi-simples dans $H$ et $\ss*{H}$ le sous-ensemble de $\ss{H}$ des éléments d'ordre inversible dans $\ld$.

\medskip

  On fixe dans tout ce papier des systèmes compatibles de racines de l'unité $(\mathbb{Q}/\mathbb{Z})_{p'} \overset{\backsim}{\rightarrow} \resalg^{\times}$ et $(\mathbb{Q}/\mathbb{Z})_{p'} \hookrightarrow \Zl^{\times}$.

\medskip

  Dans la suite $\gpalg{\G}$ désignera un groupe réductif connexe défini sur $\kk$. On supposera de plus, à partir de la section \ref{secPairesStheta}, que $\gpalg{\G}$ se déploie sur $\knr$.

\section{Systèmes cohérents d'idempotents}

\label{secSystCoh}
  
  Toutes les constructions de sous-catégories de $\rep[\ld][0]{\G}$ dans cet article se basent sur des systèmes de classes de conjugaison $0$-cohérents, comme introduit dans \cite{lanard}. Dans cette partie nous faisons quelques rappels sur ces derniers et nous nous intéressons particulièrement aux systèmes "minimaux". Ceux-ci sont les plus petits systèmes constructibles et ils permettent d'obtenir, en les réunissant, tous les systèmes de classes de conjugaison $0$-cohérents.

  \subsection{Rappels sur les systèmes cohérents}
  
  \label{secSystcoh}

  Commençons pas rappeler la définition des systèmes de classes de conjugaison $0$-cohérents ainsi que leur intérêt.

\sautintro

    On fixe une mesure de Haar sur $\G$ et on note $\hecke{\G}{\ld}$ l'algèbre de Hecke à coefficients dans $\ld$ (on rappelle que $\ld=\Ql$ ou $\Zl$).

\begin{Def}
    \label{defcoherent}
    On dit qu'un système d'idempotents $e=(e_{x})_{x\in \bts}$ de $\hecke{\G}{\ld}$ est cohérent si les propriétés suivantes sont satisfaites :
    \begin{enumerate}
    \item $e_{x}e_{y}=e_{y}e_{x}$ lorsque $x$ et $y$ sont adjacents.
    \item $e_{x}e_{z}e_{y}=e_{x}e_{y}$ lorsque $z$ est adjacent à $x$ et appartient à l'enveloppe polysimpliciale de $x$ et $y$.
    \item $e_{gx}=ge_{x}g^{-1}$ quel que soit $x\in \bts$ et $g\in \G$.
    \end{enumerate}
\end{Def}

Si $e=(e_{x})_{x\in \bts}$ est un système d'idempotents cohérent, on peut alors pour tout $\sigma \in \bt$ définir l'idempotent $e_{\sigma}:=\prod_{x} e_x$, où le produit est pris sur les sommets $x$ tels que $x \leq \sigma$.

\begin{The}[\cite{meyer_resolutions_2010}, Thm 3.1]
    \label{themoyersolleveld}
    Soit $e=(e_{x})_{x\in \bts}$ un système cohérent d'idempotents, alors la sous-catégorie pleine $\rep[\ld][e]{\G}$ des objets $V$ de $\rep[\ld]{\G}$ tels que $V=\sum_{x\in \bts}e_{x}V$ est une sous-catégorie de Serre.
\end{The}

    Voyons comment construire des idempotents sur l'immeuble. Pour une classe de conjugaison semi-simple $s \in \ss{\quotred*{\G}{\sigma}}$, on désigne par $\dl{\quotred*{\G}{\sigma}}{s}$ la série de Deligne-Lusztig rationnelle associée à $s$, et $e_{s,\Ql}^{\quotred{\G}{\sigma}} \in \Ql[\quotred{\G}{\sigma}]$ l'idempotent central la sélectionnant. Lorsque $s$ se compose d'éléments $\lprime$-réguliers, on peut former $e_{s,\Zl}^{\quotred{\G}{\sigma}}$, un idempotent dans $\Zl[\quotred{\G}{\sigma}]$, défini par $e_{s,\Zl}^{\quotred{\G}{\sigma}}:=\sum_{s' \backsim_{\lprime} s} e_{s',\Ql}^{\quotred{\G}{\sigma}}$, où $s' \backsim_{\lprime} s$ signifie que $s$ est la partie $\lprime$-régulière de $s'$ (\cite{bonnafe_rouquier} théorème A' et remarque 11.3). Ainsi pour $s \in \ss*{\quotred*{\G}{\sigma}}$ (c'est à dire d'ordre inversible dans $\ld$) on a un idempotent $e_{s,\ld}^{\quotred{\G}{\sigma}} \in \ld[\quotred{\G}{\sigma}]$ que l'on peut tirer en arrière, grâce à l'isomorphisme $\para{\G}{\sigma}/\radpara{\G}{\sigma} \to \quotred{\G}{\sigma}$, en un idempotent $e_{\sigma}^{s,\ld} \in \ld[ \para{\G}{\sigma}/\radpara{\G}{\sigma} ] \subset \hecke{\fix{\G}{\sigma}}{\ld}$.

    \bigskip
    
	Nous pouvons de la sorte construire des idempotents à partir de classes de conjugaison semi-simples. On appellera donc un système de classes de conjugaison, une famille $\systConj*{S}$ avec  $\systConjx{S}{\sigma} \subseteq \ss*{\quotred*{\G}{\sigma}}$ . Et on notera
	\[ \ensSystConj := \{  \text{systèmes de classes de conjugaison } \systConj*{S} \}. \]

    Pour traduire les propriétés de cohérence au niveau des classes de conjugaison, nous avons besoin de définir deux applications $\transss*{\sigma}{\tau}$ et $\transss*{g}{\sigma}$.

    La première est définie pour $\tau,\sigma \in \bt$ deux facettes telles que $\tau \leq \sigma$. Alors, le groupe $\quotred*{\G}{\sigma}$ peut se relever en un Levi de $\quotred*{\G}{\tau}$ et on obtient $\transss*{\sigma}{\tau}: \ss{\quotred*{\G}{\sigma}} \rightarrow \ss{\quotred*{\G}{\tau}}$, qui à une classe de conjugaison semi-simple $t$ de $\quotred*{\G}{\sigma}$, associe sa classe de conjugaison dans $\quotred*{\G}{\tau}$. La deuxième est définie pour $g \in \G$ et $\sigma \in \bt$. La conjugaison par $g$ induit un isomorphisme $\transss{g}{\sigma}: \quotred{\G}{\sigma} \overset{\sim}{\rightarrow} \quotred{\G}{g\sigma}$ ainsi qu'un isomorphisme sur les classes de conjugaison semi-simples des groupes duaux $\transss*{g}{\sigma}:\ss{\quotred*{\G}{\sigma}} \overset{\sim}{\longrightarrow} \ss{\quotred*{\G}{g\sigma}}$.

    \begin{Def}
    \label{defsystemsurcoherentconju}
    Soit $\systConj{S} \in \ensSystConj$ un système d'ensembles de classes de conjugaison. On dit que $\systConj{S}$ est 0-cohérent si
    \begin{enumerate}
    \item $\transss*{g}{x}(\systConjx{S}{x})=\systConjx{S}{gx}$ quel que soit $x\in \bts$ et $g\in \G$.
    \item $(\transss*{\sigma}{x})^{-1}(\systConjx{S}{x})=\systConjx{S}{\sigma}$ pour $x\in \bts$ et $\sigma \in \bt$ tels que $x \leq \sigma$.
    \end{enumerate}
    On note
    \[ \ensSystConjc :=\{ \systConj{S} \in \ensSystConj, \text{ 0-cohérent}\}\]
    \end{Def}

    Soient $\systConj{S} \in \ensSystConjc$ un système 0-cohérent et $\sigma \in \bt$. On définit $e_{\sigma}^{\systConj{S},\ld}:=\sum_{s \in \systConjx{S}{\sigma}} e_{\sigma}^{s,\ld}$.

    \begin{Pro}[\cite{lanard} 2.3.2]
    \label{prosystemeidempotent}
    Le système $(e_{x}^{\systConj{S},\ld})_{x \in \bts}$ est cohérent.
    \end{Pro}

    Le théorème \ref{themoyersolleveld} nous montre donc que $\systConj{S}$ définit une sous-catégorie $\rep[\ld][\systConj{S}]{\G}$ de $\rep[\ld]{\G}$.

On définit les opérations suivantes sur les systèmes de classes de conjugaison.

    \begin{Def}
    Soient $\systConj{S^1}, \systConj{S^2} \in \ensSystConj$ deux systèmes de classes de conjugaison. On définit $\systConj{S^1} \cup \systConj{S^2}:=(\systConjx{S^1}{\sigma} \cup \systConjx{S^2}{\sigma})_{\sigma \in \bt}$ et $\systConj{S^1} \cap \systConj{S^2}:=(\systConjx{S^1}{\sigma} \cap \systConjx{S^2}{\sigma})_{\sigma \in \bt}$. On dit que $\systConj{S^2} \subseteq \systConj{S^1}$ si pour tout $\sigma \in \bt$ $\systConjx{S^2}{\sigma} \subseteq \systConjx{S^1}{\sigma}$. Enfin, si $\systConj{S^2} \subseteq \systConj{S^1}$, on note $\systConj{S^1} \setminus \systConj{S^2}:=(\systConjx{S^1}{\sigma} \setminus \systConjx{S^2}{\sigma})_{\sigma \in \bt}$.
    \end{Def}

    \begin{Pro}[\cite{lanard} 2.3.5]
    \label{prodecompocategorie}
    Soient $\systConj{S^1},\cdots,\systConj{S^n} \in \ensSystConjc$ des systèmes 0-cohérents tels que $\systConj{S^i} \cap \systConj{S^j} = \emptyset$ si $i \neq j$ et $\bigcup_{i=1}^{n} \systConj{S^i} = (\ss*{\quotred*{\G}{\sigma}})_{\sigma \in \bt}$. Alors la catégorie de niveau $0$ se décompose en
\[ \rep[\ld][0]{\G} = \prod_{i=1}^{n} \rep[\ld][\systConj{S^i}]{\G} \]
    \end{Pro}

  \subsection{Systèmes minimaux}
  
 Dans cette section, nous construisons une application $\ensSystConj \to \ensSystConjc$, $\systConj{S} \mapsto \clotco{S}$, de sorte que $\clotco{S}$ soit le plus petit système 0-cohérent contenant $\systConj{S}$. Cette application nous permettra de construire les systèmes 0-cohérents minimaux en prenant l'image de systèmes minimaux.

    \label{sectionSystMin}
\sautintro

Il est aisé de vérifier que si $\systConj{S^1}$ et $\systConj{S^2}$ sont deux systèmes de classes de conjugaison 0-cohérents, alors $\systConj{S^1} \cup \systConj{S^2}$, $\systConj{S^1} \cap \systConj{S^2}$ et $\systConj{S^1} \setminus \systConj{S^2}$ (si $\systConj{S^2} \subseteq \systConj{S^1}$) sont aussi 0-cohérents.

    \begin{Def}
On définit une application  $\ensSystConj \to \ensSystConjc$, $\systConj{S} \mapsto \clotco{S}$, en posant pour $\systConj{S} \in \ensSystConj$
    \[\clotco{S} := \bigcap_{\underset{\systConj{S} \subseteq \systConj{T}}{\systConj{T} \in \ensSystConjc}} \systConj{T}.\]
    Alors $\clotco{S}$ est le plus petit système 0-cohérent contenant $\systConj{S}$.
    \end{Def}

    Nous allons construire des systèmes 0-cohérents minimaux. Fixons $\sigma \in \bt$ et $s \in \ss*{\quotred*{\G}{\sigma}}$. Définissons le système $\systConj{S}$ de classes de conjugaison par $\systConjx{S}{\sigma}=\{s\}$ et $\systConjx{S}{\tau}=\emptyset$ si $\tau\neq \sigma$. On pose alors $\systEng{\sigma}{s}:=\clotco{S}$.

    \begin{Def}
    \label{defJss}
    On appelle $\Jss$ l'ensemble des couples $(\sigma,s)$ où $\sigma \in \bt$ et $s \in \ss*{\quotred*{\G}{\sigma}}$ et on définit une relation d'équivalence $\simJss$ sur $\Jss$ par $(\sigma,s) \simJss (\tau,t)$ si et seulement si $\systEng{\sigma}{s}=\systEng{\tau}{t}$.
    \end{Def}

On obtient aisément le lemme suivant

    \begin{Lem}
    \label{lemSystEng}
    Soient $\sigma, \tau \in \bt$, et $s \in \ss*{\quotred*{\G}{\sigma}}$, $t \in \ss*{\quotred*{\G}{\tau}}$. Alors, soit $\systEng{\sigma}{s}=\systEng{\tau}{t}$, soit $\systEng{\sigma}{s} \cap \systEng{\tau}{t} = \emptyset$.
    \end{Lem}

    Exprimé en termes de classes d'équivalence, ce lemme nous dit que $[\sigma,s]$ la classe d'équivalence de la paire $(\sigma,s)$ est donnée par $[\sigma,s]=\{(\tau,t)\ |\ t \in \systEngx{\sigma}{s}{\tau}\}$.

    \bigskip

  Posons $\Smin:=\{ \systEng{\sigma}{s}, [\sigma,s] \in \Jss/{\simJss}\}$. L'ensemble $\Smin$ est constitué des ensembles 0-cohérents minimaux et par la proposition \ref{prodecompocategorie} on a

    \begin{Pro}
    \label{proDecompoMinimal}
    La catégorie de niveau $0$ se décompose en

    \[ \rep[\ld][0]{\G} = \prod_{\systConj{S} \in \Smin} \rep[\ld][\systConj{S}]{\G}.\]
    \end{Pro}

    Notons également le fait suivant. Les systèmes minimaux permettent de construire tous les systèmes $0$-cohérents. Si $\systConj{S}$ est un système de classes de conjugaison $0$-cohérent alors il existe une partie $\mathscr{I} \subseteq \Smin$ telle que $\systConj{S}=\cup_{\systConj{T} \in \mathscr{I}} \systConj{T}$.

  \subsection{Clôture transitive}

    \label{sectionResChambre}
    
    Dans cette section, nous souhaitons comprendre davantage la relation d'équivalence $\simJss$. Pour cela nous allons montrer qu'elle peut s'exprimer comme la clôture transitive d'une autre relation que l'on définira.

    \begin{Def}
    On dit qu'un système $(e_{\sigma})_{\sigma \leq C}$ est $C$-0-cohérent si
    \begin{enumerate}
    \item $e_{gx}=ge_{x}g^{-1}$ quel que soit $x\in \bts$ et $g\in \G$ tels que $x \leq C$ et $gx \leq C$.
    \item $e_{\sigma}=e_{\sigma}^{+}e_{x}=e_{x}e_{\sigma}^{+}$ pour $x \in \bts$ et  $\sigma \in \bt$ tels que $x \leq \sigma \leq C$.
    \end{enumerate}
    \end{Def}

    \begin{Pro}
    \label{proCaractC}
    Un système 0-cohérent $(e_{\sigma})_{\sigma\in \bt}$ est complètement caractérisé par le système $C$-0-cohérent $(e_{\sigma})_{\sigma \leq C}$.
    \end{Pro}
    
\begin{proof}
Soit $(e_{\sigma})_{\sigma \leq C}$ un système $C$-0-cohérent. Soit $\sigma \in \bt$, $\sigma$ appartient à une chambre $C'$ et $\G$ agit transitivement sur l'ensemble des chambres donc il existe $g \in \G$ tel que $C'=gC$. Ainsi si l'on note $\sigma_{1}:= g^{-1}\sigma$, alors $\sigma_{1} \leq C$. On définit alors $\tilde{e}_{\sigma}:=ge_{\sigma_{1}}g^{-1}$. On vérifie aisément que $(\tilde{e}_{\sigma})_{\sigma \in \bt}$ ne dépend pas des choix effectués et est $0$-cohérent.
\end{proof}

    De manière analogue on peut restreindre les systèmes de classes de conjugaison 0-cohérents à $C$.

 \begin{Def}
    Soit $\systConjC{S}$ un système d'ensembles de classes de conjugaison avec $\systConjx{S}{\sigma} \subseteq \ss*{\quotred*{\G}{\sigma}}$. On dit que $\systConj{S}$ est $C$-0-cohérent si
    \begin{enumerate}
    \item $\transss*{g}{x}(\systConjx{S}{x})=\systConjx{S}{gx}$ quels que soient $x\in \bts$ et $g\in \G$ tels que $x \leq C$ et $gx \leq C$.
    \item $(\transss*{\sigma}{x})^{-1}(\systConjx{S}{x})=\systConjx{S}{\sigma}$ pour $x \in \bts$ et $\sigma \in \bt$ tels que $x \leq \sigma \leq C$.
    \end{enumerate}
    \end{Def}

    Et comme dans la section \ref{sectionSystMin}, on définit des systèmes $C$-0-cohérents minimaux.

  \begin{Def}
    On définit une relation $\simJsst$ sur $\Jss$ par $(\sigma,s) \simJsst (\tau,t)$ si et seulement si l'une des conditions suivantes est réalisée :
    \begin{enumerate}
    \item Il existe $g \in \G$ tel que $\tau = g\sigma$ et $t = \transss{g}{\sigma}^{*}(s)$.
    \item Il existe $x \in \bts$ tel que $\sigma \geq x \leq \tau$ et $\varphi_{\sigma,x}^{*}(s)=\varphi_{\tau,x}^{*}(t)$.
    \end{enumerate}
    \end{Def}

    La relation $\simJsst$ est réflexive et symétrique sur $\Jss$.

    \begin{Pro}
    	\label{proClotureTransitive}
    La relation d'équivalence $\simJss$ sur $\Jss$ est la clôture transitive de $\simJsst$.
    \end{Pro}

\begin{proof}
Nous avons déjà que si deux paires sont $\simJsst$-équivalentes alors elles sont $\simJss$-équivalentes. Prenons donc $(\sigma,s)$ et $(\tau,t)$ tels que $(\sigma,s) \simJss (\tau,t)$.

Comme $\G$ agit transitivement sur les chambres de $\bt$, on peut supposer que $\sigma,\tau \in C$. Posons $\systConj{S^0}$ le système de classes de conjugaison sur $C$ définit par $\systConj{S^0}_{\sigma}=\{s\}$ et $\systConj{S^0}_{\omega}=\emptyset$ si $\omega \in C$ et $\omega \neq \sigma$. Pour simplifier les notations, on identifiera un système $\systConj{S}$ de classes de conjugaison avec le sous-ensemble de $\Jss$ des paires $(\omega,u)$ telles que $u \in \systConjx{S}{\omega}$ (ainsi on a par exemple $\systConj{S^0} = \{ (\sigma,s) \}$). On construit alors par récurrence les systèmes $\systConj{S^i}$, pour $i\geq 1$, par $\systConj{S^i}:=\systConj{S^{i-1}} \cup \systConj{T^i}$, où $ \systConj{T^i}$ est l'ensemble des couples $(\omega,u) \in \Jss$ avec $\omega \in C$ tels qu'il existe $(\lambda,v) \in \systConj{S^{i-1}}$ vérifiant $(\lambda,v) \simJsst (\omega,u)$. Cette suite de systèmes est croissante et comme ils ne sont définis que sur $C$ (donc il n'y a qu'un nombre fini de systèmes possibles), elle est stationnaire à partir d'un certain rang $n$. Par définition, les éléments de $\systConj{S^n}$ sont des paires qui sont obtenues par une suite d'au plus $n$ éléments de paires qui sont deux à deux $\simJsst$-équivalentes à partir de $(\sigma,s)$. Nous souhaitons donc montrer que $(\tau,t) \in \systConj{S^n}$.

Comme $\systConj{S^n} = \systConj{S^{n+1}}$, on a que $\systConj{T^{n+1}} \subseteq \systConj{S^n}$. Soit $x\in C$ un sommet et $g \in \G$ tel que $gx \in C$. Soit $u \in \systConjx{S^n}{x}$, alors comme $(x,u) \simJsst (gx,\transss{g}{x}^{*}(u))$ on a $(gx,\transss{g}{x}^{*}(u)) \in \systConj{T^{n+1}} \subseteq \systConj{S^n}$. On vient de montrer que $\transss{g}{x}^{*}(\systConjx{S^n}{x}) \subseteq \systConjx{S^n}{gx}$. Le même raisonnement montre que $\transss{g^{-1}}{gx}^{*}(\systConjx{S^n}{gx}) \subseteq \systConjx{S^n}{x}$ et donc que $\transss{g}{x}^{*}(\systConjx{S^n}{x}) = \systConjx{S^n}{gx}$. On montre de façon analogue que $(\transss*{\sigma}{x})^{-1}(\systConjx{S^n}{x})=\systConjx{S^n}{\sigma}$ et donc que $\systConj{S^n}$ est $C$-0-cohérent. Comme $(\sigma,s) \in \systConj{S^n}$, on a que $\systEng{\sigma}{s} \subseteq \systConj{S^n}$ (dans cette dernière inclusion on identifie systèmes 0-cohérents et $C$-0-cohérents grâce à la proposition \ref{proCaractC}).

Montrons maintenant que $\systConj{S^n} \subseteq \systEng{\sigma}{s}$. Faisons pour cela une récurrence. On a tout d'abord $\systConj{S^0} \subseteq \systEng{\sigma}{s}$. Supposons que $\systConj{S^i} \subseteq \systEng{\sigma}{s}$. Pour montrer que $\systConj{S^{i+1}} \subseteq \systEng{\sigma}{s}$, il faut montrer que $\systConj{T^{i+1}} \subseteq \systEng{\sigma}{s}$. Or si $(\omega,u) \in \systConj{T^{i+1}}$, alors par définition $(\omega,u)$ est $\simJsst$-équivalent à un élément de $\systConj{S^{i}}$, donc est $\simJss$-équivalent à un élément de $\systEng{\sigma}{s}$. Le lemme \ref{lemSystEng} montre que $(\omega,u) \in \systEng{\sigma}{s}$, d'où le résultat. Ainsi $\systConj{S^n} \subseteq \systEng{\sigma}{s}$ et donc $\systConj{S^n} = \systEng{\sigma}{s}$.

Pour conclure, comme $(\sigma,s) \simJss (\tau,t)$, par définition $\systEng{\sigma}{s}=\systEng{\tau}{t}$ et donc $(\tau,t) \in \systConj{S^n}$.
\end{proof}

  \subsection{Compatibilité à l'induction et à la restriction parabolique pour les systèmes cohérents}

    \label{secInductionParaboSystem}
    
    Étudions une dernière propriété, qui nous servira par la suite, qui est la compatibilité des systèmes cohérents vis à vis de l'induction et de la restriction parabolique.
    
  \sautintro

    Comme dans la section 4.4 de \cite{lanard}, nous allons utiliser dans cette partie l'immeuble de  Bruhat-Tits "étendu", que l'on notera $\bte[\G]$, qui est plus adapté pour traiter les questions d'induction et de restriction. Cela ne change pas vraiment la définition des idempotents puisque l'on traite la structure polysimpliciale de l'immeuble semi-simple de façon combinatoire.

    \bigskip

    Soit $\mathbf{P}$ un $\kk$-sous-groupe parabolique de $\gpalg{\G}$ de quotient de Levi $\gpalg{M}$ défini sur $\kk$. Prenons $\mathbf{S}$ un tore déployé maximal de $\gpalg{\G}$ contenu dans $\gpalg{P}$ et notons $\mathbf{T}$ son centralisateur dans $\gpalg{\G}$. Il existe alors un unique relèvement de $\gpalg{M}$ en un sous-groupe de $\gpalg{\G}$ contenant $\gpalg{T}$. Soit $x \in \mathcal{A}^{e}$, alors $\para{\gp{M}}{x}=M \cap \para{\G}{x}$ et $\radpara{\gp{M}}{x}=M \cap \radpara{\G}{x}$ (voir \cite{MoyPrasad} section 4.3). Alors $\quotred{P}{x}$, l'image de $P_{x}:=P \cap \para{\G}{x}$ dans $\quotred{\G}{x}$, est un sous-groupe parabolique de $\quotred{\G}{x}$ de quotient de Levi $\quotred{M}{x}$ (voir par exemple \cite{lanard} lemme 4.4.1).

    \bigskip

    Soit $\systConj{S}=(\systConjx{S}{x})_{x \in \btes[\gp{M}]}$ un système 0-cohérent de classes de conjugaison pour $M$ ($\systConjx{S}{x} \subseteq \ss*{\quotred*{M}{x}}$). Le groupe $\quotred*{M}{x}$ est un Levi de $\quotred*{\G}{x}$, on a donc une application naturelle
    \[ \varphi_{M,\G}^{x} : \ss*{\quotred*{M}{x}} \rightarrow \ss*{\quotred*{\G}{x}}\]

    Posons alors $\varphi_{M,\G}(\systConj{S})=(\varphi_{M,\G}^{x}(\systConjx{S}{x}))_{x \in \bte[\G]}$, un système de classes de conjugaison pour $\G$, défini par $\varphi_{M,\G}^{x}(\systConjx{S}{x})=\{\varphi_{M,\G}^{x}(s), s \in \systConjx{S}{x} \}$ si $x \in \bte[\gp{M}]$ et $\varphi_{M,\G}^{x}(\systConjx{S}{x})=\emptyset$ si $x \notin \bte[\gp{M}]$. On pose alors
    \[ i_{M}^{\G}(\systConj{S}):=\overline{\varphi_{M,\G}(\systConj{S})}\]
    la clôture cohérente de $\varphi_{M,\G}(\systConj{S})$.

\medskip

    Soit $\ensSystConjcM=\{\systConj{T^1},\cdots,\systConj{T^m} \}$ un ensemble de systèmes de classes de conjugaison 0-cohérents pour $M$, vérifiant $\systConj{T^i} \cap \systConj{T^j} = \emptyset$ si $i \neq j$ et $\bigcup_{i=1}^{n} \systConj{T^i} = (\ss*{\quotred*{M}{\sigma}})_{\sigma \in \bt}$, de sorte que $\rep[\ld][0]{M} = \prod_{\systConj{T} \in \ensSystConjcM} \rep[\ld][\systConj{T}]{M}$.

    \begin{Pro}
    	\label{proResParaSyst}
    	Soit $\systConj{S}$ un système de classes de conjugaison 0-cohérent pour $\G$. Alors
    	\[ \rp( \rep[\ld][\systConj{S}]{\G} ) \subseteq \prod_{\substack{\systConj{T} \in  \ensSystConjcM \\ i_{M}^{\G}(\systConj{T}) \cap \systConj{S} \neq \emptyset}} \rep[\ld][\systConj{T}]{M}\]
    	où $\rp$ désigne la restriction parabolique.
    \end{Pro}

    \begin{proof}
    Soit $V \in  \rep[\ld][\systConj{S}]{\G}$. La restriction parabolique préserve le niveau donc $\rp[V] \in \rep[\ld][0]{M}$. Soit $T \in \ensSystConjcM$ tel que $i_{M}^{\G}(\systConj{T}) \cap \systConj{S} = \emptyset$.
    Soient $x \in \mathcal{A}^{e}_{M}$ et $t \in \systConjx{T}{x}$. Il nous suffit alors de montrer que $e_{t,\ld}^{\quotred{M}{x}} \rp[V]^{\radpara{\gp{M}}{x}} =0$ pour obtenir le résultat voulu.

    Nous avons $\rp[V]^{\radpara{\gp{M}}{x}} \simeq  r_{\quotred{P}{x}}^{\quotred{\G}{x}} (V^{\radpara{\G}{x}})$ (voir \cite{datFinitude} propositions 3.1 et 6.2), donc $e_{t,\ld}^{\quotred{M}{x}} \rp[V]^{\radpara{\gp{M}}{x}}  \simeq e_{t,\ld}^{\quotred{M}{x}} r_{\quotred{P}{x}}^{\quotred{\G}{x}} (V^{\radpara{\G}{x}}) \simeq e_{t,\ld}^{\quotred{M}{x}} r_{\quotred{P}{x}}^{\quotred{\G}{x}}  ( e_{t,\ld}^{\quotred{\G}{x}}(V^{\radpara{\G}{x}}))$ (pour la dernière égalité voir \cite{dat_equivalences_2014} section 2.1.4 ). Or par définition $t \notin \systConjx{S}{x}$ donc $e_{t,\ld}^{\quotred{\G}{x}}(V^{\radpara{\G}{x}})=0$ d'où le résultat.
    \end{proof}

    \begin{Pro}
    \label{proInducParaSyst}
    Soit $\systConj{T}$ un système de classes de conjugaison 0-cohérent pour $M$. Alors
    \[ \ip[ \rep[\ld][\systConj{T}]{M} ] \subseteq \rep[\ld][i_{M}^{\G}(\systConj{T})]{\G}\]
    où $\ip$ désigne l'induction parabolique.
    \end{Pro}

      \begin{proof}
      Cela découle de la proposition \ref{proResParaSyst} (prendre $\systConj{S}=\compl{(i_{M}^{\G}(\systConj{T}))}$ et $\ensSystConjcM=\{\systConj{T},\compl{\systConj{T}}\}$)  et du fait que $\rp$ est adjoint à gauche de $\ip$.
      \end{proof}

\section{Construction des systèmes \texorpdfstring{$0$}{0}-cohérents grâce aux paires \texorpdfstring{$(\mathbf{S},\theta)$}{(S,theta)}}

\label{secPairesStheta}

Dans cette section on suppose que le groupe $\gpalg{\G}$ se déploie sur $\knr$. Dans la section \ref{secSystCoh} nous avons construit les systèmes 0-cohérents minimaux. Ces derniers sont construits "localement" en regardant des classes de conjugaison semi-simples sur l'immeuble de Bruhat-Tits. Nous souhaitons dans cette partie trouver un procédé "global" permettant de construire ces sytèmes minimaux. Pour ce faire nous allons introduire un ensemble 
\[\pairesSt = \{ (\gpalg{S},\theta) \}\]
où $\mathbf{S}$ est un tore maximal non-ramifié de $\gpalg{\G}$ et $\theta \in X^{*}(\textbf{S})/(\fr_{\mathbf{S}} -1)X^{*}(\textbf{S})$ est un élément d'ordre inversible dans $\ld$, et définir une application
\[ \pairesSt \to \Smin. \]
On montre que cette application est surjective, ainsi tout les systèmes minimaux peuvent être construit à partir des paires $(\gpalg{S},\theta)$. Cependant elle n'est pas injective. On définira donc $\sim_e$, une relation d'équivalence sur $\pairesSt$, telle que l'on ait la bijection
\[ \pairesSt/{\sim_e} \tosim \Smin. \]

\subsection{Définition de \texorpdfstring{$\pairesSt$}{P} et construction de \texorpdfstring{$\pairesSt \to \ensSystConjc$}{P->S}}

\label{secintroStheta}

Nous dirons qu'un tore $\mathbf{S}$ est non-ramifié si $\mathbf{S}$ est un $\kk$-tore $\knr$-déployé maximal de $\gpalg{\G}$, ou de façon équivalente, si $\mathbf{S}$ est un $\knr$-tore déployé maximal de $\gpalg{\G}$ et $\mathbf{S}$ est défini sur $\kk$. Pour un tel $\mathbf{S}$, notons $\vartheta_{\mathbf{S}}$ l'automorphisme de $X^{*}(\textbf{S})$ induit par l'action du Frobenius inverse et posons $\fr_{\mathbf{S}}:=\vartheta_{\mathbf{S}}\circ \psi$, où $\psi$ correspond à la multiplication par $q$.

\begin{Def}
	On définit $\pairesSt$ comme l'ensemble des paires $(\mathbf{S},\theta)$ où $\mathbf{S}$ est un tore maximal non-ramifié de $\gpalg{\G}$ et $\theta \in X^{*}(\textbf{S})/(\fr_{\mathbf{S}} -1)X^{*}(\textbf{S})$ est un élément d'ordre inversible dans $\ld$.
\end{Def}

\begin{Rem}
Notons que par l'annexe \ref{secIsoTores}, $X^{*}(\textbf{S})/(\fr_{\mathbf{S}} -1)X^{*}(\textbf{S})$ est en bijection avec $\Irr({}^{0}(\mathbf{S}^{\fr})/(\mathbf{S}^{\fr})^{+})$ (on utilise ici le système compatible de racines de l'unité fixé au début de cet article). Nous retrouvons ainsi la définition de l'introduction.
\end{Rem}

Notons que $X^{*}(\textbf{S})/(\fr_{\mathbf{S}} -1)X^{*}(\textbf{S})$ est un groupe fini. En effet pour tout $m \in \mathbb{N}^{*}$, nous pouvons injecter $X^{*}(\textbf{S})/(\fr_{\mathbf{S}} -1)X^{*}(\textbf{S})$ dans $X^{*}(\textbf{S})/(\fr_{\mathbf{S}}^m -1)X^{*}(\textbf{S})$ grâce à l'application $\Tr_{\fr^m/\fr} : \lambda  \mapsto \lambda+\fr_{\mathbf{S}}(\lambda)+ \cdots + \fr_{\mathbf{S}}^{m-1}(\lambda)$ (voir lemme \ref{lemTrace}). Il suffit alors de prendre $m$ tel que $\fr_{\mathbf{S}}^m=q^m$.

Soit $g \in \G$. La conjugaison par $g$ envoie naturellement un élément $\theta \in X^{*}(\textbf{S})/(\fr_{\mathbf{S}} -1)X^{*}(\textbf{S})$ sur un élément $g\theta \in X^{*}({}^{g}\textbf{S})/(\fr_{{}^{g}\mathbf{S}} -1)X^{*}({}^{g}\textbf{S})$. Ceci munit $\pairesSt$ d'une action de $\G$ et on note $\sim_{\G}$ la relation d'équivalence sur $\pairesSt$ induite par la $\G$-conjugaison.

\medskip

Nous construisons maintenant une application $\pairesSt \to \ensSystConj$, où l'on rappelle que $\ensSystConj$ désigne l'ensemble des systèmes de classes de conjugaison, définie en \ref{secSystcoh}. Pour ce faire nous allons passer par deux intermédiaires, un ensemble de triplets $\triplets=\{(\sigma,\gpfinialg{S},\theta)\}$ (que l'on va définir) et $\Jss$, l'ensemble des couples $(\sigma,s)$ (défini en \ref{defJss}). Cette application sera définie de la façon suivante. À partir d'une paire $(\mathbf{S},\theta) \in \pairesSt$ "globale", on va associer pour certain polysimplexes $\sigma$ dans l'appartement de $\mathbf{S}$, une paire $(\gpfinialg{S},\theta)$ "locale" sur $\quotredalg{\G}{\sigma}$. Ces paires donne des classe de conjugaison dans le dual grâce à la proposition \ref{proBijPairesSt}, et donc on obtient un système de classes de conjugaison.

Choisissons $\text{Frob}$, un Frobenius inverse dans $\Gal(\kalg/\kk)$. Ce dernier induit un automorphisme $\fr \in \Aut(G^{nr})$ qui agit naturellement sur $\bt[\knr]$ et tel que $\bt = \bt[\knr]^\fr$.

Soit $(\mathbf{S},\theta) \in \pairesSt$. Nommons $\mathcal{A}(\mathbf{S},\knr)$ l'appartement de $\bt[\knr]$ associé à $\gpalg{S}$. D'après \cite{debacker} lemme 2.2.1 (1), $\mathcal{A}(\mathbf{S},\knr)^{\fr}$ est un sous-ensemble de $\bt$ convexe, fermé , non-vide et est l'union des facettes de $\bt$ qu'il rencontre. Prenons $\sigma$ une facette de $\mathcal{A}(\mathbf{S},\knr)^{\fr}$ et posons $\textbf{\textsf{S}}$ l'image de $S^{nr} \cap \paranr{\G}{\sigma}$ dans $\quotredalg{\G}{\sigma}$ qui est alors un $\res$-tore maximal. Nous avons une identification naturelle entre $X^{*}(\gpfinialg{S})$ et $X:=X^{*}(\textbf{S})$. L'annexe B dans \cite{lanard} montre que l'action de $\fr$ sur $X$ correspond à l'action de $\fr_{\mathbf{S}}$. On a donc une bijection $X^{*}(\textbf{S})/(\fr_{\mathbf{S}} -1)X^{*}(\textbf{S}) \simeq X/(\fr-1)X$. Notons $\Irr(\gpfinialg{S}^{\fr})$ l'ensemble des caractères linéaires $\gpfinialg{S}^{\fr} \to \Ql^{\times}$. Nous avons également une bijection $X/(\fr-1)X \tosim \Irr(\gpfinialg{S}^{\fr})$ (voir annexe \ref{secIsoTores}), de sorte que $\theta$ s'identifie à un élément de $\Irr(\gpfinialg{S}^{\fr})$ que l'on note encore $\theta$.

\begin{Def}
\label{deftriplets}
	On définit $\triplets$ comme l'ensemble des triplets $(\sigma,\gpfinialg{S},\theta)$ où $\sigma \in \bt$, $\gpfinialg{S}$ est un $\res$-tore maximal de $\quotredalg{\G}{\sigma}$ et $\theta \in \Irr(\gpfinialg{S}^{\fr})$ est d'ordre inversible dans $\ld$. 
\end{Def}

On vient donc de définir une application
\[ \pairesSt \to \parties{\triplets} \]
où la notation $\parties{\cdot}$ désigne l'ensemble des parties, qui envoie $(\mathbf{S},\theta) \in \pairesSt$ sur l'ensemble des triplets $(\sigma,\gpfinialg{S},\theta)$ où $\sigma$ une facette de $\mathcal{A}(\mathbf{S},\knr)^{\fr}$ et $(\gpfinialg{S},\theta)$ est construit comme dans la discussion précédente.

\medskip

On va maintenant définir une application $\triplets \to \Jss$. Soit $(\sigma,\gpfinialg{S},\theta) \in \triplets$. Considérons $\gpfinialg*{S}$ un tore maximal $\fr$-stable de $\quotredalg*{\G}{\sigma}$ en dualité avec $\gpfinialg{S}$ sur $\res$ (comme rappelé dans l'annexe \ref{secIsoTores}). Rappelons nous que l'on a fixé au début un système compatible de racines de l'unité. Alors la proposition \ref{proBijPairesSt}, montre qu'il y a une bijection entre les classes de $\quotred{\G}{\sigma}$-conjugaison de paires $(\gpfinialg{S},\theta)$ et les classes de $\quotred*{\G}{\sigma}$-conjugaison de paires $(\gpfinialg*{S},s)$. Ainsi à un triplet $(\sigma,\gpfinialg{S},\theta)$ nous pouvons associer une classe de conjugaison semi-simple $s \in \ss{\quotred*{\G}{\sigma}}$, ce qui nous définit une application
\[ \triplets \to \Jss,\]
qui envoie $(\sigma,\gpfinialg{S},\theta)$ sur $(\sigma,s)$.
On obtient donc également une application $\parties{\triplets} \to \parties{\Jss}$.

\medskip

Notons que l'on a une bijection canonique $\parties{\Jss} \tosim \ensSystConj$, $A \mapsto \systConj{S}_{A}$, où $\systConj{S}_{A,\sigma} =\{ s, (\sigma,s) \in A \}$. On obtient ainsi
\[ \pairesSt \to \parties{\triplets} \to \parties{\Jss} \tosim \ensSystConj. \]

En composant avec l'application $ \ensSystConj \to  \ensSystConjc$ de la partie \ref{sectionSystMin} (qui on rappelle envoie un système de classes de conjugaison sur le plus petit système 0-cohérent le contenant), on construit l'application voulue
\[ \pairesSt \to  \ensSystConjc \]
Notons que comme les systèmes cohérents sont $\G$-équivariants, cette application passe au quotient pour fournir $\pairesSt/{\sim_{\G}} \to  \ensSystConjc$.

\begin{Def}
\label{defrepSt}
Soit $(\gpalg{S},\theta) \in\pairesSt$. On définit alors $\rep[\ld][(\gpalg{S},\theta)]{G}$ comme la sous-catégorie de Serre de $\rep[\ld][0]{G}$ découpée par le système cohérent d'idempotents, image de $(\gpalg{S},\theta)$ par $\pairesSt \to \ensSystConjc$.
\end{Def}

\subsection{Inclusion dans un système minimal}
\label{secLienIeJss}

Nous souhaitons démontrer dans cette section que l'application $\pairesSt \to  \ensSystConjc$ de la section \ref{secintroStheta} se factorise en $\pairesSt \to  \Smin$, où $\Smin$ est l'ensemble des systèmes 0-cohérents minimaux de la section \ref{sectionSystMin}. Pour définir $\pairesSt \to  \ensSystConjc$, nous sommes passés par deux intermédiaires $\triplets=\{(\sigma,\gpfinialg{S},\theta)\}$ et $\Jss=\{(\sigma,s)\}$. Pour montrer le résultat recherché, nous allons devoir étudier plus finement les liens entres les ensembles $\pairesSt$, $\triplets$, $\Jss$ et $\ensSystConjc$. En particulier nous allons définir $\tripletsm$ un sous-ensemble de $\triplets$ muni d'une relation d'équivalence $\sim$ et montrer que l'on a le diagramme commutatif
\[ \xymatrix{
		\pairesSt/{\sim_{\G}}  \ar@{->}[r] \ar@{->}[d]^-{\sim}& \Smin \ar@{<-}[d]^-{\sim}\\
		\tripletsm/{\sim}  \ar@{->}[r] &  \Jss/{\simJss}}.\]

\sautintro

Commençons par étudier le lien entre $\pairesSt$ et $\triplets$. Nous avons associé à $(\gpalg{S},\theta) \in \pairesSt$ un ensemble de triplets $(\sigma,\gpfinialg{S},\theta) \in \triplets$. Réciproquement si $(\sigma,\gpfinialg{S},\theta) \in \triplets$ et soit $\mathbf{S}$ un tore maximal non-ramifié de $\gpalg{\G}$ tel que l'image de $S^{nr} \cap \paranr{\G}{\sigma}$ dans $\quotredalg{\G}{\sigma}$ soit $\gpfinialg{S}$. La bijection $X^{*}(\textbf{S})/(\fr_{\mathbf{S}} -1)X^{*}(\textbf{S}) \tosim \Irr(\gpfinialg{S}^{\fr})$ de l'annexe \ref{secIsoTores}, identifie $\theta$ à un élément $\theta \in X^{*}(\textbf{S})/(\fr_{\mathbf{S}} -1)X^{*}(\textbf{S})$ et on obtient de la sorte une paire $(\mathbf{S},\theta)$ (qui est non-unique).

Le lemme 2.2.2 de \cite{debacker} (que l'on adapte ici en rajoutant les caractères $\theta$) montre que la classe de $\G$-conjugaison d'une paire $(\mathbf{S},\theta)$, obtenue à partir de $(\sigma,\gpfinialg{S},\theta)$, est indépendante des choix effectués, de sorte que l'on ait le lemme suivant

\begin{Lem}
\label{lemAppIt}
L'application $\triplets \longrightarrow \pairesSt/{\sim_{\G}}$ est bien définie.
\end{Lem}

On va alors introduire une relation d'équivalence sur $\triplets$ de sorte que l'application précédente devienne une bijection.

\medskip

Soit $\mathcal{A}$ un appartement de $\bt$. Pour $\Omega \subseteq \mathcal{A}$, on note $A(\mathcal{A},\Omega)$ le plus petit sous-espace affine de $\mathcal{A}$ contenant $\Omega$. Soient $\sigma_{1},\sigma_{2} \in \bt$. On dit que $\sigma_{1}$ et $\sigma_{2}$ sont fortement associées si pour tout appartement $\mathcal{A}$ contenant $\sigma_{1}$ et $\sigma_{2}$ (ou de façon équivalente pour un appartement $\mathcal{A}$ contenant $\sigma_{1}$ et $\sigma_{2}$) $A(\mathcal{A},\sigma_{1})=A(\mathcal{A},\sigma_{2})$. Alors, si $\sigma_{1}$ et $\sigma_{2}$ sont fortement associées, on a d'après \cite{debacker} section 3.1 une identification naturelle entre $\quotredalg{\G}{\sigma_{1}}$ et $\quotredalg{\G}{\sigma_{2}}$ que l'on note $\quotredalg{\G}{\sigma_{1}} \overset{id}{=} \quotredalg{\G}{\sigma_{2}}$. Celle-ci est définie grâce à l'application naturelle suivante, où  $\paranr{\G}{\sigma_i}$ désigne le parahorique de l'image canonique de $\sigma_{i}$ dans $\bt[\knr]$,
\[  \paranr{\G}{\sigma_1} \cap \paranr{\G}{\sigma_2} \longrightarrow \quotredalg{\G}{\sigma_{i}}(\mathfrak{\fr})\]
qui est surjective, de noyau
\[ \paranr{\G}{\sigma_1} \cap \radparanr{\G}{\sigma_2} = \radparanr{\G}{\sigma_1} \cap \paranr{\G}{\sigma_2}=\radparanr{\G}{\sigma_1} \cap \radparanr{\G}{\sigma_2}\]
($\radparanr{\G}{\sigma_i}$ est le pro-$p$-radical de $\paranr{\G}{\sigma_i}$).

\bigskip

Soient $(\sigma,\gpfinialg{S},\theta) \in \triplets$ et $g \in \G$. Alors il existe $\mathbf{S}$ un tore maximal non-ramifié de $\gpalg{\G}$ tel que l'image de $S^{nr}\cap \paranr{\G}{\sigma}$ dans $\quotredalg{\G}{\sigma}$ soit $\gpfinialg{S}$. Définissons alors ${}^{g}\textbf{\textsf{S}}$, un $\res$-tore maximal de $\quotredalg{\G}{g\sigma}$, par l'image de ${}^{g}S^{nr}\cap \paranr{\G}{g\sigma}$ dans $\quotredalg{\G}{g\sigma}$. Le tore ${}^{g}\textbf{\textsf{S}}$ correspond également à l'image de $\textbf{\textsf{S}}$ par l'isomorphisme $\transss{g}{\sigma}:\quotredalg{\G}{\sigma} \overset{\sim}{\rightarrow}\quotredalg{\G}{g\sigma}$. On définit également $g\theta$ par $g \theta := \theta \circ \transss{g}{\sigma}^{-1}$.

\begin{Def}
	Soient  $(\sigma_{1},\gpfinialg{S}_{1},\theta_{1}),(\sigma_{2},\gpfinialg{S}_{2},\theta_{2})\in \triplets$. On dit que $(\sigma_{1},\gpfinialg{S}_{1},\theta_{1})\sim (\sigma_{2},\gpfinialg{S}_{2},\theta_{2})$ si et seulement s'il existe un appartement $\mathcal{A}$ de $\bt$ et un $g \in \G$ tels que
	\begin{itemize}
		\item $\emptyset \neq A(\mathcal{A},\sigma_{1})=A(\mathcal{A},g\sigma_{2})$
		\item $\gpfinialg{S}_{1} \overset{id}{=} {}^{g}\gpfinialg{S}_{2}$ dans $\quotredalg{\G}{\sigma_{1}} \overset{id}{=} \quotredalg{\G}{g\sigma_{2}}$
		\item $\theta_{1} \overset{id}{=} g\theta_{2}$
	\end{itemize}
	
\end{Def}

Le lemme 3.2.2 de \cite{debacker} (en rajoutant les caractères $\theta_{1}$ et $\theta_{2}$) montre que $\sim$ est bien une relation d'équivalence.

\begin{Def}
\label{deftripletsm}
	Soit $(\sigma,\gpfinialg{S},\theta) \in \triplets$. On dira que $(\sigma,\gpfinialg{S},\theta)$ est minisotrope si $\gpfinialg{S}$ est $\res$-minisotrope dans $\quotredalg{\G}{\sigma}$ au sens de \cite{debacker}, c'est à dire que le tore $\res$-déployé maximal de $\textbf{\textsf{S}}$ coïncide avec le tore $\res$-déployé maximal du centre de $\quotredalg{\G}{\sigma}$. On notera alors $\tripletsm$ le sous-ensemble de $\triplets$ des triplets minisotropes.
\end{Def}

Notons que si $(\gpalg{S},\theta) \in \pairesSt$ et si $\sigma$ est une facette maximale de $\mathcal{A}(\mathbf{S},\knr)^{\fr}$ alors le triplet $(\sigma,\gpfinialg{S},\theta)$ associé est minisotrope (\cite{debacker} lemme 2.2.1 (3)).

\begin{Pro}
	\label{proBijCTIss}
	L'application du lemme \ref{lemAppIt} passe au quotient et fournit $\triplets/{\sim} \longrightarrow \pairesSt/{\sim_{\G}}$. De plus, cette dernière induit une bijection $\tripletsm/{\sim} \overset{\sim}{\longrightarrow} \pairesSt/{\sim_{\G}}$.
\end{Pro}

\begin{proof}
	Le lemme 3.3.3 et le théorème 3.4.1 de \cite{debacker} montrent le résultat pour les paires $(\sigma,\gpfinialg{S})$ à équivalence près (la même que celle définie sur les triplets en prenant des caractères triviaux) et les classes de $\G$-conjugaison de tores maximaux non-ramifiés. Cette démonstration s'adapte, en rajoutant les caractères $\theta$, pour obtenir le résultat voulu.
\end{proof}

Passons maintenant au lien entre $\triplets=\{(\sigma,\gpfinialg{S},\theta)\}$ et $\Jss=\{(\sigma,s)\}$. Nous avons déjà défini en \ref{secintroStheta} une application $\triplets \to \Jss$. En particulier pour obtenir le diagramme commutatif annoncé au début de cette section, il nous reste à montrer que cette application passe au quotient et fournit $\tripletsm/{\sim} \to \Jss/{\simJss}$.

\begin{Lem}
	\label{lemEgaliteid}
	Soient $\sigma_{1},\sigma_{2} \in \bt$ deux facettes fortement associées, $x \in \bts$ tels que $\sigma_{1} \geq x \leq \sigma_{2}$, et $(\sigma_{1},\gpfinialg{S}_{1},\theta_{1}),(\sigma_{2},\gpfinialg{S}_{2},\theta_{2}) \in \triplets$. Notons $t_{1} \in \ss{\quotred*{\G}{\sigma_{1}}}, t_{2} \in \ss{\quotred*{\G}{\sigma_{2}}}$ les classes de conjugaison rationnelles associées respectivement aux triplets $(\sigma_{1},\gpfinialg{S}_{1},\theta_{1})$ et $(\sigma_{2},\gpfinialg{S}_{2},\theta_{2})$. Supposons que $(\gpfinialg{S}_{1},\theta_{1}) \overset{id}{=} (\gpfinialg{S}_{2},\theta_{2})$ alors $\varphi^{*}_{\sigma_{1},x}(t_{1})=\varphi^{*}_{\sigma_{2},x}(t_{2})$.
\end{Lem}

\begin{proof}
	
	Les polysimplexes $\sigma_1$, $\sigma_2$ et $x$ sont dans un même appartement $\mathcal{A}$ associé à un tore $\gpalg{T}$. Notons $\textsf{\textbf{T}}$ le tore induit par $\gpalg{T}$ sur $\quotredalg{\G}{x}$. Comme $\sigma_{1} \geq x \leq \sigma_{2}$, les groupes $\quotredalg{\G}{\sigma_{1}}$ et $\quotredalg{\G}{\sigma_{2}}$ peuvent se relever de façon unique en des Levi $\textsf{\textbf{M}}_{\sigma_1}$ et $\textsf{\textbf{M}}_{\sigma_2}$ de $\quotredalg{\G}{x}$ contenant $\textsf{\textbf{T}}$. Comme $A(\mathcal{A},\sigma_{1})=A(\mathcal{A},\sigma_{2})$, les Levis $\gpfinialg{M}_{\sigma_i}$ ont même système de racines donc sont égaux. Notons $\gpfinialg{U}_{\sigma_i}$ le radical unipotent du parabolique de $\quotredalg{G}{x}$ image de $\paranr{\G}{\sigma_i}$ par l'application $\paranr{\G}{\sigma_i} \to \paranr{\G}{x} \to \quotredalg{G}{x}$. L'image de $\paranr{\G}{\sigma_1} \cap \paranr{\G}{\sigma_2}$ par l'application $\paranr{\G}{\sigma_1} \cap \paranr{\G}{\sigma_2} \to \paranr{\G}{x} \to \quotredalg{G}{x}$ est un sous-groupe algébrique  de radical unipotent $\gpfinialg{U}_{\sigma_1} \cap \gpfinialg{U}_{\sigma_2}$ et ayant pour facteur de Levi $\gpfinialg{M}_{\sigma_1}=\gpfinialg{M}_{\sigma_2}$. Comme $id$ est défini grâce à la surjection $\paranr{\G}{\sigma_1} \cap \paranr{\G}{\sigma_2} \longrightarrow \quotredalg{\G}{\sigma_{i}}$, la discussion précédente montre que $id$ correspond à l'égalité entre $\textsf{\textbf{M}}_{\sigma_1}$ et $\textsf{\textbf{M}}_{\sigma_2}$.

	Maintenant, les paires $(\gpfinialg{S}_{1},\theta_{1})$ et $(\gpfinialg{S}_{2},\theta_{2})$ induisent des paires $(\gpfinialg{S}_{1,x},\theta_{1,x})$ et $(\gpfinialg{S}_{2,x},\theta_{2,x})$ avec $\gpfinialg{S}_{1,x}$ et $\gpfinialg{S}_{2,x}$ des tores de $\quotredalg{\G}{x}$. Si $(\gpfinialg{S}_{i},\theta_{i})$ correspond à la classe de conjugaison semi-simple $t_{i} \in \ss{\quotred*{\G}{\sigma_{i}}}$ alors $(\gpfinialg{S}_{i,x},\theta_{i,x})$ correspond à $\varphi^{*}_{\sigma_{i},x}(t_{i}) \in \ss{\quotred*{\G}{x}}$. L'identification $(\gpfinialg{S}_{1},\theta_{1}) \overset{id}{=} (\gpfinialg{S}_{2},\theta_{2})$ revient à dire que $(\gpfinialg{S}_{1,x},\theta_{1,x}) = (\gpfinialg{S}_{2,x},\theta_{2,x})$, et donc que $\varphi^{*}_{\sigma_{1},x}(t_{1})=\varphi^{*}_{\sigma_{2},x}(t_{2})$.
\end{proof}

\begin{Pro}
\label{protripletsmjss}
	L'application $\tripletsm \to \Jss$ passe au quotient et donne $\tripletsm/{\sim} \to \Jss/{\simJss}$.
\end{Pro}

\begin{proof}
	Soient $(\sigma_{1},\gpfinialg{S}_{1},\theta_{1}),(\sigma_{2},\gpfinialg{S}_{2},\theta_{2})\in \tripletsm$ deux triplets tels que $(\sigma_{1},\gpfinialg{S}_{1},\theta_{1}) \sim (\sigma_{2},\gpfinialg{S}_{2},\theta_{2})$. Notons $(\sigma_1,s_1),(\sigma_{2},s_2) \in \Jss$ les images respectives de $(\sigma_{1},\gpfinialg{S}_{1},\theta_{1}),(\sigma_{2},\gpfinialg{S}_{2},\theta_{2})$ par l'application $\tripletsm \to \Jss$. Il s'agit de montrer que  $(\sigma_1,s_1) \simJss (\sigma_{2},s_2)$.

Comme $(\sigma_{1},\gpfinialg{S}_{1},\theta_{1}) \sim (\sigma_{2},\gpfinialg{S}_{2},\theta_{2})$, il existe un appartement $\mathcal{A}$ de $\bt$ et un $g \in \G$ tels que
	\begin{itemize}
		\item $\emptyset \neq A(\mathcal{A},\sigma_{1})=A(\mathcal{A},g\sigma_{2})$
		\item $\gpfinialg{S}_{1} \overset{id}{=} {}^{g}\gpfinialg{S}_{2}$ dans $\quotredalg{\G}{\sigma_{1}} \overset{id}{=} \quotredalg{\G}{g\sigma_{2}}$
		\item $\theta_{1} \overset{id}{=} g\theta_{2}$
	\end{itemize}
	
	Notons $s_{1} \in (\quotred*{\G}{\sigma_{1}})_{ss,\ld}$ (resp. $s_{2} \in (\quotred*{\G}{\sigma_{2}})_{ss,\ld}$) la classe de conjugaison semi-simple associée à $(\gpfinialg{S}_{1},\theta_{1})$ (resp. à $(\gpfinialg{S}_{2},\theta_{2})$). Alors $\varphi^{*}_{g,\sigma_{2}}(s_{2})$ est la classe de conjugaison associée à $({}^{g}\gpfinialg{S}_{2},g\theta_{2})$. Or nous savons que $(\sigma_{2},s_{2}) \sim (g\sigma_{2},\varphi^{*}_{g,\sigma_{2}}(s_{2}))$ et donc nous pouvons supposer que $g=1$.
	
	Les polysimplexes $\sigma_{1}$ et $\sigma_{2}$ sont maximaux dans $A(\mathcal{A},\sigma_{1})=A(\mathcal{A},\sigma_{2})$. Prenons une suite $\tau_{1}=\sigma_{1},\cdots,\tau_{m}=\sigma_{2}$ de polysimplexes maximaux de $A(\mathcal{A},\sigma_{1})$ et une suite de sommets $x_{1},\cdots,x_{m-1}$ telle que pour tout $i \in [\![1,m-1]\!]$, $\tau_{i} \geq x_{i} \leq \tau_{i+1}$. Pour chaque $i \in [\![1,m]\!]$, $\tau_{i}$ est un polysimplexe maximal de $A(\mathcal{A},\sigma_{1})$. En particulier $A(\mathcal{A},\sigma_{1})=A(\mathcal{A},\tau_{i})$ et donc $\sigma_{1}$ et $\tau_{i}$ sont fortement associés. On a donc une identification naturelle  $\quotredalg{\G}{\sigma_{1}} \overset{id}{=} \quotredalg{\G}{\tau_{i}}$ qui nous permet de définir $(\gpfinialg{S}^{(i)},\theta^{(i)})$ tel que $(\gpfinialg{S}^{(i)},\theta^{(i)}) \overset{id}{=} (\gpfinialg{S}_{1},\theta_{1})$. Appelons $t_{i} \in \ss*{\quotred*{\G}{\tau_{i}}}$ la classe de conjugaison semi-simple associée à $(\gpfinialg{S}^{(i)},\theta^{(i)})$.

	Fixons un $i \in [\![1,m-1]\!]$. Alors $\tau_{i} \geq x_{i} \leq \tau_{i+1}$, $A(\mathcal{A},\tau_{i})=A(\mathcal{A},\tau_{i+1})$ et $(\gpfinialg{S}^{(i)},\theta^{(i)}) \overset{id}{=} (\gpfinialg{S}^{(i+1)},\theta^{(i+1)})$ donc par le lemme \ref{lemEgaliteid} $\varphi^{*}_{\tau_{i},x_{i}}(t_{i})=\varphi^{*}_{\tau_{i+1},x_{i}}(t_{i+1})$. Ceci nous montre que $(\tau_{i},t_{i}) \sim (\tau_{i+1},t_{i+1})$ et donc que $(\sigma_{1}, s_{1}) = (\tau_{1},t_{1}) \sim (\tau_{m},t_{m})=(\sigma_{2},s_{2})$ ce qui achève la preuve.
\end{proof}

On peut maintenant regrouper toutes les applications que l'on vient de construire dans cette section pour obtenir les résultats que l'on recherchait.

\begin{Pro}
\label{proDecompoAppPS}
Les applications précédentes rendent le diagramme suivant commutatif
	\[ \xymatrix{
		\pairesSt/{\sim_{\G}}  \ar@{->}[r] \ar@{<-}[d]^-{\sim}& \ensSystConjc \ar@{<-}[d]\\
		\tripletsm/{\sim}  \ar@{->}[r]  &  \Jss/{\simJss}}.\]
\end{Pro}

\begin{proof}

Soit $(\tau,\gpfinialg{S}_{\tau},\theta) \in \tripletsm$ et $(\mathbf{S},\theta) \in \pairesSt$ un représentant de la classe de conjugaison qui lui est associée par $\tripletsm/{\sim} \to \pairesSt/{\sim_{\G}}$. L'application $\pairesSt \to \ensSystConjc$ de la partie \ref{secintroStheta} est également définie à partir de $\triplets \to \Jss \to \ensSystConjc$. Prenons $\sigma$ une facette de $\mathcal{A}(\mathbf{S},\knr)^{\fr}$ et nommons $(\sigma,\gpfinialg{S}_{\sigma},\theta) \in \triplets$ le triplet associé. Prenons également $\tau'$ une facette maximale de $\mathcal{A}(\mathbf{S},\knr)^{\fr}$ telle que $\sigma \leq \tau'$. Nous avons un triplet $(\tau',\gpfinialg{S}_{\tau'},\theta)$ associé à $(\mathbf{S},\theta)$. Comme $\tau'$ est une facette maximale de $\mathcal{A}(\mathbf{S},\knr)^{\fr}$, $(\tau',\gpfinialg{S}_{\tau'},\theta) \in \tripletsm$ (\cite{debacker} lemme 2.2.1 (3)). De plus, $(\tau',\gpfinialg{S}_{\tau'},\theta) \sim (\tau,\gpfinialg{S}_{\tau},\theta)$ et on peut donc supposer que $(\tau',\gpfinialg{S}_{\tau'},\theta) = (\tau,\gpfinialg{S}_{\tau},\theta)$. Il s'agit donc de montrer que $(\sigma,\gpfinialg{S}_{\sigma},\theta)$ et $(\tau,\gpfinialg{S}_{\tau},\theta)$ ont même image par  $\triplets \to \Jss \to \ensSystConjc$.

Notons $(\sigma,s_\sigma)$ et $(\tau,s_{\tau})$ les images respectives de $(\sigma,\gpfinialg{S}_{\sigma},\theta)$ et $(\tau,\gpfinialg{S}_{\tau},\theta)$ par $\triplets \to \Jss$. Comme $\sigma \leq \tau$, le groupe $\quotredalg{\G}{\tau}$ peut se relever en un unique Levi de $\quotredalg{\G}{\sigma}$ contenant $\gpfinialg{S}_{\sigma}$. Ceci nous permet d'identifier $\gpfinialg{S}_{\tau}$ à $\gpfinialg{S}_{\sigma}$. Donc si $(\gpfinialg*{S},t)$ est une paire duale à $(\gpfinialg{S}_{\tau},\theta)$ elle l'est également pour $(\gpfinialg{S}_{\sigma},\theta)$. On en déduit que $s_{\sigma}=\varphi^{*}_{\tau,\sigma}(s_{\tau})$. Ainsi $(\sigma,s_\sigma) \simJss (\tau,s_{\tau})$ et on a le résultat.
\end{proof}

\begin{Pro}
\label{proInclusSmin}
L'application $\pairesSt/{\sim_{\G}} \to \ensSystConjc$ se factorise en  $\pairesSt/{\sim_{\G}} \to \Smin$.
\end{Pro}

\begin{proof}
Nous savons que l'on a la bijection $\Jss/{\simJss} \tosim \Smin$ (section \ref{sectionSystMin}). Ainsi la proposition \ref{proDecompoAppPS} montre que l'on a le diagramme commutatif
\[ \xymatrix{
		\pairesSt/{\sim_{\G}}  \ar@{->}[r] \ar@{->}[d]^-{\sim}& \Smin \ar@{<-}[d]^-{\sim}\\
		\tripletsm/{\sim}  \ar@{->}[r] &  \Jss/{\simJss}}\]
d'où le résultat.
\end{proof}

En particulier on en déduit le corollaire suivant.

\begin{Cor}
Soient $(\gpalg{S},\theta),(\gpalg{S}',\theta') \in \pairesSt$. Alors $\rep[\ld][(\gpalg{S},\theta)]{G}=\rep[\ld][(\gpalg{S}',\theta')]{G}$ ou $\rep[\ld][(\gpalg{S},\theta)]{G} \perp \rep[\ld][(\gpalg{S}',\theta')]{G}$.
\end{Cor}

\subsection{Paires minimales}

Dans la partie \ref{secLienIeJss} nous avons montré que l'on a une application $\pairesSt/{\sim_{\G}} \to \Smin$, qui envoie une classe d'équivalence de paires $(\gpalg{S},\theta)$ sur un système 0-cohérent minimal. Cette application est surjective mais n'est pas injective. Nous proposons ici de construire $\pairesStmin$, un sous-ensemble de $\pairesSt$, tel que  $\pairesStmin/{\sim_{\G}} \tosim \Smin$.

\sautintro

Nous rappelons que l'on a défini en \ref{deftriplets} $\triplets$ l'ensemble des triplets $(\sigma,\gpfinialg{S},\theta)$, et en \ref{deftripletsm} $\tripletsm$ le sous-ensemble des triplets minisotropes. Nous allons maintenant définir les triplets elliptiques.
	
\begin{Def}
	On dira qu'un triplet $(\sigma,\gpfinialg{S},\theta) \in \triplets$ est elliptique si la classe de conjugaison $s$ associée est elliptique dans $\quotred*{G}{\sigma}$, au sens où elle ne rencontre aucun sous-groupe de Levi propre, et on note $\tripletse$ l'ensemble des triplets $(\sigma,\gpfinialg{S},\theta)$ elliptiques. 
\end{Def}

Notons que si $(\sigma,\gpfinialg{S},\theta)$ est elliptique alors $\gpfinialg{S}$ est $\res$-minisotrope dans $\quotredalg{\G}{\sigma}$, d'où les inclusions
\[ \tripletse \subseteq \tripletsm \subseteq \triplets.\]

\begin{Def}
\label{defpairesmin}
	On dira alors qu'une paire $(\mathbf{S},\theta) \in \pairesSt$ est minimale si elle est dans l'image de $\tripletse/{\sim}$, par l'application de la proposition \ref{proBijCTIss}. On notera $\pairesStmin$ l'ensemble des paires $(\mathbf{S},\theta)$ minimales.
\end{Def}

Avant de passer à la démonstration de la bijection $\pairesStmin/{\sim_{\G}} \tosim \Smin$ décrivons un peu plus précisément l'ensemble $\pairesStmin$. Soit $(\mathbf{S},\theta) \in \pairesStmin$ une paire minimale. Par définition si $\sigma \in \mathcal{A}(\mathbf{S},\knr)^{\fr}$ est une facette maximale alors le triplet $(\sigma,\gpfinialg{S},\theta)$ (obtenu par le procédé de la section \ref{secintroStheta}) est elliptique. Mais qu'en est-il lorsque $\sigma$ n'est pas maximal? Nous allons décrire dans la suite les triplets $(\sigma,\gpfinialg{S},\theta)$ que l'on peut obtenir à partir d'une paire $(\mathbf{S},\theta)$ minimale.

\begin{Def}
	Soit $(\sigma,\gpfinialg{S},\theta) \in \triplets$. Notons $(\sigma,\gpfinialg*{S},t)$ un triplet qui lui est associé (qui est défini à $\quotred*{\G}{\sigma}$-conjugaison près, comme dans la proposition \ref{proBijPairesSt}). On dit que $(\sigma,\gpfinialg{S},\theta)$ est minimal si et seulement si le $\res$-rang de $\gpfinialg{S}$ est égal à celui de $C_{\quotredalg*{\G}{\sigma}}(t)$, le centralisateur de $t$ dans $\quotredalg*{\G}{\sigma}$, c'est à dire si $\gpfinialg*{S}$ est maximalement déployé  dans ce centralisateur.
\end{Def}

Soient $\sigma,\tau \in \bt$ tels que $\tau \leq \sigma$. Prenons une classe de $\quotred*{\G}{\sigma}$-conjugaison de paires $(\gpfinialg*{S},t)$, où $\gpfinialg*{S}$ est un $\res$-tore maximal de $\quotredalg*{\G}{\tau}$ et $t \in (\gpfinialg*{S})^{\fr}$. Le choix d'un relèvement de $\quotredalg*{\G}{\sigma}$ en un Levi de $\quotredalg*{\G}{\tau}$ permet de relever $\gpfinialg*{S}$ en $\gpfinialg{S}'^{*}$ un $\res$-tore maximal de $\quotredalg*{\G}{\tau}$ et $t$ en $t' \in (\gpfinialg{S}'^{*})^{\fr}$. Cette nouvelle paire $(\gpfinialg{S}'^{*},t')$ est alors définie à $\quotred*{\G}{\tau}$-conjugaison près, et on note $\varphi_{\sigma,\tau}^{*}$ l'application qui à la classe de conjugaison de $(\gpfinialg*{S},t)$ associe celle de $(\gpfinialg{S}'^{*},t')$.

Notons que si $s$ est la classe de conjugaison de $t$ dans $\quotred*{\G}{\sigma}$ alors $\varphi_{\sigma,\tau}^{*}(s)$ est celle de $t'$.

\begin{Lem}
	\label{lemMaxiDeploye}
	Soient $\sigma,\tau \in \bt$ tels que $\tau \leq \sigma$ et $(\sigma,\gpfinialg{S},\theta),(\tau,\gpfinialg{S}',\theta') \in \triplets$ tels que $\varphi_{\sigma,\tau}^{*}(\gpfinialg*{S},t)=(\gpfinialg{S}'^{*},t')$, où $(\sigma,\gpfinialg*{S},t)$ (resp. $(\tau,\gpfinialg{S}'^{*},t')$) est associé à $(\sigma,\gpfinialg{S},\theta)$ (resp. $(\tau,\gpfinialg{S}',\theta')$). Alors
	
	\begin{enumerate}
		\item $(\sigma,\gpfinialg{S},\theta)$ est minimal et $\gpfinialg{S}$ est $\res$-minisotrope dans $\quotredalg{G}{\sigma}$ si et seulement si $(\sigma,\gpfinialg{S},\theta)$ est elliptique
		\item $(\sigma,\gpfinialg{S},\theta)$ est minimal si et seulement si $(\tau,\gpfinialg{S}',\theta')$ est minimal.
	\end{enumerate}
	
\end{Lem}

\begin{proof}
	Notons $\gpfinialg{S}^{*d}$ le sous-tore déployé maximal de $\gpfinialg*{S}$.
	
	Commençons par montrer 1. Supposons que $(\sigma,\gpfinialg{S},\theta)$ est minimal et que $\gpfinialg{S}$ est $\res$-minisotrope dans $\quotredalg{G}{\sigma}$. Prenons $\textsf{\textbf{L}}$ un Levi minimal de $\quotredalg*{\G}{\sigma}$ tel que $t \in \textsf{\textbf{L}}$, de sorte que $t$ soit elliptique dans $\textsf{\textbf{L}}$. Comme $t$ est un élément semi-simple, il existe $\gpfinialg{S}_{\textsf{\textbf{L}}}$ un tore maximal $\fr$-stable de $\textsf{\textbf{L}}$ tel que $t \in \gpfinialg{S}_{\textsf{\textbf{L}}}$. Appelons $\gpfinialg{S}_{\textsf{\textbf{L}}}^{d}$ sa composante déployée. Comme $t$ est elliptique dans $\textsf{\textbf{L}}$, $\gpfinialg{S}_{\textsf{\textbf{L}}}$ est $\res$-minisotrope dans $\textsf{\textbf{L}}$ et donc $\gpfinialg{S}_{\textsf{\textbf{L}}}^{d}$ coïncide avec la composante déployée du centre de $\textsf{\textbf{L}}$. Maintenant $\gpfinialg{S}$ est $\res$-minisotrope, donc $\gpfinialg{S}^{*d}$ est la composante déployée de $Z(\quotredalg*{\G}{\sigma})$, le centre de $\quotredalg*{\G}{\sigma}$. Et comme $Z(\quotredalg*{\G}{\sigma}) \subseteq Z(\textsf{\textbf{L}})$, on en déduit par maximalité de $\gpfinialg{S}_{\textsf{\textbf{L}}}^{d}$ que $\gpfinialg{S}^{*d}\subseteq \gpfinialg{S}_{\textsf{\textbf{L}}}^{d}$. Enfin, $\gpfinialg{S}^{*d}\subseteq \gpfinialg{S}_{\textsf{\textbf{L}}}^{d} \subset C_{\quotredalg*{\G}{\sigma}}(t)$, et comme $(\sigma,\gpfinialg{S},\theta)$ est minimal, $\gpfinialg{S}^{*d} = \gpfinialg{S}_{\textsf{\textbf{L}}}^{d}$ et donc $\textsf{\textbf{L}}=\quotredalg*{\G}{\sigma}$.
	
	Réciproquement, supposons $t$ elliptique. Nous savons déjà que $\gpfinialg{S}$ est $\res$-minisotrope. Il reste à montrer que $\gpfinialg{S}^{*d}$ est maximal parmi les tores déployés de $C_{\quotredalg*{\G}{\sigma}}(t)$. Prenons alors $\textsf{\textbf{T}}^{d}$ un tore déployé de $C_{\quotredalg*{\G}{\sigma}}(t)$ contenant $\gpfinialg{S}^{*d}$. Notons $\textsf{\textbf{M}}$ le centralisateur de $\textsf{\textbf{T}}^{d}$ dans $\quotredalg*{\G}{\sigma}$ qui est un Levi de $\quotredalg*{\G}{\sigma}$ contenant $t$. Comme $t$ est elliptique $\textsf{\textbf{M}}=\quotredalg*{\G}{\sigma}$, et comme $\gpfinialg{S}$ est $\res$-minisotrope, $\textsf{\textbf{T}}^{d}=\gpfinialg{S}^{*d}$.
	
	\medskip
	
	Démontrons maintenant 2. Commençons par relever $\gpfinialg*{S}$ en un tore de $\quotredalg*{\G}{\tau}$, que l'on note encore $\gpfinialg*{S}$, de sorte que l'on puisse identifier $(\gpfinialg{S}'^{*},t')$ à $(\gpfinialg*{S},t)$ (après conjugaison). Notons $\textsf{\textbf{M}}_{\sigma}$ l'unique relèvement de $\quotredalg*{\G}{\sigma}$ en un Levi de $\quotredalg*{\G}{\tau}$ contenant $\gpfinialg*{S}$.
	
	Il faut alors montrer que $\gpfinialg{S}^{*d}$ est maximal parmi les tores déployés de $C_{\quotredalg*{\G}{\tau}}(t)$ si et seulement s'il l'est parmi ceux de $C_{\textsf{\textbf{M}}_{\sigma}}(t)$. Il est clair que si $\gpfinialg{S}$ est maximalement déployé dans $C_{\quotredalg*{\G}{\tau}}(t)$ il l'est aussi dans $C_{\textsf{\textbf{M}}_{\sigma}}(t)$.
	
	Supposons donc $\gpfinialg{S}^{*d}$ maximal parmi les tores déployés de $C_{\textsf{\textbf{M}}_{\sigma}}(t)$. Notons $\textsf{\textbf{M}}$ le centralisateur de $\gpfinialg{S}^{*d}$ dans $\textsf{\textbf{M}}_{\sigma}$ qui est un Lévi de $\textsf{\textbf{M}}_{\sigma}$ contenant $t$. Le tore $\gpfinialg{S}^{*d}$ est alors $\res$-minisotrope dans $\textsf{\textbf{M}}$ et il est maximal parmi les tores déployé de $C_{\textsf{\textbf{M}}}(t)$ donc par la propriété 1, $t$ est elliptique dans $\textsf{\textbf{M}}$. Comme $\gpfinialg{S}^{*d}$ est la composante déployée de $Z(\textsf{\textbf{M}})$, on a également que $\textsf{\textbf{M}}$ est le centralisateur de $\gpfinialg{S}^{*d}$ dans $\quotredalg*{\G}{\tau}$. Prenons $\textsf{\textbf{T}}^{d}$ un tore déployé de $C_{\quotredalg*{\G}{\tau}}(t)$ contenant $\gpfinialg{S}^{*d}$. Son centralisateur dans $\quotredalg*{\G}{\tau}$ est un Lévi $\textsf{\textbf{L}}$ contenant $t$ et contenu dans $\textsf{\textbf{M}}$. Mais comme $t$ est elliptique dans $\textsf{\textbf{M}}$, $\textsf{\textbf{L}}=\textsf{\textbf{M}}$ et $\gpfinialg{S}^{*d}=\textsf{\textbf{T}}^{d}$.
\end{proof}

\begin{Cor}
	Soit $(\mathbf{S},\theta) \in \pairesSt$. Alors $(\mathbf{S},\theta)$ est minimale si et seulement si pour tout $\sigma \in \mathcal{A}(\mathbf{S},\knr)^{\fr}$ le triplet associé $(\sigma,\gpfinialg{S},\theta)$ est minimal, si et seulement si pour un $\sigma \in \mathcal{A}(\mathbf{S},\knr)^{\fr}$ le triplet $(\sigma,\gpfinialg{S},\theta)$ est minimal.
	
	Notons que dans ce cas, si $\sigma$ est une facette maximale alors $(\sigma,\gpfinialg{S},\theta)$ est elliptique.
\end{Cor}

Revenons à la démonstration de $\pairesStmin/{\sim_{\G}} \tosim \Smin$. La proposition \ref{proInclusSmin} montre que l'on peut identifier l'application $\pairesSt/{\sim_{\G}} \to \Smin$ à $\tripletsm/{\sim} \to \Jss/{\simJss}$. Comme par construction des triplets elliptique on a une bijection $\pairesStmin/{\sim_{\G}} \tosim \tripletse/{\sim}$, il nous suffit donc de montrer que l'on a une bijection $\tripletse/{\sim} \tosim \Jss/{\simJss}$.

\begin{Pro}
\label{proEqTripletse}
	L'application $\triplets \to \Jss$, de la section \ref{secintroStheta}, induit une bijection $\tripletse/{\sim} \tosim \Jss/{\simJss}$.
\end{Pro}

\begin{proof}
Nous savons déjà par la proposition \ref{protripletsmjss} que l'on a une application $\tripletsm/{\sim} \to \Jss/{\simJss}$ et donc $\tripletse/{\sim} \to \Jss/{\simJss}$.

L'application $\tripletse/{\sim} \to \Jss/{\simJss}$ est surjective. En effet soit $(\sigma,s) \in \Jss$. Il existe un polysimplexe maximal $\tau \geq \sigma$ tel que $\varphi^{*-1}_{\tau,\sigma}(s) \neq \emptyset$. Prenons alors $t \in \varphi^{*-1}_{\tau,\sigma}(s)$. La classe de conjugaison $t$ est alors elliptique dans $\quotred*{\G}{\tau}$. Il suffit alors de prendre un triplet $(\tau,\gpfinialg{S},\theta)$ qui donne $(\tau,t)$. Ce dernier est elliptique puisque $t$ est elliptique et il convient car $(\sigma,s) \simJss (\tau,t)$.

\medskip

Il reste à montrer l'injectivité.	Soient $(\sigma_{1},\gpfinialg{S}_{1},\theta_{1}), (\sigma_{2},\gpfinialg{S}_{2},\theta_{2}) \in \tripletse$ deux triplets auxquels sont associés respectivement $(\sigma_1,s_1), (\sigma_2,s_2) \in \Jss$. On suppose que $(\sigma_1,s_1) \simJss (\sigma_2,s_2)$ et on veut montrer que $(\sigma_{1},\gpfinialg{S}_{1},\theta_{1}) \sim (\sigma_{2},\gpfinialg{S}_{2},\theta_{2})$. Nous savons par la proposition \ref{proClotureTransitive} que $\sim$ est la clôture transitive de $\simJsst$, on peut donc supposer que $(\sigma_1,s_1)\simJsst (\sigma_2,s_2)$. Prenons un triplet $(\sigma_{1}, \gpfinialg{S}_{1}^{*},t_{1})$ (resp. $(\sigma_{2}, \gpfinialg{S}_{2}^{*},t_{2})$) avec $t_{1} \in (\gpfinialg{S}_{1}^{*})^{\fr}$ (resp. $t_{2} \in (\gpfinialg{S}_{2}^{*})^{\fr}$) en dualité avec $(\sigma_{1},\gpfinialg{S}_{1},\theta_{1})$ (resp. $(\sigma_{2},\gpfinialg{S}_{2},\theta_{2})$), tel que $t_{1}$ (resp. $t_{2}$) a pour classe de conjugaison $s_{1}$ (resp. $s_{2}$).

	Commençons par traiter le cas $\sigma_{1}=\sigma_{2}$ et $s_{1}=s_{2}$. 
 Il existe donc $g \in \quotred*{\G}{\sigma_{1}}$ tel que $t_{1}=\Ad(g)(t_{2})$. Par le lemme \ref{lemMaxiDeploye}, $\gpfinialg{S}_{1}^{*}$ et ${}^{g}\gpfinialg{S}_{2}^{*}$ sont maximalement déployés dans $C_{\quotredalg*{\G}{\sigma_{1}}}(t_{1})$, donc quitte à modifier $g$ on peut supposer que $\gpfinialg{S}_{1}^{*} = {}^{g}\gpfinialg{S}_{2}^{*}$.
	On en déduit que les paires $(\gpfinialg{S}_{1}^{*},t_{1})$ et $(\gpfinialg{S}_{2}^{*},t_{2})$ sont conjuguées sous $\quotred*{\G}{\sigma_{1}}$ et donc que les paires $(\gpfinialg{S}_{1},\theta_{1})$ et $(\gpfinialg{S}_{2},\theta_{2})$ sont conjuguées sous $\quotred{\G}{\sigma_{1}}$. On a bien que $(\sigma_{1},\gpfinialg{S}_{1},\theta_{1})\sim(\sigma_{2},\gpfinialg{S}_{2},\theta_{2})$.

	Revenons au cas général. Comme $(\sigma_{1},s_{1}) \simJsst (\sigma_{2},s_{2})$, deux cas se présentent :
	\begin{enumerate}
		\item Il existe $g \in \G$ tel que $g\sigma_{2}=\sigma_{1}$ et $s_{1}=\varphi^{*}_{g,\sigma_{2}}(s_{2})$.
		
		Alors $(\gpfinialg{S}_{1},\theta_{1})$ et $g(\gpfinialg{S}_{2},\theta_{2})$ donnent la même classe de conjugaison dans $\ss{\quotred*{\G}{\sigma_{1}}}$ et par le premier cas $(\sigma_{1},\gpfinialg{S}_{1},\theta_{1})\sim(\sigma_{2},\gpfinialg{S}_{2},\theta_{2})$.
		
		\item Il existe $x \in \bts$ tel que $\sigma_{1} \geq x \leq \sigma_{2}$ et $\varphi_{\sigma_{2},x}^{*}(s_{2}) = \varphi_{\sigma_{1},x}^{*}(s_{1})$.

		Choisissons un appartement contenant $\sigma_1$ et $\sigma_2$. Cet appartement est associé à un tore $\gpalg{T}$ et l'on note $\gpfinialg{T}$ le tore induit sur $\quotredalg{G}{x}$ et $\gpfinialg*{T}$ un dual sur $\quotredalg*{G}{x}$. Le tore $\gpfinialg{T}$ permet de relever de façon unique $\quotredalg{G}{\sigma_1}$ et $\quotredalg{G}{\sigma_2}$ en des sous-groupes de Levi de $\quotredalg{G}{x}$. De même $\gpfinialg*{T}$ permet de relever $\quotredalg*{G}{\sigma_1}$ et $\quotredalg*{G}{\sigma_2}$. Ceci nous permet de relever $\gpfinialg{S}_1$, $\gpfinialg{S}_2$ en des tores de $\quotredalg{G}{x}$ et $\gpfinialg*{S}_1$, $\gpfinialg*{S}_2$ en des tores de $\quotredalg*{\G}{x}$. Dans ce cas, $(\gpfinialg{S}_{1},\theta_{1})$ correspond à $(\gpfinialg*{S}_{1},t_{1})$ et $(\gpfinialg{S}_{2},\theta_{2})$ à $(\gpfinialg*{S}_{2},t_{2})$ avec $t_{1}$ et $t_{2}$ conjugués dans $\quotred*{\G}{x}$ ($t_{i}$ est un représentant de la classe de conjugaison $\varphi_{\sigma_{i},x}^{*}(s_{i})$). Il existe donc $g^{*} \in \quotred*{\G}{x}$ tel que $t_{1}=\Ad(g^{*})(t_{2})$. Par le lemme \ref{lemMaxiDeploye}, $\gpfinialg*{S}_{1}$ et ${}^{g^{*}}\gpfinialg*{S}_{2}$ sont maximalement déployés dans $C_{\quotred*{\G}{x}}(t_{1})$. Deux tores maximalement déployés sont conjugués donc quitte à modifier $g^{*}$ on peut supposer que $\gpfinialg*{S}_{1} = {}^{g^{*}}\gpfinialg*{S}_{2}$. Par conséquent les paires $(\gpfinialg*{S}_{1},t_{1})$ et $(\gpfinialg*{S}_{2},t_{2})$ sont conjuguées sous $\quotred*{\G}{x}$ donc les paires $(\gpfinialg{S}_{1},\theta_{1})$ et $(\gpfinialg{S}_{2},\theta_{2})$ sont conjuguées sous $\quotred{\G}{x}$.

		Il existe $\overline{g} \in \quotred{\G}{x}$ tel que $(\gpfinialg{S}_{1},\theta_{1}) = \overline{g}(\gpfinialg{S}_{2},\theta_{2})$. Prenons $g \in \para{\G}{x}$ un relèvement de $\overline{g} \in \quotred{\G}{x} \simeq \para{\G}{x} / \radpara{\G}{x}$. Considérons également $\mathbf{S}_{1}$ et $\mathbf{S}_{2}$ des tores maximaux non-ramifiés de $\gpalg{G}$ qui relèvent $\gpfinialg{S}_{1}$ et $\gpfinialg{S}_{2}$. Alors $\mathbf{S}_{1}$ et ${}^{g}\mathbf{S}_{2}$ ont même image dans $\quotredalg{\G}{x}$ donc par \cite{debacker} lemme 2.2.2, $\mathbf{S}_{1}$ et ${}^{g}\mathbf{S}_{2}$ sont $\radpara{\G}{x}$-conjugués. Quitte à changer le relèvement de $\overline{g}$ choisi, on peut supposer que $\mathbf{S}_{1}={}^{g}\mathbf{S}_{2}$. Le polysimplexe $\sigma_{1}$ est une facette de dimension maximale dans $\mathcal{A}(\gpalg{S}_1,\knr)^\fr$ et $g\sigma_{2}$ est une facette de dimension maximale dans $\mathcal{A}({}^{g}\gpalg{S}_2,\knr)^\fr=\mathcal{A}(\gpalg{S}_1,\knr)^\fr$. Prenons $\mathcal{A}$ un appartement contenant $\sigma_{1}$ et $g\sigma_{2}$. Par \cite{debacker} lemme 2.2.1 (4) nous avons $A(\mathcal{A},\sigma_{1})=A(\mathcal{A},g\sigma_{2})$. On a bien que $\gpfinialg{S}_{1} \overset{id}{=} {}^{g} \gpfinialg{S}_{2}$. De plus dans $\quotredalg{\G}{x}$ on a $\gpfinialg{S}_{1} = {}^{\overline{g}} \gpfinialg{S}_{2}$ et $\theta_{1}=\overline{g}\theta_{2}$ donc $\theta_{1} \overset{id}{=} g\theta_{2}$. On a bien $(\sigma_{1},\gpfinialg{S}_{1},\theta_{1})\sim(\sigma_{2},\gpfinialg{S}_{2},\theta_{2})$.
	\end{enumerate}
\end{proof}

\begin{Pro}
\label{proBijMinSyst}
L'application $\pairesSt \to \ensSystConjc$ induit une bijection $\pairesStmin/{\sim_{\G}} \tosim \Smin$.
\end{Pro}

\begin{proof}
Les propositions \ref{proInclusSmin} et \ref{proEqTripletse} montrent que l'on a le diagramme commutatif
\[ \xymatrix{
		\pairesStmin/{\sim_{\G}}  \ar@{->}[r] \ar@{->}[d]^-{\sim}& \Smin \ar@{<-}[d]^-{\sim}\\
		\tripletse/{\sim}  \ar@{->}[r]^-{\sim} &  \Jss/{\simJss}}\]
d'où le résultat.
\end{proof}

\subsection{Relations d'équivalence sur \texorpdfstring{$\pairesSt$}{CTss,Lambda}}

\label{secreleqpairesst}

Nous avons obtenu dans les sections précédentes une application $\pairesSt/{\sim_{\G}} \to \Smin$ ainsi que $\pairesStmin$, un sous-ensemble de $\pairesSt$, tel que $\pairesStmin/{\sim_{\G}} \tosim \Smin$. Cependant, lorsque les paires ne sont pas minimales, nous ne savons pas encore quand elles donnent un même système cohérent. Nous cherchons donc une relation d'équivalence sur $\pairesSt$ qui le caractériserait. Le but de cette partie est d'introduire les relations d'équivalences $\sim_{\infty}$, $\sim_r$ et $\sim_e$ sur $\pairesSt$.  Nous verrons par la suite que $\sim_e$ est la relation d'équivalence recherchée qui fournira une bijection $\pairesSt/{\sim_e} \tosim \Smin$. Les deux autres relations  $\sim_{\infty}$ et $\sim_r$ sont tout de même intéressantes et produiront également des décompositions remarquables de $\rep[\ld][0]{\G}$.

\sautintro

Soit $m\in \mathbb{N}^{*}$, alors l'application
\[\begin{array}{ccccc}
\Tr_{\fr^m/\fr} & : &  X^{*}(\textbf{S}) & \to & X^{*}(\textbf{S}) \\
& & \lambda & \mapsto & \lambda+\fr_{\mathbf{S}}(\lambda)+ \cdots + \fr_{\mathbf{S}}^{m-1}(\lambda) \\
\end{array}\]
induit une application, que l'on note encore $\Tr_{\fr^m/\fr}$ :
\[\Tr_{\fr^m/\fr}  :   X^{*}(\textbf{S})/(\fr_{\mathbf{S}} -1)X^{*}(\textbf{S})  \longrightarrow  X^{*}(\textbf{S})/(\fr_{\mathbf{S}}^m -1)X^{*}(\textbf{S}).\]
On note alors pour $\theta \in X^{*}(\textbf{S})/(\fr_{\mathbf{S}} -1)X^{*}(\textbf{S})$, $\theta \langle m \rangle := \Tr_{\fr^m/\fr}(\theta)$.

\begin{Lem}
\label{lemTrace}
L'application $\Tr_{\fr^m/\fr}$ induit une bijection 
\[\Tr_{\fr^m/\fr}: X^{*}(\textbf{S})/(\fr_{\mathbf{S}} -1)X^{*}(\textbf{S})  \overset{\sim}{\longrightarrow}  \left[X^{*}(\textbf{S})/(\fr_{\mathbf{S}}^m -1)X^{*}(\textbf{S})\right]^{\fr_{\gpalg{S}}}.\]
\end{Lem}

\begin{proof}
Commençons par l'injectivité. Notons que $\Tr_{\fr^m/\fr} = (\fr_{\gpalg{S}}^m-1)/(\fr_{\gpalg{S}}-1)$. Comme pour tout $m \in \mathbb{N}^{*}$, $\fr_{\gpalg{S}}^m-1$ est injectif, l'application $\Tr_{\fr^m/\fr} : X^{*}(\textbf{S})  \to  X^{*}(\textbf{S})$ est injective. Mais alors s'il existe $\theta \in X^{*}(\textbf{S})$ et $\lambda \in X^{*}(\textbf{S})$ tels que $\Tr_{\fr^m/\fr}(\lambda)=(\fr_{\gpalg{S}}^m-1)(\theta)$ alors $\Tr_{\fr^m/\fr}(\lambda)=\Tr_{\fr^m/\fr}((\fr_{\gpalg{S}}-1)(\theta))$ et donc $\lambda=(\fr_{\gpalg{S}}-1)(\theta)$.

Passons maintenant à la surjectivité. Nous avons déjà que $\Tr_{\fr^m/\fr}$ est à valeurs dans $\left[X^{*}(\textbf{S})/(\fr_{\mathbf{S}}^m -1)X^{*}(\textbf{S})\right]^{\fr_{\gpalg{S}}}$. Prenons donc $\lambda \in X^{*}(\textbf{S})$ tel qu'il existe $\theta$ vérifiant $(\fr_{\gpalg{S}}-1)(\lambda)=(\fr_{\gpalg{S}}^m-1)(\theta)$. Alors $(\fr_{\gpalg{S}}-1)(\lambda)=(\fr_{\gpalg{S}}-1)(\Tr_{\fr^m/\fr}(\theta))$ et par injectivité de $(\fr_{\gpalg{S}}-1$), $\lambda = \Tr_{\fr^m/\fr}(\theta)$.
\end{proof}

\medskip

Soient $\mathbf{S}_{1}$ et $\mathbf{S}_{2}$ deux tores non-ramifiés et $g\in \G^{nr}$. Pour simplifier les écritures nous noterons $\gpalg{S}_{1}^{\fr^m}$ pour $(S_{1}^{nr})^{\fr^m}=\gpalg{S}_{1}(\kk_m)$, où $\kk_m$ est l'extension non-ramifiée de degré $m$ de $\kk$. Rappelons (voir annexe \ref{secIsoTores}) que $X^{*}(\textbf{S}_{1})/(\fr_{\mathbf{S}_{1}}^m -1)X^{*}(\textbf{S}_{1})$ est en bijection avec $\Irr({}^{0}(\mathbf{S}_{1}^{\fr^m})/(\mathbf{S}_{1}^{\fr^m})^{+})$ (notons que l'on utilise ici le système compatible de racines de l'unité fixé au début de cet article). Ainsi si ${}^{g} (\mathbf{S}_{1}^{\fr^{m}})=\mathbf{S}_{2}^{\fr^{m}}$ (ce qui est équivalent à dire que $\Ad(g)$ est défini sur $\kk_m$),  $\Ad(g)$ induit une bijection $\Irr({}^{0}(\mathbf{S}_{1}^{\fr^m})/(\mathbf{S}_{1}^{\fr^m})^{+}) \tosim \Irr({}^{0}(\mathbf{S}_{2}^{\fr^m})/(\mathbf{S}_{2}^{\fr^m})^{+})$ et donc une bijection $\Ad(g):X^{*}(\textbf{S}_{1})/(\fr_{\mathbf{S}_{1}}^m -1)X^{*}(\textbf{S}_{1}) \tosim X^{*}(\textbf{S}_{2})/(\fr_{\mathbf{S}_{2}}^m -1)X^{*}(\textbf{S}_{2})$. On définit alors la relation d'équivalence suivante

\begin{Def}
	Soit $m \in \mathbb{N}^{*}$. On dit que $(\mathbf{S}_{1},\theta_{1}) \sim_{m} (\mathbf{S}_{2},\theta_{2})$ si et seulement s'il existe $g\in \G^{nr}$ tel que
	\begin{enumerate}
		\item ${}^{g} (\mathbf{S}_{1}^{\fr^{m}})=\mathbf{S}_{2}^{\fr^{m}}$
		\item $g \theta_{1}\langle m \rangle = \theta_{2}\langle m \rangle$
	\end{enumerate}
	
	On dit que $(\mathbf{S}_{1},\theta_{1}) \sim_{\infty} (\mathbf{S}_{2},\theta_{2})$ si et seulement s'il existe un $m \in \mathbb{N}^{*}$ tel que $(\mathbf{S}_{1},\theta_{1}) \sim_{m} (\mathbf{S}_{2},\theta_{2})$.
\end{Def}

Nous pouvons remarquer qu'il existe un entier $d$ tel que pour tout tore non-ramifié $\mathbf{S}$ de $\G$, $\fr_{\mathbf{S}}^d=\psi^{d}$ est la multiplication par $q^{d}$. Ainsi $\sim_{\infty}=\sim_{d}$.

\bigskip

On note  $N^{nr}(\mathbf{S})$ le normalisateur de $S^{nr}$ dans $\G^{nr}$. Le groupe de Weyl étendu de $S^{nr}$ dans $\G^{nr}$ est le quotient $\We[S]:=N^{nr}(\mathbf{S})/{}^{0}S^{nr}$, où ${}^{0}S^{nr}$ est le sous-groupe borné maximal de $S^{nr}$. On note $\Wa[S]$, le groupe de Weyl affine, qui est le sous-groupe de $\We[S]$ engendré par les réflexions des murs des chambres de $\mathcal{A}(\mathbf{S},\knr)$. Le groupe de Weyl $\Wf[S]$ est défini par $\Wf[S]:=N^{nr}(\mathbf{S})/S^{nr}$.

Notons $\NSte*:=\{n \in N^{nr}(\mathbf{S}), {}^{n} (\mathbf{S}^{\fr^{m}})=\mathbf{S}^{\fr^{m}} \text{ et } n \theta\langle m \rangle = \theta\langle m \rangle \text{ pour un } m \in \mathbb{N}^{*}\}$ qui est un sous-groupe  de $N^{nr}(\mathbf{S})$. Définissons aussi $\WSte*:=\NSte*/{}^0S^{nr} \leq \We[S]$ et $\WStf* := \NSte*/S^{nr} \leq \Wf[S]$.

Considérons un $d$ tel que $\sim_{\infty}=\sim_{d}$. En particulier on a $(\fr_{\mathbf{S}}^{d} -1)X^{*}(\textbf{S}) = (q^{d} -1)X^{*}(\textbf{S})$. Alors l'application naturelle $\langle\cdot,\cdot\rangle:X_{*}(\textbf{S}) \times X^{*}(\textbf{S}) \rightarrow \mathbb{Z}$ induit une application
\[\langle\cdot,\cdot\rangle:X_{*}(\textbf{S}) \times X^{*}(\textbf{S})/(\fr_{\mathbf{S}}^{d} -1)X^{*}(\textbf{S})  \rightarrow \mathbb{Z}/(q^d-1)\mathbb{Z}.\]

Posons $\WStf$ le sous-groupe de $\WStf*$ engendré par les $s_{\alpha}$ où $\alpha$ est une co-racine vérifiant $\langle \alpha, \theta \langle d \rangle \rangle = 0$. Les $\alpha$ tels que $\langle \alpha, \theta \langle d \rangle \rangle = 0$ forment un système de racines $\fr$-stable. Une vérification directe est possible, mais on peut également montrer que $\WStf$ s'identifie au groupe de Weyl de $C_{\gpfinialg*{\G}}(t)^{\circ}$, où $t\in \gpfini*{\G}$, ce que nous ferons plus tard (lemme \ref{lemIdentifWeyl}). Notons $\mathcal{A}_{(\mathbf{S},\theta)}$ l'appartement $\mathcal{A}(\mathbf{S},\knr)$ mais dont la structure polysimpliciale est déduite de celle de $\mathcal{A}(\mathbf{S},\knr)$ en ne gardant que les murs correspondants aux co-racines affines dont la partie vectorielle correspond aux $\alpha$ tels que $\langle \alpha, \theta \langle d \rangle \rangle = 0$. Désignons $\WSta$ le sous-groupe de $\Wa$ engendré par les réflexions des murs de $\mathcal{A}_{(\mathbf{S},\theta)}$. Ce dernier est alors un groupe de Weyl affine comme défini dans \cite{bourbaki} chapitre VI, §2.

Enfin définissons $\NSte$ (resp. $\WSte$) le sous-groupe de $\NSte*$ (resp. $\WSte*$) image réciproque de $\WStf$ par l'application $\NSte* \twoheadrightarrow \WStf*$ (resp. $\WSte* \twoheadrightarrow \WStf*$). Notons que $\WSta$ est un sous-groupe de $\WSte$ et posons $\NSta$ l'image réciproque de $\WSta$ dans $\NSte$.

Notons que tous les ensembles définis précédemment sont $\fr$-stables.

\medskip

Notons que si  $(\mathbf{S}_{1},\theta_{1})$ et $(\mathbf{S}_{2},\theta_{2})$ sont $\sim_{\infty}$-équivalentes, alors il existe $g$ tel que ${}^{g} (\mathbf{S}_{1}^{\fr^{m}})=\mathbf{S}_{2}^{\fr^{m}}$ et $g \theta_{1}\langle m \rangle = \theta_{2}\langle m \rangle$ pour un certain $m$. Comme ${}^{g} S_{1}^{nr}=S_{2}^{nr}$ (par densité de Zariski de $\gpalg{S}_{1}^{\fr^m}$), $g^{-1}\fr(g) \in  N^{nr}(\mathbf{S}_1)$. De plus $\theta_{1}\langle m \rangle$ et $\theta_{2}\langle m \rangle$ sont $\fr$-stables donc $g^{-1}\fr(g)  \theta_{1}\langle m \rangle = \theta_{1}\langle m \rangle$. On en déduit que $g^{-1}\fr(g) \in \NStecomplet*{\mathbf{S}_1}{\theta_1}$. On peut alors raffiner $\sim_{\infty}$ de la façon suivante.

\begin{Def}
	On dit que $(\mathbf{S}_{1},\theta_{1}) \sim_{r} (\mathbf{S}_{2},\theta_{2})$ (resp. $(\mathbf{S}_{1},\theta_{1}) \sim_{e} (\mathbf{S}_{2},\theta_{2})$) si et seulement s'il existe $m\in \mathbb{N}^{*}$ et $g \in \G^{nr}$ tels que
	\begin{enumerate}
		\item ${}^{g} (\mathbf{S}_{1}^{\fr^{m}})=\mathbf{S}_{2}^{\fr^{m}}$
		\item $g \theta_{1}\langle m \rangle = \theta_{2}\langle m \rangle$
		\item $g^{-1}\fr(g) \in \NStecomplet{\gpalg{S}_1}{\theta_1}$ (resp. $g^{-1}\fr(g) \in \NStacomplet{\gpalg{S}_1}{\theta_1}$)
	\end{enumerate}
\end{Def}
Expliquons les notations précédentes. Nous notons $\sim_r$ car cette relation d'équivalence va correspondre à des classes de conjugaison rationnelles. Et $\sim_e$ est notée ainsi car cette relation va caractériser le fait que deux paires $(\gpalg{S},\theta)$ définissent le même système d'idempotents cohérent $e=(e_x)$.

	Il faut vérifier que $\sim_{r}$ et $\sim_{e}$ sont bien des relations d'équivalence. Cela découle de la remarque suivante. Soit $g \in \G^{nr}$ tel que ${}^{g} (\mathbf{S}_{1}^{\fr^{m}})=\mathbf{S}_{2}^{\fr^{m}}$ et $g \theta_{1}\langle m \rangle = \theta_{2}\langle m \rangle$. Alors $Ad(g)$ induit une bijection $\Ad(g):\NStecomplet*{\gpalg{S}_1}{\theta_1} \overset{\sim}{\rightarrow} \NStecomplet*{\gpalg{S}_2}{\theta_2}$. Si $\alpha$ est une co-racine telle que $\langle \alpha, \theta_1 \langle d \rangle \rangle = 0$ alors $\langle g \cdot \alpha, g \cdot \theta_1 \langle d \rangle \rangle = 0$ donc $\Ad(g):\NStecomplet{\gpalg{S}_1}{\theta_1} \overset{\sim}{\rightarrow} \NStecomplet{\gpalg{S}_2}{\theta_2}$. De plus $\Ad(g)$ envoie les murs de l'appartement relatif à $\mathbf{S}_{1}$ sur les murs de l'appartement relatif à $\mathbf{S}_{2}$ donc $\Ad(g):\NStacomplet{\gpalg{S}_1}{\theta_1} \overset{\sim}{\rightarrow} \NStacomplet{\gpalg{S}_2}{\theta_2}$. Pour conclure, il suffit de remarquer que l'application $x \mapsto g x \fr(g)^{-1} = (gxg^{-1})(g\fr(g)^{-1})$ induit bien une bijection $\NStecomplet{\gpalg{S}_1}{\theta_1} \overset{\sim}{\rightarrow} \NStecomplet{\gpalg{S}_2}{\theta_2}$ (resp. $\NStacomplet{\gpalg{S}_1}{\theta_1} \overset{\sim}{\rightarrow} \NStacomplet{\gpalg{S}_2}{\theta_2}$) si $g\fr(g)^{-1} \in \NStecomplet{\gpalg{S}_2}{\theta_2}$ (resp. $g\fr(g)^{-1} \in \NStacomplet{\gpalg{S}_2}{\theta_2}$).

\medskip

On a alors les liens suivants :
\[ \sim_{\G} \Rightarrow \sim_{e} \Rightarrow \sim_{r} \Rightarrow \sim_{\infty}\]

\subsection{Caractérisation des systèmes de classes de conjugaison \texorpdfstring{$0$}{0}-cohérents minimaux dans \texorpdfstring{$\pairesSt$}{CTss,Lambda}}

Nous souhaitons montrer dans cette section que l'application $\pairesSt \to \Smin$ de la section \ref{secLienIeJss} induit une bijection $\pairesSt/{\sim_e} \tosim \Smin$.

\sautintro

Considérons une paire $(\mathbf{S},\theta)$ et $\sigma \in \mathcal{A}(\mathbf{S},\knr)^\fr$. Notons $(\sigma,\gpfinialg{S},\theta) \in \triplets$ le triplet induit par $(\mathbf{S},\theta)$ comme dans la section \ref{secintroStheta}. Le groupe des caractères d'un $\res$-tore est aussi muni d'applications  $\Tr_{\fr^m/\fr}$, et celles-ci sont compatibles avec l'identification naturelle $X^*(\mathbf{S}) \simeq X^*(\gpfinialg{S})$. Nous rappelons que $\Wx{\sigma}[S]$ est le sous-groupe de $\We[S]$ engendré par les réflexions des murs passant par $\sigma$, qui s'identifie au groupe de Weyl de $\quotredalg{\G}{\sigma}$ relatif à $\gpfinialg{S}$. On notera $\Wxd{\sigma}[\gpalg{S}]$ le stabilisateur de $\sigma$ dans $\We[S]$. Soit $d$ un entier tel que $\sim_{\infty}=\sim_{d}$. On note alors $\WStx{\sigma}$ le sous-groupe de $\Wx{\sigma}[S]$ engendré par les $s_{\alpha}$ tels que $s_{\alpha} \theta\langle d \rangle = \theta\langle d \rangle$  et $\langle \alpha, \theta \langle d \rangle \rangle = 0$. Le lemme suivant découle de \cite{bourbaki}, chapitre V, §3, proposition 1.

\begin{Lem}
	\label{lemWeylts}
	
	Soient $(\mathbf{S},\theta) \in \pairesSt$ et $\sigma \in \mathcal{A}(\mathbf{S},\knr)$. Alors $\WSta \cap \Wxd{\sigma}[\gpalg{S}] = \WStx{\sigma}$.
\end{Lem}

Rappelons que la classe de $\quotred{\G}{\sigma}$-conjugaison de la paire $(\gpfinialg{S},\theta)$ correspond à une classe de $\quotred*{\G}{\sigma}$-conjugaison de paires $(\gpfinialg*{S},t)$ (proposition \ref{proBijPairesSt}). Fixons $(\gpfinialg*{S},t)$ un représentant de la classe associée à $(\gpfinialg{S},\theta)$. Le groupe de Weyl de $\quotredalg*{\G}{\sigma}$ relatif à $\gpfinialg*{S}$ s'identifie au groupe de Weyl $\Wx{\sigma}[S]$ de $\quotredalg{\G}{\sigma}$ relatif à $\gpfinialg{S}$.

\begin{Lem}
	\label{lemIdentifWeyl}
	
	L'isomorphisme entre $\Wx{\sigma}[S]$ et $\Wf(\gpfinialg*{S},\quotredalg*{\G}{\sigma})$ envoie $\WStx{\sigma}$ sur $\Wf(\gpfinialg*{S},C_{\quotredalg*{\G}{\sigma}}(t)^{\circ})$.
\end{Lem}

\begin{proof}
 D'après \cite{carter} théorème 3.5.4, $\Wf(\gpfinialg*{S},C_{\quotredalg*{\G}{\sigma}}(t)^{\circ})$ correspond au sous-groupe de $\Wx{\sigma}[S]$ engendré par les $s_{\alpha}$ où $\alpha$ est une racine (pour $\gpfinialg*{S}$) vérifiant $\alpha(t)=1$.

	La bijection reliant $\theta$ et $t$ est donnée par la formule $\theta(N_{\fr^d/\fr}(\alpha(\zeta))=\kappa(\alpha(t))$ (ici $\zeta$ est une racine de l'unité et on a identifié les caractères de $\gpfinialg*{S}$ avec les co-caractères de $\gpfinialg{S}$, voir annexe \ref{secIsoTores} pour plus de détails sur la formule et les notations). Or $\theta(N_{\fr^d/\fr}(x))=\Tr_{\fr^d/\fr}(\theta)(x).$ On a alors $\alpha(t)=1$ si et seulement si $\langle \alpha, \theta \langle d \rangle \rangle =0$ d'où le résultat.
\end{proof}

Soit $(\sigma,\gpfinialg{S},\theta) \in \triplets$. Nous savons que l'on peut associer à ce triplet une classe de conjugaison rationnelle semi-simple $t \in \ss*{\quotred*{\G}{\sigma}}$ (voir section \ref{secintroStheta}). On dira alors que deux triplets $(\sigma,\gpfinialg{S}_{1},\theta_{1}),(\sigma,\gpfinialg{S}_{2},\theta_{2}) \in \triplets$ sont rationnellement équivalents si la classe de conjugaison rationnelle associée est la même.

\begin{Lem}
	\label{lemEquivRatio}
	Soient $\sigma \in \bt$ et $(\sigma,\gpfinialg{S}_{1},\theta_{1}),(\sigma,\gpfinialg{S}_{2},\theta_{2}) \in \triplets$. Alors les conditions suivantes sont équivalentes
	\begin{enumerate}
		\item $(\sigma,\gpfinialg{S}_{1},\theta_{1})$ et $(\sigma,\gpfinialg{S}_{2},\theta_{2})$ sont rationnellement équivalents.
		\item Il existe $g \in \quotredalg{\G}{\sigma}$  et $m \in \mathbb{N}^{*}$ tels que $g(\gpfinialg{S}_{1}^{\fr^m})=\gpfinialg{S}_{2}^{\fr^m}$, $g\theta_{1}\langle m \rangle=\theta_{2}\langle m \rangle$ et $g^{-1}\fr(g)$ se projette sur $\WStxcomplet{\sigma}{\gpalg{S}_1}{\theta_1}$.
	\end{enumerate}
\end{Lem}

\begin{proof}
	Notons $(\gpfinialg*{S}_{1},t_{1})$ et $(\gpfinialg*{S}_{2},t_{2})$ des paires associées à $(\gpfinialg{S}_{1},\theta_{1})$ et $(\gpfinialg{S}_{2},\theta_{2})$.
	
	Supposons 1. Les triplets $(\sigma,\gpfinialg{S}_{1},\theta_{1})$ et $(\sigma,\gpfinialg{S}_{2},\theta_{2})$ étant rationnellement équi\-valents, on a par définition qu'il existe $g^{*}\in (\quotredalg*{\G}{\sigma})^{\fr}$ tel que $t_{1}=\Ad(g^{*})t_{2}$. Les deux tores $\gpfinialg{S}_{1}^{*}$ et ${}^{g^{*}}\gpfinialg{S}_{2}^{*}$ contiennent $t_{1}$ donc sont conjugués sous $C_{\quotredalg*{\G}{\sigma}}(t_{1})^{\circ}$. Il existe donc un $h^{*} \in \quotredalg*{\G}{\sigma}$ tel que $h^{*}(\gpfinialg*{S}_{1},t_{1}) = (\gpfinialg*{S}_{2},t_{2})$ et $w=(h^{*})^{-1}\fr(h^{*}) \in \Ws{t_{1}}[\gpfinialg{S}_{1}]$. Rappelons que l'on peut voir la correspondance entre les paires $(\gpfinialg{S},\theta)$ et les paires $(\gpfinialg*{S},t)$ en fixant une paire de tores de référence en dualité sur $\res$ (comme rappelé dans le paragraphe après la proposition \ref{proBijPairesSt}). Fixons ici, $\gpfinialg{S}_1$ et $\gpfinialg*{S}_1$. La paire $(\gpfinialg*{S}_{2},t_{2})$ correspond alors à $(w,t_1)$ où $t_1 \in (\gpfinialg{S}_{1}^{*})^{\fr_w}$. Considérons $\tilde{\theta}_{1} \in X^{*}(\gpfinialg{S}_{1})/(w\fr -1)X^{*}(\gpfinialg{S}_{1})$ l'élément qui correspond à $(w,t_1)$. Prenons également $g \in \quotredalg{G}{\sigma}$ tel que $g^{-1}\fr(g)$ relève $w$. Alors la paire $g(\gpfinialg{S}_{1},\tilde{\theta}_{1})$ est une paire en dualité avec $(\gpfinialg*{S}_{2},t_{2})$ et est donc $\G$-conjuguée à $(\gpfinialg{S}_{2},\theta_2)$. Le lemme \ref{lemIdentifWeyl} nous permet d'identifier $w$ à un élément de $\WStxcomplet{\sigma}{\gpalg{S}_1}{\theta_1}$. Enfin, pour un $m$ tel que $(w\fr)^m=\fr^m$, en remarquant que $\tilde{\theta}_1$ et $\theta_1$ correspondent tous les deux à $t_1$, on a que $\tilde{\theta}_{1} \langle m \rangle = \theta_1 \langle m \rangle$, d'où 2.

	Supposons dorénavant 2. Nommons $w \in \WStxcomplet{\sigma}{\gpalg{S}_1}{\theta_1}$ la projection de $g^{-1}\fr(g)$ qui s'identifie à $w \in \Ws{t_{1}}[\gpfinialg{S}_{1}]$ par le lemme \ref{lemIdentifWeyl}. Posons $(\gpfinialg{S}_{1},\tilde{\theta}_1)=g^{-1}(\gpfinialg{S}_{2},\theta_2)$, où $\tilde{\theta}_1 \in X^{*}(\gpfinialg{S}_{1})/(w\fr -1)X^{*}(\gpfinialg{S}_{1})$. Alors $\tilde{\theta}_{1} \langle m \rangle = \theta_1 \langle m \rangle$ et donc $\tilde{\theta}_1$ s'envoie sur $t_1 \in (\gpfinialg{S}_{1}^{*})^{w\fr}$. Prenons alors $g^{*} \in C_{\quotredalg*{\G}{\sigma}}(t_{1})^{\circ}$ tel que $(g^{*})^{-1}\fr(g^{*})$ relève $w$. La paire $(\gpfinialg*{S}_{2},t_{2})$ est alors $\gpfini*{\G}$ conjuguée à $g^{*}(\gpfinialg*{S}_{1},t_{1})$. Et comme $\Ad(g^{*})t_{1}=t_1$, $t_{1}$ et $t_{2}$ sont rationnellement conjugués et on a 1.
\end{proof}

\begin{Lem}
	\label{lemEqRatio}
	Soient $(\mathbf{S}_{1},\theta_{1})$ et $(\mathbf{S}_{2},\theta_{2})$ deux éléments de $\pairesSt$. Supposons qu'il existe $\sigma \in \bt \cap \mathcal{A}(\mathbf{S}_{1},\knr)\cap \mathcal{A}(\mathbf{S}_{2},\knr)$. Notons $(\sigma,\gpfinialg{S}_{1},\theta_{1})$ et $(\sigma,\gpfinialg{S}_{2},\theta_{2})$ les triplets induits sur $\quotredalg{\G}{\sigma}$. Alors les assertions suivantes sont équivalentes
	\begin{enumerate}
		\item $(\sigma,\gpfinialg{S}_{1},\theta_{1})$ et $(\sigma,\gpfinialg{S}_{2},\theta_{2})$ sont rationnellement équivalents.
		\item Il existe $g_{\sigma} \in \paranr{\G}{\sigma}$ et $m\in \mathbb{N}^{*}$ tel que ${}^{g_{\sigma}} (\mathbf{S}_{1}^{\fr^{m}})=\mathbf{S}_{2}^{\fr^{m}}$, $g_{\sigma} \theta_{1}\langle m \rangle = \theta_{2}\langle m \rangle$ et $g_{\sigma}^{-1}\fr(g_\sigma) \in \NStacomplet{\gpalg{S}_1}{\theta_1}$.
		
	\end{enumerate}
\end{Lem}

\begin{proof}

	Par le lemme \ref{lemEquivRatio}, 1. est équivalent à l'existence d'un $\bar{g}_{\sigma} \in \quotredalg{\G}{\sigma}$ et d'un $m$ tels que ${}^{\bar{g}_{\sigma}}(\gpfinialg{S}_{1}^{\fr^{m}})=\gpfinialg{S}_{2}^{\fr^{m}}$, $\bar{g}_{\sigma}\theta_{1}\langle m \rangle=\theta_{2}\langle m \rangle$ et $\bar{g}_{\sigma}^{-1}\fr(\bar{g}_{\sigma}) \in \WStxcomplet{\sigma}{\gpalg{S}_1}{\theta_1}$.
	
	Supposons l'existence d'un tel $\bar{g}_{\sigma}$. On peut alors relever celui-ci en un élément $g_{\sigma} \in \paranr{\G}{\sigma}$ tel que ${}^{g_{\sigma}} (\mathbf{S}_{1}^{\fr^{m}})=\mathbf{S}_{2}^{\fr^{m}}$ et $g_{\sigma} \theta_{1}\langle m \rangle = \theta_{2}\langle m \rangle$ (voir le lemme 2.2.2 de \cite{debacker} pour les tores). Par le lemme \ref{lemWeylts} $w_{\sigma}=g_{\sigma}^{-1}\fr(g_\sigma) \in \WStxcomplet{\sigma}{\gpalg{S}_1}{\theta_1} \subseteq \WStacomplet{\gpalg{S}_1}{\theta_1}$.
	
	Supposons maintenant 2. Notons $\bar{g}_{\sigma}$ la réduction de $g_{\sigma}$ dans $\quotredalg{\G}{\sigma}$. Alors ${}^{\bar{g}_{\sigma}}\gpfinialg{S}_{1}=\gpfinialg{S}_{2}$ et $\bar{g}_{\sigma}\theta_{1}\langle m \rangle=\theta_{2}\langle m \rangle$. Et par le lemme \ref{lemWeylts} $w_{\sigma}$, la projection de $g_{\sigma}^{-1}\fr(g_\sigma)$, vérifie $w_{\sigma} \in \WStacomplet{\gpalg{S}_1}{\theta_1} \cap \Wxd{\sigma}[\gpalg{S}_1] = \WStxcomplet{\sigma}{\gpalg{S}_1}{\theta_1}$ d'où 1.
\end{proof}

\begin{Lem}
	\label{lemSommetStable}
	Soit $\mathbf{S}$ un tore maximal non-ramifié de $\gpalg{\G}$, alors il existe $x \in \bts \cap \mathcal{A}(\mathbf{S},\knr)$.
\end{Lem}

\begin{proof}
	D'après \cite{debacker} Lemme 2.2.1, $\mathcal{A}(\mathbf{S},\knr)^{\fr}$ est non-vide et est l'union des facettes de $\bt$ qui le rencontrent.
\end{proof}

\begin{Lem}
	\label{lemxcommun}
	Soient $(\mathbf{S}_{1},\theta_{1})$ et $(\mathbf{S}_{2},\theta_{2})$ deux éléments de $\pairesSt$ avec $(\mathbf{S}_{1},\theta_{1})$ minimale (définition \ref{defpairesmin}). Supposons que $(\mathbf{S}_{1},\theta_{1})$ et $(\mathbf{S}_{2},\theta_{2})$ définissent le même système de classes de conjugaison 0-cohérent. Alors, quitte à conjuguer l'une des paires par un élément de $\G$, il existe un sommet $x \in \bts$ vérifiant :
	\begin{enumerate}
		\item $x \in \mathcal{A}(\mathbf{S}_{1},\knr) \cap \mathcal{A}(\mathbf{S}_{2},\knr)$
		\item $(x,\gpfinialg{S}_{1},\theta_{1})$ et $(x,\gpfinialg{S}_{2},\theta_{2})$, les éléments de $\triplets$ correspondant respectivement à $(\mathbf{S}_{1},\theta_{1})$ et $(\mathbf{S}_{2},\theta_{2})$, sont rationnellement équivalents.
	\end{enumerate}
\end{Lem}

\begin{proof}
	Prenons $x \in \bts \cap \mathcal{A}(\mathbf{S}_{2},\knr)$ par le lemme \ref{lemSommetStable}, et considérons $(x,\gpfinialg{S}_{2},\theta_{2}) \in \triplets$ correspondant  à $(\mathbf{S}_{2},\theta_{2})$. Notons $t$ la classe de conjugaison semi-simple dans $(\quotredalg*{\G}{x})^{\fr}$ associée à $(x,\gpfinialg{S}_{2},\theta_{2})$. Soit $\gpfinialg*{S}$ un tore maximalement déployé du centralisateur de $t$ de sorte qu'il existe un triplet minimal $(x,\gpfinialg{S},\theta)$ associé à $t$. Choisissons $(\textbf{S},\theta) \in \pairesSt$ correspondant à $(x,\gpfinialg{S},\theta)$. Les triplets $(x,\gpfinialg{S},\theta)$ et $(x,\gpfinialg{S}_{2},\theta_{2})$ sont rationnellement équivalents par construction, donc $(\mathbf{S}_{2},\theta_{2})$ et $(\mathbf{S},\theta)$ définissent le même système de classes de conjugaison 0-cohérent. Ainsi, par hypothèse, $(\mathbf{S}_{1},\theta_{1})$ et $(\mathbf{S},\theta)$ définissent également le même système. Or ces deux paires sont minimales, donc d'après la proposition \ref{proBijMinSyst}, $(\mathbf{S}_{1},\theta_{1})$ est $\G$-conjuguée à $(\mathbf{S},\theta)$ ce qui démontre le résultat.
\end{proof}

\begin{Lem}
	\label{lemMemeSystemeMemePI}
	Soient $(\mathbf{S}_{1},\theta_{1}),(\mathbf{S}_{2},\theta_{2}) \in \pairesSt$. Alors si $(\mathbf{S}_{1},\theta_{1})$ et $(\mathbf{S}_{2},\theta_{2})$ définissent le même système de classes de conjugaison 0-cohérent, $(\mathbf{S}_{1},\theta_{1}) \sim_{e} (\mathbf{S}_{2},\theta_{2})$.
\end{Lem}

\begin{proof}
	Tous les systèmes de classes de conjugaison 0-cohérents sont décrits par une paire minimale par le théorème \ref{proBijMinSyst}. Ainsi on peut supposer, sans perte de généralité, que $(\mathbf{S}_{1},\theta_{1})$ est minimale. Le résultat découle alors des lemmes \ref{lemEqRatio} et \ref{lemxcommun}.
\end{proof}

Si $n \in N^{nr}(\mathbf{S})$ alors $n$ induit une action sur $X^{*}(\mathbf{S})$ et on note $\fr_{n,\mathbf{S}}:=n\vartheta_{\mathbf{S}} \circ \psi$. Notons que $\fr_{n,\mathbf{S}}$ ne dépend que de l'image de $n$ dans $\Wf[S]$.

\begin{Lem}
	\label{lemgStheta}
	Soient $(\mathbf{S},\theta) \in \pairesSt$ et $g\in \G^{nr}$ tels que $g^{-1}\fr(g) \in \NSte$. Alors il existe $(\mathbf{S}',\theta') \in \pairesSt$ et $m \in \mathbb{N}^{*}$ tels que ${}^{g} (\mathbf{S}^{\fr^{m}})=\mathbf{S}'^{\fr^{m}}$ et $g \theta\langle m \rangle = \theta'\langle m \rangle$.
\end{Lem}

\begin{proof}
	Posons $n:=g^{-1}\fr(g) \in \NSte$ et $w\in \WSte$ sa réduction. Soit $m$ tel que $n \theta\langle m \rangle = \theta\langle m \rangle$ et $\fr_{\mathbf{S}}^m=\fr_{n,\mathbf{S}}^{m}$. Alors $\theta\langle m \rangle \in \left[X^{*}(\textbf{S})/(\fr_{n,\mathbf{S}}^m -1)X^{*}(\textbf{S})\right]^{\fr_{n,\mathbf{S}}}$ et par le lemme \ref{lemTrace} il existe $\tilde{\theta} \in X^{*}(\mathbf{S})/(\fr_{n,\mathbf{S}}-1)X^{*}(\mathbf{S})$ tel que $\tilde{\theta}\langle m \rangle = \theta \langle m \rangle$. Définissons $\mathbf{S}':={}^{g}\mathbf{S}$ et $\theta':=g\tilde{\theta} \in X^{*}(\textbf{S}')/(\fr_{\mathbf{S}'} -1)X^{*}(\textbf{S}')$ de sorte que $\theta'\langle m \rangle = g \theta \langle m \rangle$.
\end{proof}

\begin{Lem}
	\label{lemChambreStableSime}
	Dans chaque classe de $\sim_{e}$-équivalence il existe un représentant $(\mathbf{S},\theta)$ tel que l'appartement $\mathcal{A}_{(\mathbf{S},\theta)}$ contienne une chambre $\fr$-stable.
\end{Lem}

\begin{proof}
	Soit $(\mathbf{S},\theta) \in \pairesSt$. Prenons $C$ une chambre de $\mathcal{A}_{(\mathbf{S},\theta)}$. L'automorphisme $\fr$ agit de façon polysimpliciale sur l'appartement donc $\fr(C)$ est encore une chambre. Le groupe $\WSta$ agit transitivement sur l'ensemble des chambres (\cite{bourbaki} Chapitre VI, §2, Proposition 2), donc il existe $w \in \WSta$ tel que $w\fr(C)=C$.
	En particulier $w \in \Wa$ et la proposition \ref{proCommuteNW} montre alors qu'il existe $g \in \G^{nr}$ tel que $g^{-1}\fr(g) \in \NSta$ soit un relèvement de $w$. Prenons $(\mathbf{S}',\theta') \in \pairesSt$ comme dans le lemme \ref{lemgStheta}. On a bien, par définition, que $(\mathbf{S}',\theta') \sim_{e} (\mathbf{S},\theta)$. Posons $C':=gC$ une chambre de $\mathcal{A}_{(\mathbf{S}',\theta')}$. Comme $w\fr(C)=C$, c'est à dire $g^{-1}\fr(g)\fr(C)=C$, on obtient $\fr(C')=C'$ d'où le résultat.
\end{proof}

\begin{Lem}
	\label{lemPointfixew} Soient $(\mathbf{S},\theta) \in \pairesSt$ et $g\in \G^{nr}$. On suppose que $\mathcal{A}_{(\mathbf{S},\theta)}$ contient $C$ une chambre $\fr$-stable et que $g^{-1}\fr(g) \in \NSta$. Alors il existe $\sigma \in \bt \cap \mathcal{A}(\mathbf{S},\knr)$ et $h \in \NSta$ tels que $h^{-1}g^{-1}\fr(g)\fr(h)$ fixe $\sigma$.
\end{Lem}

\begin{proof}
	Notons $w \in \WSta$ la réduction de $g^{-1}\fr(g)$. Posons $\mathbf{S}':={}^{g}\mathbf{S}$. Par le lemme \ref{lemSommetStable} il existe $y$ un point $\fr$-stable de $\mathcal{A}(\mathbf{S}',\knr)$. Le point $z:=g^{-1}y$ est alors un point $w\fr$-stable de $\mathcal{A}_{(\mathbf{S},\theta)}$.

	Le groupe $\WSta$ agit transitivement sur les chambres de $\mathcal{A}_{(\mathbf{S},\theta)}$ donc il existe $u \in \WSta$ tel que $x := u \cdot z \in C$. De l'égalité $w\fr(z)=z$ on déduit que $w'\fr(x)=x$ avec $w'=uw\fr(u)^{-1}$. Posons $h \in \NSta$ un relèvement de $u$.

	Nommons $C'$ la chambre $w'\cdot C$. Ainsi $x=w'\fr(x)\in w'\fr(C)=w'C=C'$ et on a aussi $x \in C$ donc il existe $v \in \WStx{x}$ tel que $v\cdot C'=C$ ($\WStx{x}$ agit transitivement sur les chambres de $\mathcal{A}_{(\mathbf{S},\theta)}$ qui contiennent $x$, \cite{bourbaki} Chapitre VI, §2, Proposition 2 et \cite{bourbaki} Chapitre V, §3, Proposition 1). Alors $vw' \cdot C = C$ et comme $vw' \in \WSta$, $vw'=1$. De plus $x=w'\fr(x)$ donc $x=\fr(x)$. On prend alors $\sigma$ le plus petit polysimplexe de $\mathcal{A}(\mathbf{S},\knr)$ qui contient $x$. Comme $x=\fr(x)$ on a également $\sigma=\fr(\sigma)$ et $\sigma$ convient.
\end{proof}

\begin{Lem}
	\label{lemMemePiMemeSysteme}
	Soient $(\mathbf{S}_{1},\theta_{1})$ et $(\mathbf{S}_{2},\theta_{2})$ deux paires telles que $(\mathbf{S}_{1},\theta_{1}) \sim_{e} (\mathbf{S}_{2},\theta_{2})$. Alors $(\mathbf{S}_{1},\theta_{1})$ et $(\mathbf{S}_{2},\theta_{2})$ définissent le même système de classes de conjugaison 0-cohérent.
\end{Lem}

\begin{proof}
	D'après le lemme \ref{lemChambreStableSime} il existe une paire $(\mathbf{S},\theta)$ telle que $\mathcal{A}_{(\mathbf{S},\theta)}$ contienne une chambre $\fr$-stable et $(\mathbf{S}_{1},\theta_{1}) \sim_{e} (\mathbf{S}_{2},\theta_{2}) \sim_{e} (\mathbf{S},\theta)$. Il nous suffit donc de montrer par exemple que $(\mathbf{S}_{1},\theta_{1})$ et $(\mathbf{S},\theta)$ définissent le même système de classes de conjugaison 0-cohérent. Prenons $g\in \G^{nr}$ et $m\in\mathbb{N}^{*}$ tel que ${}^{g} (\mathbf{S}^{\fr^{m}})=\mathbf{S}_{1}^{\fr^{m}}$, $g \theta\langle m \rangle = \theta_{1}\langle m \rangle$ et $g^{-1}\fr(g) \in \NSta$. Par le lemme précédent \ref{lemPointfixew} on peut supposer qu'il existe $\sigma \in \bt \cap \mathcal{A}(\mathbf{S},\knr)$ tel que $g^{-1}\fr(g)$ fixe $\sigma$. Par le lemme 2.3.1 de \cite{DebackerReeder} nous savons que $H^{1}(\fr,\paranr{\G}{\sigma})=1$ donc il existe $g_{\sigma} \in \paranr{\G}{\sigma}$ tel que $g^{-1}\fr(g)=g_{\sigma}^{-1}\fr(g_{\sigma})$ et donc $g_{\sigma}g^{-1} \in \G$. Ainsi, quitte à remplacer $(\mathbf{S},\theta)$ par un de ses $\G$-conjugués on peut supposer que $g=g_{\sigma}$. On conclut par le lemme \ref{lemEqRatio}.
\end{proof}

On déduit alors du lemme \ref{lemMemeSystemeMemePI} et du lemme \ref{lemMemePiMemeSysteme} le théorème suivante

\begin{The}
\label{theBijSime}
L'application $\pairesSt \to \Smin$ induit une bijection 
\[\pairesSt/{\sim_e} \tosim \Smin.\]
On obtient donc une décomposition de $\rep[\ld][0]{G}$
\[ \rep[\ld][0]{G} = \prod_{[\gpalg{S},\theta]_e \in \pairesSt/{\sim_e}} \rep[\ld][[\gpalg{S},\theta]_e]{G},\]
où $\rep[\ld][[\gpalg{S},\theta]_e]{G} = \rep[\ld][(\gpalg{S},\theta)]{G}$, pour $(\gpalg{S},\theta)$ n'importe quel représentant de la classe de $\sim_e$-équivalence $[\gpalg{S},\theta]_e$.
\end{The}

Cette décomposition est minimale pour la méthode utilisée.

\begin{Rem}
\label{remdecompo}
Toute relation d'équivalence $\sim$ sur $\pairesSt$ plus faible que $\sim_e$ (en particulier $\sim_{\infty}$ et $\sim_r$) conduit à une décomposition de $\rep[\ld][0]{\G}$ : 
\[ \rep[\ld][0]{G} = \prod_{\mathcal{C} \in \pairesSt/{\sim}} \rep[\ld][\mathcal{C}]{G},\]
où $\rep[\ld][\mathcal{C}]{G}=\prod \rep[\ld][[\gpalg{S},\theta]_e]{G}$, le produit étant pris sur les classes de $\sim_e$-équivalence $[\gpalg{S},\theta]_e$ telles que $(\gpalg{S},\theta) \in \mathcal{C}$.
\end{Rem}

Notons également que par construction les objets simples de $\rep[\ld][(\gpalg{S},\theta)]{G}$ sont décrits par : $\pi \in \Irr_{\ld}(G) \cap \rep[\ld][(\gpalg{S},\theta)]{G}$ si et seulement s'il existe $(\gpalg{S}',\theta') \in \pairesSt$ et $\sigma \in \mathcal{A}(\gpalg{S}')^{\fr}$ tels que 
\begin{itemize}
\item $(\gpalg{S}',\theta') \sim_e (\gpalg{S},\theta)$
\item $\langle {}^{*}\mathcal{R}_{\gpfinialg{S}'}^{\quotredalg{G}{\sigma}}(\pi^{\radpara{G}{\sigma}}),\theta'\rangle_{\gpfini{S}'_{\ld}} \neq 0$
\end{itemize}
où $\gpfini{S}'_{\ld}$ est  $\gpfini{S}'$ si $\ld=\Ql$ ou le sous-groupe maximal de $\gpfini{S}'$ d'ordre premier à $\lprime$ si $\ld=\Zl$.
\section{La relation d'équivalence \texorpdfstring{$\sim_{r}$}{~r}}

\label{secSimr}

On suppose toujours que $\gpalg{\G}$ se déploie sur une extension non-ramifiée de $\kk$. Le but de cette section est d'étudier les décompositions produites par $\sim_r$ et $\sim_\infty$, comme expliqué dans la remarque \ref{remdecompo} (on traitera la relation $\sim_e$ dans la section suivante).

Dans \cite[Théorème 3.4.5]{lanard} nous avons obtenu une décomposition de $\rep[\ld][0]{G}$ indexée par $\Lpm{\Ild}$, l'ensemble des paramètres de Langlands inertiels modérés,
\[ \rep[\ld][0]{\G} = \prod_{\phi \in \Lpm{\Ild}} \rep[\ld][\phi]{\G}.\]
Celle ci est obtenue en identifiant $\Lpm{\Ild}$ à $(\ss*{\gpfinialg*{G}})^{\fr}$. Nous procédons de même ici en introduisant un ensemble $\Lpbm{\Ild}$, et en montrant que l'on a un diagramme commutatif
\[ \xymatrix{
\pairesSt/{\sim_r} \ar@{^{(}->}[r] \ar@{->}[d] & \ss*{\gpfini*{\G}} \ar@{->}[d] \ar@{<->}[r]^-{\sim} & \Lpbm{\Ild} \ar@{->}[d]\\
\pairesSt/{\sim_\infty} \ar@{^{(}->}[r]  & (\ss*{\gpfinialg*{\G}})^{\fr} \ar@{<->}[r]^-{\sim} & \Lpm{\Ild} }\]
La décomposition donnée par $\sim_\infty$ est la même que celle de \cite{lanard} (l'objet principal de cette section est donc $\sim_r$) et $\sim_r$ donne une décomposition plus fine
\[ \rep[\ld][0]{\G} = \prod_{(\phi, \sigma) \in \Lpbm{\Ild}} \rep[\ld][(\phi,\sigma)]{\G}.\]

Haines définit dans \cite{haines} définition 5.3.3, une notion d'équivalence inertielle pour des paramètres $\lambda : \weil \to \Ldual{G}[\Ql]$. Nous étendrons cette dernière en une notion de $\lprime$-équivalence inertielle. On notera  $\mathcal{B}_{m,\ld}^{st}$ l'ensemble des paramètres de Weil modérés à équivalence inertielle près ($\lprime$-équivalence inertielle si $\ld=\Zl$). On construira une bijection $\mathcal{B}^{st}_{m,\ld} \tosim \Lpbm{\Ild}$, permettant de réinterpréter la décomposition précédente en
\[ \rep[\ld][0]{\G} = \prod_{(\phi, \sigma) \in \Lpbm{\Ild}} \rep[\ld][(\phi,\sigma)]{\G} = \prod_{[\varphi] \in \mathcal{B}_{m,\ld}^{st}} \rep[\ld][[\varphi]]{\G}.\]
Dans le cas où $\G$ est un groupe classique non-ramifié, $\ld=\Ql$ et $p \neq 2$, cette décomposition est compatible à la correspondance de Langlands locale et est la décomposition de $\rep[\ld][0]{\G}$ en "blocs stables" (c'est à dire que ces facteurs correspondent à des idempotents primitifs du centre de Bernstein stable).

\subsection{Interprétation sur\texorpdfstring{ $\gpfinialg*{\G}$}{G*}}

\label{secpairestconjss}

Dans cette partie, nous allons expliquer le lien entre les $\sim_r$-classes (resp. les $\sim_\infty$-classes, resp. les $\sim_1$-classes) et les classes de conjugaison dans $\gpfinialg*{\G}$. De façon plus précise, nous introduisons un ensemble $\pairesSt*=\{(\gpfinialg*{S},t)\}$, où $\textsf{\textbf{S}}^{*}$ est un $\res$-tore maximal de $\gpfinialg*{\G}$ et $t \in (\textsf{\textbf{S}}^{*})^{\fr}$ est d'ordre inversible dans $\ld$ ainsi qu'une application $\pairesSt \to \pairesSt*/{\sim_{\gpfini*{G}}}$, où $\sim_{\gpfini*{G}}$ est la conjugaison par $\gpfini*{G}$. Nous montrerons alors que cette application induit le diagramme commutatif suivant
\[ \xymatrix{
	\pairesSt/{\sim_1} \ar@{^{(}->}[d] \ar@{->}[r] &  \pairesSt/{\sim_r} \ar@{^{(}->}[d] \ar@{->}[r] & \pairesSt/{\sim_\infty} \ar@{^{(}->}[d]\\
	\pairesSt*/{\sim_{\gpfini*{G}}} \ar@{->}[r] & \ss*{\gpfini*{G}} \ar@{->}[r] & (\ss*{\gpfinialg*{G}})^{\fr} }\] 
Si $\G$ est de plus quasi-déployé, les injections verticales sont alors des bijections.

\begin{Def}
On note $\pairesSt*$ l'ensemble des paires $(\gpfinialg*{S},t)$ où $\textsf{\textbf{S}}^{*}$ est un $\res$-tore maximal de $\gpfinialg*{\G}$ et $t \in (\textsf{\textbf{S}}^{*})^{\fr}$ est d'ordre inversible dans $\ld$.
\end{Def}

Définissons une application $\pairesSt \to \pairesSt*/{\sim_{\gpfini*{G}}}$, où $\sim_{\gpfini*{G}}$ désigne la conjugaison par $\gpfini*{G}$.

Soit $(\gpalg{S},\theta) \in \pairesSt$. Prenons $\gpfinialg*{S}$ un $\res$-tore maximal de $\gpfinialg*{\G}$ dual de $\gpalg{S}$ (la dualité ici signifie que la donnée radicielle de $\gpalg{S}$ munie de l'action du générateur de $\Gal(\knr/\kk)$ est la duale de celle de $\gpfinialg*{S}$ munie de l'action du générateur de $\Gal(\resalg/\res)$). Notons que ce dernier est bien défini à $\gpfini*{G}$-conjugaison près. L'identification $X^{*}(\gpalg{S}) \simeq X_{*}(\gpfinialg*{S})$ fournit une bijection $X^{*}(\mathbf{S})/(\fr_{\mathbf{S}}-1)X^{*}(\mathbf{S}) \tosim X_{*}(\gpfinialg*{S})/(\fr-1)X_{*}(\gpfinialg*{S})$. Composée avec la bijection $X_{*}(\gpfinialg*{S})/(\fr-1)X_{*}(\gpfinialg*{S}) \tosim (\gpfinialg*{S})^{\fr}$ de l'annexe \ref{secIsoTores}, on obtient la bijection
\[X^{*}(\mathbf{S})/(\fr_{\mathbf{S}}-1)X^{*}(\mathbf{S}) \tosim (\gpfinialg*{S})^{\fr}.\]
À $\theta$ est donc associé un $t \in (\gpfinialg*{S})^{\fr}$. La paire  $(\gpfinialg*{S},t)$ est bien définie à $\gpfini*{G}$-conjugaison près, et on obtient de la sorte une application
\[ \pairesSt \to \pairesSt*/{\sim_{\gpfini*{G}}}.\]

\begin{Pro}
\label{proInjSim1}
L'application $\pairesSt \to \pairesSt*/{\sim_{\gpfini*{G}}}$ passe au quotient et induit une injection $\pairesSt/{\sim_1} \hookrightarrow \pairesSt*/{\sim_{\gpfini*{G}}}$.

Si $\G$ est de plus quasi-déployé, cette injection est alors une bijection.
\end{Pro}

\begin{proof}
Pour traiter la classe de $\sim_1$-équivalence, nous allons utiliser les résultats de \cite{debacker}. Pour cela nous fixons $\gpalg{T}$ un tore non ramifié de $\gpalg{\G}$ et $\textsf{\textbf{T}}^{*}$ un $\res$-tore maximal de $\gpfinialg*{\G}$ en dualité avec $\gpalg{T}$. On note $\Wf$ le groupe de Weyl fini de $\gpalg{T}$.

Notre groupe $\gpalg{\G}$ étant déployé sur $\knr$, nous savons par le lemme 4.3.1 de \cite{debacker} qu'il y a une injection entre les classes de $\sim_{1}$-équivalence de tores maximaux non-ramifiés et les classes de $\fr$-conjugaison dans $\Wf$, c'est à dire $H^{1}(\fr,\Wf)$ (cette injection est bijective si on suppose de plus $\G$ quasi-déployé) . Cette application fonctionne de la manière suivante. Soient $\gpalg{S}$ et $\gpalg{T}$ deux tores non-ramifiés. Ceux-ci étant non-ramifiés, ils sont conjugués sous $\G^{nr}$. Ainsi, il existe $g \in \G^{nr}$ tel que $S^{nr}={}^{g}T^{nr}$. Or $\fr(S^{nr})=S^{nr}$ donc cela définit un élément $n := g^{-1}\fr(g) \in Z^{1}(\fr,N^{nr})$ où $N^{nr}:=N(\G^{nr},T^{nr})$. On obtient $w := nT^{nr} \in Z^{1}(\fr,\Wf)$ dont la classe $[w] \in H^{1}(\fr,\Wf)$ est indépendante du choix de $g$. Notons également que $(S^{nr})^{\fr}={}^{g}((T^{nr})^{\fr_{w}})$ où $\fr_{w}:=\Ad(n) \circ \fr$.

Soit maintenant $(\gpalg{S},\theta) \in \pairesSt$. Comme ci-dessus, choisissons $g \in G^{nr}$ tel que $S^{nr}={}^{g}T^{nr}$ et posons $w$ l'image de $g^{-1}\fr(g)$ dans $\Wf$ et $\theta_w = g^{-1}\theta \in X^{*}(\textbf{T})/((w\vartheta) \circ \psi -1)X^{*}(\textbf{T})$. Alors la même preuve que le lemme 4.3.1 de \cite{debacker} en rajoutant les caractères $\theta$ montre que la classe de $\sim_{1}$-équivalence de $(\gpalg{S},\theta)$ est caractérisée par la classe de $\fr$-conjugaison de $(w,\theta_w)$.

La bijection  $X^{*}(\textbf{T})/((w\vartheta)\circ \psi -1)X^{*}(\textbf{T}) \simeq X_{*}(\gpfinialg*{T})/(\fr_{w}-1)X_{*}(\gpfinialg*{T}) \simeq (\gpfinialg*{T})^{\fr_{w}}$ (rappelée dans l'annexe \ref{secIsoTores}) envoie $\theta_w$ sur $t_w\in (\textsf{\textbf{T}}^{*})^{\fr_{w}}$. La classe de $\fr$-conjugaison de $(w,t_w)$ est elle même en correspondance avec la classe de conjugaison d'une paire $(\gpfinialg*{S},t)$ (cette correspondance est aussi rappelée dans l'annexe \ref{secIsoTores}). Notons que $\gpfinialg*{S}$ est en dualité avec $\gpalg{S}$, et $t$ correspond à $\theta$ via la bijection $X^{*}(\mathbf{S})/(\fr_{\mathbf{S}}-1)X^{*}(\mathbf{S}) \tosim (\gpfinialg*{S})^{\fr}$, donc cette classe de conjugaison est bien l'image de $(\gpalg{S},\theta)$ par l'application $\pairesSt \to \pairesSt*/{\sim_{\gpfini*{G}}}$, d'où le résultat.
\end{proof}

Nous avons une application bien définie $\pairesSt*/{\sim_{\gpfini*{G}}} \to \ss*{\gpfini*{G}}$, $(\gpfinialg*{S},t) \mapsto t$. Celle-ci induit donc deux applications $\pairesSt \to \ss*{\gpfini*{G}}$ et $\pairesSt \to (\ss*{\gpfinialg*{G}})^{\fr}$ (en composant avec $\ss*{\gpfini*{G}} \to (\ss*{\gpfinialg*{G}})^{\fr}$).

\begin{Pro}
\label{prosimconj}
Les applications précédentes passent au quotient pour les relations d'équivalence $\sim_r$ et $\sim_\infty$ et donnent des injections rendant le diagramme suivant commutatif
\[ \xymatrix{
	\pairesSt/{\sim_1} \ar@{^{(}->}[d] \ar@{->}[r] &  \pairesSt/{\sim_r} \ar@{^{(}->}[d] \ar@{->}[r] & \pairesSt/{\sim_\infty} \ar@{^{(}->}[d]\\
	\pairesSt*/{\sim_{\gpfini*{G}}} \ar@{->}[r] & \ss*{\gpfini*{G}} \ar@{->}[r] & (\ss*{\gpfinialg*{G}})^{\fr} }\] 
Si $\G$ est de plus quasi-déployé, ces injections sont des bijections.
\end{Pro}

\begin{proof}
Soit $m\in \mathbb{N}^{*}$. Passer de $\sim_{1}$ à $\sim_{m}$ revient à changer $\fr$ en $\fr^{m}$. De plus l'application $\theta \mapsto \theta \langle m \rangle$ correspond au niveau des groupes finis à l'application $t\in (\textsf{\textbf{T}}^{*})^\fr \mapsto t\in (\textsf{\textbf{T}}^{*})^{\fr^{m}}$ (voir par example \cite{DeligneLusztig} section 5.3). On en déduit (comme dans \ref{proInjSim1}) que l'on a $\pairesSt/{\sim_m} \hookrightarrow \pairesSt*/{\sim_{(\gpfinialg*{G})^{\fr^m}}}$. En particulier on a que $\pairesSt/{\sim_\infty} \hookrightarrow (\ss*{\gpfinialg*{G}})^{\fr}$ .

Soient $(\mathbf{S}, \theta),(\mathbf{S}', \theta') \in \pairesSt$. Notons $(\gpfinialg*{S},t)$ et $((\gpfinialg{S}')^*,t')$ des représentants des images respectives de $(\mathbf{S}, \theta)$ et $(\mathbf{S}', \theta')$ par $\pairesSt \to \pairesSt*/{\sim_{\gpfini*{G}}}$. Alors $(\mathbf{S}, \theta)$ et $(\mathbf{S}', \theta')$ sont $\sim_{r}$-équivalentes si et seulement si elles sont $\sim_{\infty}$-équivalentes et que le $g$ reliant les deux paires vérifie $g^{-1}\fr(g) \in \NSte$. Par ce qui précède, être $\sim_{\infty}$-équivalent est équivalent à ce que $t$ et $t'$ soient géométriquement conjugués. Notons $w$ la réduction de $g^{-1}\fr(g)$ dans $\WStf$. Comme dans le lemme \ref{lemIdentifWeyl}, nous savons que $\WStf$ correspond au groupe de Weyl de $C_{\gpfinialg*{\G}}(t)^{\circ}$. Ainsi $(\mathbf{S}, \theta)$ et $(\mathbf{S}', \theta')$ sont $\sim_{r}$-équivalentes si et seulement si $t'$ est conjugué à $t$ par un élément $g^{*} \in \gpfinialg*{\G}$ vérifiant $w=(g^{*})^{-1}\fr(g^{*}) \in \WStf$ (c'est à dire tel que $g^{*} \in (\gpfinialg*{\G})^{\fr} C_{\gpfinialg*{\G}}(t)^{\circ}$) (une fois le tore $\gpalg{S}$ fixé la dualité fonction comme celle sur les groupes finis rappelée après la proposition \ref{proBijPairesSt}), si et seulement si $t$ et $t'$ sont rationnellement conjugués.
\end{proof}

\subsection{Classes de conjugaison rationnelles}

\label{secClassesRatioX}

La proposition \ref{prosimconj} montre que les classes de $\sim_r$-équivalence correspondent à des classes de conjugaison rationnelles. Dans \cite{lanard}, pour relier les classes de conjugaison géométriques semi-simples $\fr$-stables aux paramètres de l'inertie, nous avons interprété $(\ss{\gpfinialg*{\G}})^{\fr}$ comme $((X\otimes_{\mathbb{Z}} \resalg^{\times})/\Wf)^{\fr}$. Nous souhaitons faire de même ici, mais pour les classes de conjugaison rationnelles.  Nous introduisons donc un ensemble $\widetilde{X}_{\ld}$ muni d'une action de $\Wf$ de sorte que l'on ait le diagramme commutatif
\[ \xymatrix{
	\widetilde{X}_{\ld}/\Wf \ar@{->}[r]^-{\sim} \ar@{->}[d] & \ss*{\gpfini*{\G}} \ar@{->}[d]\\
	((X\otimes_{\mathbb{Z}} \resalg^{\times})_{\ld}/\Wf)^{\fr} \ar@{->}[r]^-{\sim} & (\ss*{\gpfinialg*{\G}})^{\fr} }\]
($(X\otimes_{\mathbb{Z}} \resalg^{\times})_{\ld}$ désigne les éléments de $(X\otimes_{\mathbb{Z}} \resalg^{\times})$ d'ordre inversible dans $\ld$).

\sautintro

Soit $s \in \gpfinialg*{T} \simeq X \otimes_{\mathbb{Z}} \resalg^{\times}$, on note $\Ws*{s}$ le sous-groupe de $\Wf$ formé par les éléments $w$ tels que $w\cdot s=s$, qui correspond au groupe de Weyl de $C_{\gpfinialg*{\G}}(s)$ par rapport à $\gpfinialg*{T}$, et $\Ws{s}$ le sous-groupe de $\Wf$ engendré par les $s_{\alpha}$ tels que $\alpha(s)=1$, qui correspond au groupe de Weyl de $C_{\gpfinialg*{\G}}(s)^{\circ}$ par rapport à $\gpfinialg*{T}$. Posons
\[\widetilde{X}:=\{ (s,\bar{w})\in (X \otimes_{\mathbb{Z}} \resalg^{\times}) \times (\Ws{s}\backslash \Wf), s = \bar{w}\cdot \fr(s) \}\]
et $\widetilde{X}_{\ld}$ le sous-ensemble de $\widetilde{X}$ formé par les $(s,\bar{w})$ avec $s$ d'ordre inversible dans $\ld$. Le groupe de Weyl $\Wf$ agit sur $\widetilde{X}_{\ld}$ par
\[v \cdot (s,\bar{w})=(v \cdot s, \overline{vw\fr(v)^{-1}}),\ v \in \Wf.\]
Vérifions que $\overline{vw\fr(v)^{-1}}$ a bien un sens. Si $\bar{w}_1=\bar{w}_2$ dans $\Ws{s}\backslash \Wf$, c'est à dire $w_1=u w_2 $ avec $u \in \Ws{s}$, alors $vw_1\fr(v)^{-1}=(vuv^{-1})(vw_2\fr(v)^{-1})$ et comme $vuv^{-1} \in \Ws{v.s}$ on a bien le résultat voulu.

\bigskip

On va maintenant relier $\widetilde{X}_{\ld}/\Wf$ aux classes de conjugaison rationnelles semi-simples dans $\gpfini*{\G}$ d'ordre inversible dans $\ld$.

Soient $(s,\bar{w}) \in \widetilde{X}_{\ld}$ et $n \in \gpfinialg*{\G}$ un relèvement de $\bar{w}$. Comme $s=\bar{w}\fr(s)$, on a $s=n\fr(s)n^{-1}$. Écrivons $n$ sous la forme $n=g^{-1}\fr(g)$ (grâce à la surjectivité de l'application de Lang), alors $t:=gsg^{-1}$ vérifie $t=\fr(t)$. On associe à $(s,\bar{w})$ la classe de conjugaison rationnelle de $t$.

\begin{Lem}
Ce procédé ne dépend pas des choix effectués et  définit une application $\widetilde{X}_{\ld} \rightarrow \ss*{\gpfini*{\G}}$.
\end{Lem}

\begin{proof}
Prenons un autre relèvement $n'=(g')^{-1}\fr(g')$ de $\bar{w}$. Alors $n'n^{-1} \in C_{\gpfinialg*{\G}}(s)^{\circ}$. Le groupe réductif connexe $C_{\gpfinialg*{\G}}(s)^{\circ}$ est $\fr_{n}:=\Ad(n) \circ \fr$ stable, donc il existe $h \in C_{\gpfinialg*{\G}}(s)^{\circ}$ tel que $n'n^{-1}=h^{-1}\fr_{n}(h)=h^{-1}n\fr(h)n^{-1}$. Donc $n'=h^{-1}n\fr(h)=(gh)^{-1}\fr(gh)$. Ainsi $gh(g')^{-1} \in (\gpfinialg*{\G})^{\fr}$ et $t' := g's(g')^{-1} = (gh(g')^{-1})^{-1}t(gh(g')^{-1})$ est rationnellement conjugué à $t$.
\end{proof}

\begin{Pro}
\label{proXtildeG}
Cette application passe au quotient pour fournir une bijection naturelle $\widetilde{X}_{\ld}/\Wf \overset{\sim}{\longrightarrow} \ss*{\gpfini*{\G}}$.
\end{Pro}

\begin{proof}
Montrons d'abord qu'elle passe au quotient. Soient $(s,\bar{w}),(s',\bar{w'}) \in \widetilde{X}_{\ld}$ et $v \in \Wf$ tels que $(s,\bar{w})=v \cdot (s',\bar{w'})$. Soient $n=g^{-1}\fr(g)$ un relèvement de $\bar{w}$ et $m$ un relèvement de $v$. Alors $n':=m^{-1}n\fr(m)=(gm)^{-1}\fr(gm)$ est un relèvement de $\bar{w}'$. Notons $t$ la classe de conjugaison semi-simple rationnelle associée à $(s,\bar{w})$ et $t'$ celle associée à $(s',\bar{w}')$. Alors $t'$ est la classe de conjugaison rationnelle de $gms'm^{-1}g^{-1}=gsg^{-1}$ donc $t'=t$.

Il reste à vérifier la bijectivité. Construisons la bijection réciproque. Soit $t\in\ss*{\gpfini*{\G}}$, alors il existe $\gpfinialg*{S}$ un tore maximal $\fr$-stable de $\gpfinialg*{G}$ tel que $t \in (\gpfinialg*{S})^{\fr}$. Prenons $g \in \gpfinialg*{G}$ tel que ${}^{g}\gpfinialg*{S} = \gpfinialg*{T}$. Alors $g^{-1}\fr(g)$ normalise $\gpfinialg*{T}$ et on peut former le couple $(s,\bar{w})=(gtg^{-1},\overline{g^{-1}\fr(g)})$. Ceci nous définit une application $\ss*{\gpfini*{\G}} \rightarrow \widetilde{X}_{\ld}/\Wf$ (elle est indépendante des choix effectués) qui est la réciproque de $\widetilde{X}_{\ld}/\Wf \rightarrow \ss*{\gpfini*{\G}}$.
\end{proof}

Nous avons une application naturelle $\widetilde{X}_{\ld}/\Wf \to ((X\otimes_{\mathbb{Z}} \resalg^{\times})_{\ld}/\Wf)^{\fr}$, $(s,\bar{w}) \mapsto s$ et il est clair au vu de la définition de l'application $\widetilde{X}_{\ld} \rightarrow \ss*{\gpfini*{\G}}$ que l'on a
\[ \xymatrix{
	\widetilde{X}_{\ld}/\Wf \ar@{->}[r]^-{\sim} \ar@{->}[d] & \ss*{\gpfini*{\G}} \ar@{->}[d]\\
	((X\otimes_{\mathbb{Z}} \resalg^{\times})_{\ld}/\Wf)^{\fr} \ar@{->}[r]^-{\sim} & (\ss*{\gpfinialg*{\G}})^{\fr} }\]

\subsection{Réinterprétation de la relation d'équivalence \texorpdfstring{$\sim_r$}{~r}}
\label{secSimpler}

Dans cet article, nous obtenons des décompositions de $\rep[\ld][0]{G}$ à partir de relations d'équivalence ($\sim_e$, $\sim_r$ et $\sim_\infty$) sur $\pairesSt$. Dans \cite{lanard}, la décomposition de $\rep[\ld][0]{G}$, indexée par des paramètres inertiels, ne passe pas par les paires $(\gpalg{S},\theta)$. Elle est construite à partir d'une application $\Jss \to (\ss{\gpfinialg*{G}})^{\fr}$ (on rappelle que $\Jss$ désigne  l'ensemble des couples $(\sigma,s)$ où $\sigma \in \bt$ et $s \in \ss*{\quotred*{\G}{\sigma}}$, voir définition \ref{defJss}). La bijection $\pairesSt/{\sim_e} \tosim \Smin \tosim \Jss/{\simJss}$ (théorème \ref{theBijSime}) permet de définir sur $\Jss$ deux relations d'équivalence $\sim_r$ et $\sim_\infty$ qui correspondent aux relations éponymes sur $\pairesSt$. Nous montrons dans cette partie que l'application $\Jss \to (\ss{\gpfinialg*{G}})^{\fr}$ de \cite{lanard} passe au quotient pour donner une injection $\Jss/{\sim_\infty} \hookrightarrow (\ss{\gpfinialg*{G}})^{\fr}$ (on aura donc une compatibilité entre la décomposition dans \cite{lanard} et celle obtenue ici grâce à $\sim_\infty$). Nous souhaitons également trouver une application plus simple pour les classes de $\sim_r$-équivalence. On construira donc $\Jss \to \ss{\gpfini*{G}}$ qui passe au quotient pour $\sim_r$ et donne le diagramme commutatif
\[  \xymatrix{
    \Jss/{\sim_r} \ar@{^{(}->}[rr] \ar@{<->}[rd]^-{\sim} \ar@{->}[dd] &&  \ss*{\gpfini*{G}} \ar@{->}[dd] \\
    &  \pairesSt/{\sim_r} \ar@{^{(}->}[ru] \ar@{->}[dd] \\
     \Jss/{\sim_\infty}  \ar@{^{(}->}[rr]|\hole \ar@{<->}[rd]^-{\sim} && (\ss*{\gpfinialg*{G}})^{\fr} \\
    &  \pairesSt/{\sim_\infty} \ar@{^{(}->}[ru] }\]

\sautintro

Soient $\sigma \in \bts$ et $\gpalg{T}$ un tore maximal  maximalement déployé tel que $\sigma \in \app{T}{\kk}$. Rappelons que l'on a une bijection  canonique  $\widetilde{X}_{\ld}/\Wf \simeq \ss*{\gpfini*{\G}}$ (proposition \ref{proXtildeG}). On a également $\widetilde{X}_{\ld}/\Wx{\sigma} \simeq \ss*{\quotred*{\G}{\sigma}}$. L'application $\Wx{\sigma} \hookrightarrow \Wf$ induit une application $\Wx{\sigma}/\Ws{\sigma,s} \rightarrow \Wf/\Ws{s}$ et donc $\widetilde{X}_{\ld}/\Wx{\sigma} \rightarrow \widetilde{X}_{\ld}/\Wf$. On obtient de la sorte notre application $\ss*{\quotred*{\G}{\sigma}} \rightarrow \ss*{\gpfini*{\G}}$.

\begin{Pro}
\label{proCquotr}
L'application précédente est indépendante du choix du tore $\gpalg{T}$ et définit une application $\Jss \to \ss*{\gpfini*{\G}}$. Cette dernière passe au quotient pour $\sim_r$ et rend le diagramme suivant commutatif
     \[  \xymatrix{
    \Jss/{\sim_r} \ar@{->}[rd] \ar@{<->}[rr]^-{\sim} && \pairesSt/{\sim_r}  \ar@{->}[ld] \\
    &   \ss*{\gpfini*{G}} &}\]
\end{Pro}

\begin{proof}
Soit $t \in \ss*{\quotred*{\G}{\sigma}}$ et notons $s$ son image par $\ss*{\quotred*{\G}{\sigma}} \rightarrow \ss*{\gpfini*{\G}}$. Prenons $\gpfinialg*{S}$ un $\res$-tore maximal de $\ss*{\quotred*{\G}{\sigma}}$ tel que $t \in (\gpfinialg*{S})^{\fr}$. Considérons $(\gpfinialg{S},\theta)$ en dualité avec $(\gpfinialg*{S},t)$ et que l'on relève en $(\gpalg{S},\theta)$ où $\gpalg{S}$ est un tore maximal de $\gpalg{G}$. Il est alors clair que la classe de $\sim_r$-équivalence de $(\gpalg{S},\theta)$ correspond à la classe de $\sim_r$-équivalence de $(\sigma,t)$ via la bijection $\Jss/{\sim_r} \tosim \pairesSt/{\sim_r}$. Nous souhaitons donc montrer que $s$ est l'image de $(\gpalg{S},\theta)$ par l'application $\pairesSt/{\sim_r} \to \ss*{\gpfini*{G}}$.

L'application $\ss*{\quotred*{\G}{\sigma}} \rightarrow \ss*{\gpfini*{\G}}$ étant définie à partir des groupes de Weyl, nous allons réinterpréter les résultats grâce à ces derniers. Notons $\gpfinialg{T}$ le tore induit par $\gpalg{T}$ sur $\quotredalg{\G}{\sigma}$ et  $\gpfinialg*{T}$ en dualité avec  $\gpfinialg{T}$. Notons $\pairesStw{\Wf}/{\fr}$ l'ensemble des classes de $\fr$-conjugaison de paires $(w,\theta)$ où $w \in \Wf$, $\theta \in X/((w\vartheta \circ \psi -1)X$ et $\theta$ est d'ordre inversible dans $\ld$. Notons de même $\pairesStw{\Wx{\sigma}}/{\fr}$ le même ensemble mais avec $w \in \Wx{\sigma}$. L'injection $\Wx{\sigma} \hookrightarrow \Wf$ permet de définir une application $\pairesStw{\Wx{\sigma}}/{\fr} \rightarrow \pairesStw{\Wf}/{\fr}$. Soit $(w_\sigma,\theta_\sigma) \in \pairesStw{\Wx{\sigma}}/{\fr}$ correspondant à la classe de $\quotred*{\G}{\sigma}$-conjugaison de la paire $(\gpfinialg*{S},t)$ (voir annexe \ref{secIsoTores}) et nommons $(w_0,\theta_0) \in  \pairesStw{\Wf}/{\fr}$ l'image de $(w_\sigma,\theta_\sigma)$ par l'application $\pairesStw{\Wx{\sigma}}/{\fr} \rightarrow \pairesStw{\Wf}/{\fr}$. La classe de $\fr$-conjugaison de la paire $(w_0,\theta_0)$ correspond à la classe de $\gpfini*{G}$-conjugaison d'une paire $(\gpfinialg*{S}_0,t_0)$ et il est clair que $s$ est la classe de conjugaison de $t_0$.

En termes de groupes de Weyl, la proposition \ref{proInjSim1} peut se réinterpréter en une injection $\pairesSt/{\sim_1} \hookrightarrow \pairesStw{\Wf}/{\fr}$. Grâce à la proposition \ref{prosimconj}, il nous suffit pour conclure de montrer que la classe de $(w_0,\theta_0)$ est l'image de $(\gpfinialg{S},\theta)$ par $\pairesSt/{\sim_1} \hookrightarrow \pairesStw{\Wf}/{\fr}$. Le corollaire 4.3.2 de \cite{debacker} montre ce résultat mais sans les caractères $\theta$. La même preuve s'adapte aisément en rajoutant les caractères $\theta$, ce qui nous démontre la proposition.

Notons que comme l'application $\pairesSt/{\sim_r} \to \ss*{\gpfini*{G}}$ ne dépend pas du choix du tore $\gpalg{T}$, il en est de même pour $\ss*{\quotred*{\G}{\sigma}} \rightarrow \ss*{\gpfini*{\G}}$.
\end{proof}

\begin{Pro}
\label{procompatibilitelanard}
L'application $\Jss \to (\ss*{\gpfinialg*{\G}})^{\fr}$ (de \cite{lanard}) passe au quotient pour $\sim_\infty$ et rend le diagramme suivant commutatif
     \[  \xymatrix{
    \Jss/{\sim_\infty} \ar@{->}[rd] \ar@{<->}[rr]^-{\sim} && \pairesSt/{\sim_\infty}  \ar@{->}[ld] \\
    &  (\ss*{\gpfinialg*{\G}})^{\fr} &}\]
\end{Pro}

\begin{proof}
Le résultat se déduit de la proposition \ref{proCquotr} et des diagrammes commutatifs
\[  \xymatrix{
    \Jss \ar@{->}[rd] \ar@{->}[r] & \ss*{\gpfini*{G}}  \ar@{->}[d] \\
    &  (\ss*{\gpfinialg*{\G}})^{\fr}}\]
et (proposition \ref{prosimconj})
\[  \xymatrix{
    \Jss \ar@{->}[rd] \ar@{->}[r] &  \Jss/{\sim_r}  \ar@{->}[d] \ar@{<->}[r]^-{\sim} &  \pairesSt/{\sim_r} \ar@{->}[r]  \ar@{->}[d]&  \ss*{\gpfini*{G}}  \ar@{->}[d]\\
    &  \Jss/{\sim_\infty}  \ar@{<->}[r]^-{\sim} &  \pairesSt/{\sim_\infty} \ar@{->}[r] & (\ss*{\gpfinialg*{\G}})^{\fr}}\]
\end{proof}
Il est aisé de voir que l'on a bien le diagramme commutatif annoncé dans l'introduction.

\subsection{Interprétation en termes duaux}
Les classes d'équivalence pour $\sim_{\infty}$ correspondent aux classes de conjugaison semi-simples géométriques $\fr$-stables dans $\gpfinialg*{\G}$ par la proposition \ref{prosimconj}. Nous avons interprété ces dernières par des paramètres de Langlands inertielles dans \cite{lanard}. On obtient donc une décomposition de $\rep[\ld][0]{\G}$ indéxée par  $\Lpm{\Ild}$, l'ensemble des paramètres inertielles,  qui est la même que celle dans \cite{lanard} par la section \ref{secSimpler}. Nous souhaitons obtenir de même, une interprétation en des termes duaux pour la relation d'équivalence $\sim_r$. Nous introduisons pour cela un ensemble $\Lpbm{\Ild}$ rendant commutatif le diagramme suivant
\[ \xymatrix{
\Lpbm{\Ild} \ar@{->}[r]^-{\sim} \ar@{->}[d] & \ss*{\gpfini*{\G}} \ar@{->}[d]\\
\Lpm{\Ild} \ar@{->}[r]^-{\sim} & (\ss*{\gpfinialg*{\G}})^{\fr} }\]
Ainsi, grâce à la proposition \ref{prosimconj}, les classes de $\sim_r$-équivalence fourniront une décomposition de $\rep[\ld][0]{\G}$ indexée par $\Lpbm{\Ild}$.

\sautintro

Commençons par quelques rappels succincts sur les paramètres inertiels et la décomposition de \cite{lanard}.

Soit $\mathbf{T}$ un $\kk$-tore maximal $\kk$-déployé maximalement déployé. On notera $\mathcal{G}_{\kk}=\Gal(\kalg/\kk)$ le groupe de Galois absolu de $\kk$, et on choisit $\text{Frob}$, un Frobenius inverse dans $\mathcal{G}_{\kk}$, qui induit un automorphisme d'ordre fini  sur $X_{*}(\mathbf{T})$ que l'on nomme $\vartheta$. La dualité entre $X_{*}(\mathbf{T})$ et $X^{*}(\mathbf{T})$ permet d'associer de façon naturelle à $\vartheta$ un automorphisme $\widehat{\vartheta} \in \Aut(\gpalg*{T}[\Ql])$ (car $\gpalg*{T}[\Ql]:=X^{*}(\mathbf{T})\otimes \Ql^{\times}$). Puis, via le choix d'un épinglage $(\gpalg*{\G},\gpalg*{B},\gpalg*{T},\{x_{\alpha}\}_{\alpha \in \Delta})$, on étend ce dernier en un automorphisme $\widehat{\vartheta} \in \Aut(\gpalg*{\G})$. On définit alors $\Ldual{\G}$ par $\Ldual{\G}:=\gpalg*{\G} \rtimes \langle\widehat{\vartheta}\rangle$.

Soient $\weil$ le groupe de Weil et $\weil'=\weil\ltimes \Ql$ le groupe de Weil-Deligne. On note $\Lparametre{\G}[\Ql]$ l'ensemble des morphismes admissibles $\varphi : W'_{k} \rightarrow \Ldual{\G}[\Ql]$ modulo les automorphismes intérieurs par des éléments de $\gpalg*{\G}[\Ql]$. Pour simplifier les notations, on notera généralement $\Ldual{\G}$ à la place de $\Ldual{\G}[\Ql]$ lorsqu'il n'y a pas d'ambiguïté, par exemple $\Lparametre{\G}[\Ql]$ devient $\Lparametre{\G}$. Soit $I$ un sous-groupe de $\weil$. On note alors $\Lp{I}$ l'ensemble des classes de $\gpalg*{\G}$-conjugaison des morphismes continus $I \rightarrow \Ldual{\G}[\Ql]$ (où $\Ql$ est muni de la topologie discrète) qui admettent une extension à un $L$-morphisme de $\Lparametre{\G}$. Si $I$ contient $\inersauv$, l'inertie sauvage, on dit qu'un paramètre $\phi \in \Lp{I}$ est modéré s'il est trivial sur $\inersauv$, et on note $\Lpm{I}$ pour l'ensemble des paramètres de $I$ modérés.

Notons $\iner$ le groupe d'inertie et $\inerl:=\ker\{\iner\rightarrow\mathbb{Z}_{l}(1)\}$ son sous-groupe fermé maximal de pro-ordre premier à $\lprime$. On unifiera les notations en posant $\Ild$ qui vaut $\iner$ si $\ld=\Ql$ et $\inerl$ si $\ld = \Zl$.

\begin{Pro}[\cite{lanard} proposition 3.1.4]
	\label{proIdentifPhiInert}
    On a une identification
    \[ \Lpm{\Ild} \longleftrightarrow ((X^{*}(\mathbf{T})\otimes_{\mathbb{Z}} \resalg^{\times})_{\ld}/\Wf)^{\fr} \longleftrightarrow (\ss*{\gpfinialg*{G}})^{\fr}\]
    ($(X^{*}(\mathbf{T})\otimes_{\mathbb{Z}} \resalg^{\times})_{\ld}$ désigne les éléments de $(X^{*}(\mathbf{T})\otimes_{\mathbb{Z}} \resalg^{\times})$ d'ordre inversible dans $\ld$).
\end{Pro}

Combinée à la proposition \ref{prosimconj}, nous obtenons une injection 
\[\pairesSt/{\sim_\infty} \hookrightarrow \Lpm{\Ild},\]
qui fournit (voir remarque \ref{remdecompo}) une décomposition
\[ \rep[\ld][0]{\G} = \prod_{\phi \in \Lpm{\Ild}} \rep[\ld][\phi]{\G}.\]
De plus, cette dernière est la même que celle de \cite{lanard} théorème 3.4.5 d'après la proposition \ref{procompatibilitelanard}.

\medskip

Revenons maintenant à la relation d'équivalence $\sim_r$. Pour cela, reprenons les notations de \cite{datFunctoriality}. Étant donné une suite exacte $H \hookrightarrow \widetilde{H} \twoheadrightarrow W$ de groupes topologiques, on note $\Sigma(W,\widetilde{H})$ l'ensemble des sections continues $W \rightarrow \widetilde{H}$ qui scindent la suite exacte, et $\overline{\Sigma}(W,\widetilde{H})$ l'ensemble des classes de $H$-conjugaison dans $\Sigma(W,\widetilde{H})$.

\medskip

Soit $\phi \in \Lp{\Ild}$. On note $C_{\gpalg*{\G}}(\phi)$ le centralisateur de $\phi(\Ild)$ dans $\gpalg*{\G}$ qui est un groupe réductif, possiblement non connexe. Par définition $\phi$ peut s'étendre en un paramètre $\varphi : \weil \longrightarrow \Ldual{\G}$. Alors d'après la section 2.1.1 de \cite{datFunctoriality} le sous-groupe $\widetilde{C}_{\gpalg*{\G}}(\phi):=C_{\gpalg*{\G}}(\phi) \varphi(\weil)$ de $\Ldual{\G}$ est indépendant du choix de $\varphi$. Notons maintenant $\pi_{0}(\phi):=\pi_{0}(C_{\gpalg*{\G}}(\phi))$ et $\widetilde{\pi}_{0}(\phi):=\widetilde{C}_{\gpalg*{\G}}(\phi)/C_{\gpalg*{\G}}(\phi)^{\circ}$ qui donne une suite exacte scindée (\cite{datFunctoriality} section 2.1.4)
\[ 0 \longrightarrow \pi_{0}(\phi) \longrightarrow \widetilde{\pi}_{0}(\phi) \longrightarrow  \langle \widehat{\vartheta} \rangle \longrightarrow 0.\]

\begin{Def}
Notons $\Lpb{\Ild}$ l'ensemble des couples $(\phi,\sigma)$ à $\gpalg*{\G}$-conjugaison près, où $\phi : \Ild \to \Ldual{\G}$ est un paramètre inertiel et $\sigma \in \Sigma( \langle \widehat{\vartheta} \rangle ,\widetilde{\pi}_{0}(\phi))$. Appelons également  $\Lpbm{\Ild}$ le sous-ensemble de $\Lpb{\Ild}$ des $(\phi,\sigma)$ avec $\phi$ modéré.
\end{Def}

Nous souhaitons établir une bijection canonique entre $\Lpbm{\Ild}$ et $\ss*{\gpfini*{\G}}$ (puisque $\ss*{\gpfini*{\G}}$ correspond à l'ensemble des classes de $\sim_r$-équivalence par la proposition \ref{prosimconj}).

\bigskip

La section \ref{secClassesRatioX} construit une bijection canonique entre $\ss*{\gpfini*{\G}}$ et $\widetilde{X}_{\ld}/\Wf$. Il nous suffit donc de construite une bijection $\Lpbm{\Ild} \overset{\sim}{\rightarrow} \widetilde{X}_{\ld}/\Wf$.

Prenons $(\phi,\sigma) \in \Lpbm{\Ild}$. Quitte à conjuguer $(\phi,\sigma)$, on peut supposer que $\phi(\Ild) \subseteq \gpalg*{T}$. Le groupe $\Ild/\inersauv$ étant procyclique, $\phi$ est déterminé par l'image d'un générateur. Le système compatible de racines de l'unité, que l'on a fixé au début, nous fournit un tel générateur, ainsi $\phi$ est déterminé par un élément semi-simple $s \in \gpalg*{T}$ d'ordre inversible dans $\ld$. Comme dans \cite{lanard} section 3.1, on identifie $s$ à un élément $s \in X \otimes_{\mathbb{Z}} \resalg^{\times}$ (toujours grâce au même système compatible de racines de l'unité).

\begin{Lem}
\label{lemRelevementN}
Soit $f\widehat{\vartheta}^{n} \in \widetilde{C}_{\gpalg*{\G}}(\phi)$. Alors il existe $c \in C_{\gpalg*{\G}}(\phi)^{\circ}$ tel que $cf \in N$, où $N:=N(\gpalg*{T},\gpalg*{\G})$ est le normalisateur de $\gpalg*{T}$ dans $\gpalg*{\G}$.
\end{Lem}

\begin{proof}
On a $f\widehat{\vartheta}^{n} \in N(C_{\gpalg*{\G}}(\phi)^{\circ},\Ldual{\G})$ (le normalisateur dans $\Ldual{\G}$ de $C_{\gpalg*{\G}}(\phi)^{\circ}$) donc $(f\widehat{\vartheta}^{n}) \gpalg*{T}(f\widehat{\vartheta}^{n})^{-1}$ est un tore de $C_{\gpalg*{\G}}(\phi)^{\circ}$. Ainsi, il existe $c \in C_{\gpalg*{\G}}(\phi)^{\circ}$ tel que $(f\widehat{\vartheta}^{n}) \gpalg*{T}(f\widehat{\vartheta}^{n})^{-1}=c^{-1}\gpalg*{T}c$. Donc $(cf\widehat{\vartheta}^{n}) \gpalg*{T}(cf\widehat{\vartheta}^{n})^{-1}=\gpalg*{T}$, et comme $\gpalg*{T}$ est $\widehat{\vartheta}$-stable on obtient le résultat.
\end{proof}

La section $\sigma$ est déterminée par l'image de $\widehat{\vartheta}$. Le lemme \ref{lemRelevementN} nous permet de prendre un relèvement de $\sigma(\widehat{\vartheta})$ de la forme $f\widehat{\vartheta}$ avec $f \in N$. On pose alors $\bar{w}$ l'image de $f$ par l'application $N \rightarrow \Wf \rightarrow \Ws{s}\backslash \Wf$. Cette application ne dépend pas du choix du relèvement choisi. En effet si $f'\widehat{\vartheta}$ est un autre relèvement avec $f' \in N$, alors il existe $c \in C_{\gpalg*{\G}}(\phi)^{\circ}$ tel que $f'=cf$. Ainsi $c \in C_{\gpalg*{\G}}(\phi)^{\circ} \cap N$ donc $f$ et $f'$ ont même image dans $\Ws{s}\backslash \Wf$.

On vient donc d'associer à $(\phi,\sigma)$ avec $\phi(\Ild) \subseteq \gpalg*{T}$ une paire $(s,\bar{w})$. Pour arriver dans $\widetilde{X}_{\ld}$ il reste à vérifier que $s=\bar{w}\fr(s)$. Le Frobenius $\fr$ sur $\textsf{\textbf{T}}^{*}$ correspond à $\widehat{\vartheta} \circ \psi$ sur $\gpalg*{T}$. Il faut donc vérifier que $s=\Ad(f) \circ \widehat{\vartheta}(s^q)$. Prenons $\varphi \in \Lp{\weil}$ un relèvement de $\phi$. Notons $\varphi(Frob)=u\widehat{\vartheta}$. Comme $\widetilde{C}_{\gpalg*{\G}}(\phi)=C_{\gpalg*{\G}}(\phi)\varphi(\weil)$, on a que $u$ et $f$ diffèrent d'un élément de $C_{\gpalg*{\G}}(\phi)$ et donc $\Ad(f^{-1})(s)=\Ad(u^{-1})(s)$. Or pour un élément $x \in \iner/\inersauv$, $\text{Frob}^{-1} x \text{Frob} = x^q$, donc $\varphi(x)=\Ad(u) \circ \widehat{\vartheta} (\varphi(x)^q)$ et on a le résultat souhaité, en prenant pour $x$ le progénérateur de $\Ild/\inersauv$ choisit plus haut.

Montrons que le procédé précédent définit une application $\Lpbm{\Ild} \rightarrow \widetilde{X}_{\ld}/\Wf$.

\begin{Lem}
Soient $(\phi,\sigma)$ et $(\phi',\sigma')$ conjugués sous $\gpalg*{\G}$ avec $\phi(\Ild) \subseteq \gpalg*{T}$ et $\phi'(\Ild) \subseteq \gpalg*{T}$. Nommons $(s,\bar{w})$ et $(s',\bar{w}')$ les éléments de $\widetilde{X}_{\ld}$ leur étant associés respectivement. Alors il existe $v \in \Wf$ tel que $(s,\bar{w})=v \cdot (s',\bar{w}')$.
\end{Lem}

\begin{proof}
Soit $g \in \gpalg*{\G}$ tel que $s'=gsg^{-1}$ (ici on voit $s$,$s'$ comme des éléments de $\gpalg*{T}$) et $\sigma'=g\sigma g^{-1}$. Des éléments d'un tore qui sont conjugués le sont sous le groupe de Weyl, donc il existe $v \in \Wf$ tel que $s'=v \cdot s$. Appelons $n$ un relèvement de $v$ dans $N$. Comme $s'=nsn^{-1}=gsg^{-1}$, $n^{-1}g \in C_{\gpalg*{\G}}(s)$. Notons $f\widehat{\vartheta}$ (resp. $f'\widehat{\vartheta}$) un relèvement de $\sigma(\widehat{\vartheta})$ (resp. $\sigma'(\widehat{\vartheta})$) avec $f \in N$ (resp. $f' \in N$).

On a donc $f'=cgf\widehat{\vartheta}(g)^{-1}$ avec $c \in C_{\gpalg*{\G}}(\phi')^{\circ}$. Soit $c':=gn^{-1} \in C_{\gpalg*{\G}}(\phi')$ de telle sorte que $f'=cc'nf\widehat{\vartheta}(n)^{-1}\widehat{\vartheta}(c')^{-1}$. Par le lemme \ref{lemRelevementN}, il existe $h \in C_{\gpalg*{\G}}(\phi')^{\circ}$ et $m\in N$ tel que $c'=hm$. On obtient donc 
\[f'=ch[mnf\widehat{\vartheta}(n)^{-1}\widehat{\vartheta}(m)^{-1}]\widehat{\vartheta}(h)^{-1}\]
que l'on peut réécrire en 
\[hf'\widehat{\vartheta} = ch[mnf\widehat{\vartheta}(n)^{-1}\widehat{\vartheta}(m)^{-1}] \widehat{\vartheta}.\]
Nous obtenons de la sorte que $[mnf\widehat{\vartheta}(n)^{-1}\widehat{\vartheta}(m)^{-1}] \widehat{\vartheta}$ est un autre relèvement de $\sigma'(\widehat{\vartheta})$ avec $[mnf\widehat{\vartheta}(n)^{-1}\widehat{\vartheta}(m)^{-1}] \in N$. Comme $\bar{w}'$ ne dépend pas du relèvement choisi, en notant $v'$ la réduction de $mn$ dans $\Wf$, on obtient que $\bar{w}'=\overline{v'w\fr(v')^{-1}}$. Remarquons de plus que $m \in C_{\gpalg*{\G}}(\phi')$ pour obtenir $s' = v' \cdot s$. On a bien $(s',\bar{w}')=v' \cdot (s,\bar{w})$.
\end{proof}

\begin{Lem}
\label{lemPhiXtilde}
Le procédé précédent donne une bijection $\Lpbm{\Ild} \overset{\sim}{\longrightarrow} \widetilde{X}_{\ld}/\Wf$.
\end{Lem}

\begin{proof}
On a construit une application $\Lpbm{\Ild} \longrightarrow \widetilde{X}_{\ld}/\Wf$. Donnons l'application réciproque. Soit $(s,\bar{w}) \in \widetilde{X}_{\ld}$. Nous savons associer à $s$ un $\phi \in \Lpm{\Ild}$ (proposition \ref{proIdentifPhiInert}). Soit $f$ un relèvement de $\bar{w}$ dans $N$. Montrons que $f\widehat{\vartheta} \in \widetilde{C}_{\gpalg*{\G}}(\phi)$. Soit $\varphi \in \Lp{\weil}$ un relèvement de $\phi$ et notons $\varphi(\text{Frob})=u\widehat{\vartheta}$. Alors (comme dans la preuve du lemme \ref{lemRelevementN}) $s=\Ad(u) \circ \widehat{\vartheta}(s^q)$. Or $s = \bar{w}\fr(s)$, donc $s=\Ad(f) \circ \widehat{\vartheta}(s^q)$. On en déduit que $fu^{-1} \in C_{\gpalg*{\G}}(\phi)$, puis que $f\widehat{\vartheta} = (fu^{-1})(u\widehat{\vartheta})\in C_{\gpalg*{\G}}(\phi) \varphi(\weil)=\widetilde{C}_{\gpalg*{\G}}(\phi)$. On définit alors la section $\sigma$ en envoyant $\widehat{\vartheta}$ sur l'image de $f\widehat{\vartheta}$ dans $\widetilde{\pi}_{0}(\phi)$. Comme précédemment cette application passe au quotient et fournit l'application recherchée.
\end{proof}

En combinant le lemme \ref{lemPhiXtilde} et la proposition \ref{proXtildeG} on obtient:

\begin{Pro}
\label{proBijPhiSigmaRatio}
On a une bijection naturelle
\[\Lpbm{\Ild} \overset{\sim}{\longrightarrow} \ss*{\gpfini*{\G}}\]
\end{Pro}

Remarquons que cette bijection est compatible avec la proposition \ref{proIdentifPhiInert} dans le sens suivant :

\begin{Pro}
Le diagramme suivant est commutatif
\[ \xymatrix{
\Lpbm{\Ild} \ar@{->}[r]^{\sim} \ar@{->}[d] & \ss*{\gpfini*{\G}} \ar@{->}[d]\\
\Lpm{\Ild} \ar@{->}[r]^{\sim} & (\ss*{\gpfinialg*{\G}})^{\fr} }\]
où la première flèche verticale est la projection sur la première coordonnée.
\end{Pro}

On déduit des propositions \ref{prosimconj} et \ref{proBijPhiSigmaRatio} la proposition suivante

\begin{Pro}
\label{proinjsimr}
Nous avons une injection $\pairesSt/{\sim_r} \hookrightarrow \Lpbm{\Ild}$.
\end{Pro}

Nous pouvons écrire directement l'application $\pairesSt \to \Lpbm{\Ild}$ (construite pas les propositions \ref{prosimconj} et \ref{proBijPhiSigmaRatio}) grâce à la correspondance de Langlands locale pour les tores. En effet prenons $(\gpalg{S},\theta) \in \pairesSt$ et relevons $\theta$ en un caractère $\tilde{\theta}$ de $\gpalg{S}(\kk)$. Soit $\varphi : \weil \to \Ldual{S}$ le paramètre de Langlands associé à $\tilde{\theta}$ via la correspondance de Langlands locale pour les tores. Choisissons $\iota$ un plongement dual (non-canonique, voir preuve ci-dessous) $\iota : \Ldual{S} \hookrightarrow \Ldual{G}$ qui nous permet d'obtenir un paramètre de Langlands pour $\gpalg{G}$ : $\iota \circ \varphi : \weil \to \Ldual{G}$. On pose alors $\phi:=(\iota \circ \varphi)_{|\Ild}$ et $\sigma(\widehat{\vartheta}):=(\iota \circ \varphi)(\text{Frob})$.

\begin{Pro}
Le procédé précédent définit une application $\pairesSt \to  \Lpbm{\Ild}$ qui est la même que celle construite par les propositions \ref{prosimconj} et \ref{proBijPhiSigmaRatio}.
\end{Pro}

\begin{proof}
Rappelons comment est construit le plongement $\iota$. À $\gpalg{S}\subseteq \gpalg{G}$ est associé une classe de $\gpalg*{G}$-conjugaison $\widehat{\vartheta}$-équivariante de plongements $\hat{\iota} : \gpalg*{S} \hookrightarrow \gpalg*{G}$. Ainsi il existe $g \in \gpalg*{G}$ tel que $\Ad(g)\circ \widehat{\vartheta} \hat{\iota} = \hat{\iota}$ et on peut définir $\iota : \Ldual{S} \hookrightarrow \Ldual{G}$ par $\gpalg*{S} \overset{\hat{\iota}}{\hookrightarrow} \gpalg*{G}$ et $\iota(\widehat{\vartheta}_{\gpalg*{S}})=g\widehat{\vartheta}_{\gpalg*{G}}$. Montrons alors que l'application $\pairesSt \to  \Lpbm{\Ild}$ est indépendante des choix effectués. Le choix de $g$ est non-canonique mais deux $g$ diffèrent par un élément de $\gpalg*{S}$ donc donnent la même section $\sigma$ (puisque $\gpalg*{S} \subseteq C_{\gpalg*{\G}}(\phi)^{\circ}$). Enfin le procédé est indépendant du choix de $\tilde{\theta}$ puisque deux tels caractères diffèrent par un caractère non-ramifié. Et on obtient de la sorte une application bien définie $\pairesSt \to  \Lpbm{\Ild}$.

Il reste à vérifier que l'application est bien la même que celle des propositions \ref{prosimconj} et \ref{proBijPhiSigmaRatio}. Ceci est aisé à faire en utilisant la section 4.3 de \cite{DebackerReeder} qui reconstruit la correspondance de Langlands locale pour les tores dans le cas modéré avec une méthode très proche de celle de cet article.
\end{proof}

La proposition \ref{proinjsimr} combinée avec la remarque \ref{remdecompo}, démontre le théorème suivant.

\begin{The}
	\label{theDecompoPhiSigma}

	Soit $\phi : \Ild \to \Ldual{\G}$ un paramètre inertiel. La relation d'équivalence $\sim_r$ permet de décomposer $\rep[\ld][\phi]{\G}$ de la façon suivante :
	
	\[ \rep[\ld][\phi]{\G} = \prod_{ \sigma \in \overline{\Sigma}( \langle \widehat{\vartheta} \rangle ,\widetilde{\pi}_{0}(\phi))} \rep[\ld][(\phi,\sigma)]{\G}\]
\end{The}

Notons que, si $\G$ est quasi-déployé, la proposition \ref{prosimconj} donne une bijection, et donc toutes les catégories $\rep[\ld][(\phi,\sigma)]{\G}$ sont non nulles. Lorsque $\G$ n'est pas quasi-déployé, nous n'avons qu'une injection et donc, les éléments qui ne sont pas dans l'image de cette injection, donnent des catégories nulles. Nous donnerons, dans la section \ref{secRelevance}, une condition sur $(\phi,\sigma)$ pour déterminer si $\rep[\ld][(\phi,\sigma)]{\G}$ est nulle ou non.

\subsection{Lien entre \texorpdfstring{$\Zl$}{Zl} et \texorpdfstring{$\Ql$}{Ql}}
\label{seclienzlqlr}

Nous souhaitons expliciter dans cette section le lien qui existe entre les catégories construites précédemment sur $\Ql$ et sur $\Zl$.

\sautintro

Soit $(\phi,\sigma) \in \Lpbm{\iner}$. Définissons alors $\phi'=\phi_{|\inerl} \in \Lpm{\inerl}$. Alors $C_{\gpalg*{G}}(\phi) \subseteq C_{\gpalg*{G}}(\phi')$ et donc $\widetilde{C}_{\gpalg*{G}}(\phi) \subseteq \widetilde{C}_{\gpalg*{G}}(\phi')$. De plus $C_{\gpalg*{G}}(\phi)^{\circ} \subseteq C_{\gpalg*{G}}(\phi')^{\circ}$, ainsi nous avons un morphisme $\Sigma( \langle \widehat{\vartheta} \rangle ,\widetilde{\pi}_{0}(\phi)) \to \Sigma( \langle \widehat{\vartheta} \rangle ,\widetilde{\pi}_{0}(\phi'))$. Notons $\sigma'$ l'image de $\sigma$ par ce morphisme. L'application qui à $(\phi,\sigma)$ associe $(\phi',\sigma')$ est alors une application naturelle $\Lpbm{\iner} \to \Lpbm{\inerl}$.

Il est aisé de voir que l'on obtient un diagramme commutatif

\[ \xymatrix{
\Lpbm{\iner} \ar@{->}[r]^-{\sim} \ar@{->}[d] & \ss{\gpfini*{\G}} \ar@{->}[d]\\
\Lpbm{\inerl} \ar@{->}[r]^-{\sim} & \gpfini*{\G}_{ss,\Zl} },\]
où l'application $\ss{\gpfini*{\G}} \to \gpfini*{\G}_{ss,\Zl}$ revient à prendre la partie $\lprime$-régulière d'une classe de conjugaison semi-simple.

Par construction des catégories $\rep[\ld][(\phi,\sigma)]{G}$, nous avons donc la proposition suivante.

\begin{Pro}
Nous avons $\rep[\Zl][(\phi',\sigma')]{G} \cap \rep[\Ql]{G} = \prod_{(\phi,\sigma)} \rep[\Ql][(\phi,\sigma)]{G}$, où le produit est pris sur les $(\phi,\sigma) \in \Lpbm{\iner}$ s'envoyant sur $(\phi',\sigma')$ par $\Lpbm{\iner} \to \Lpbm{\inerl}$.
\end{Pro}

\subsection{Compatibilité à l'induction et à la restriction parabolique}

\label{secCompaInductionParaboliquePhiSigma}

Le but de cette section est d'étudier les propriétés de $\rep[\ld][(\phi,\sigma)]{\G}$ vis à vis des foncteurs d'induction et de restriction parabolique.

\sautintro

Soit $\gpalg{P}$ un $\kk$-sous-groupe parabolique de $\gpalg{\G}$ de quotient de Levi $\gpalg{M}$ défini sur $\kk$. Considérons $\gpalg*{M}$ un dual de $\gpalg{M}$ sur $\Ql$ muni d'un plongement $\iota:\Ldual{M} \hookrightarrow \Ldual{\G}$ (voir \cite{borel} section 3.4), qui induit une application $\Lparamm{\Ild}{M} \rightarrow \Lpm{\Ild}$. Soit $\phi_{M} \in \Lparamm{\Ild}{M}$ et notons $\phi:=\iota \circ \phi_{M} \in \Lpm{\Ild}$. Comme $\iota(\widetilde{C}_{\gpalg*{M}}(\phi_{M})) \subseteq \widetilde{C}_{\gpalg*{\G}}(\phi)$  et $\iota(C_{\gpalg*{M}}(\phi_{M}))^{\circ} \subseteq C_{\gpalg*{\G}}(\phi)^{\circ}$, le plongement $\iota$ induit un morphisme $\iota : \widetilde{\pi}_{0}(\phi_{M}) \rightarrow\widetilde{\pi}_{0}(\phi)$. En particulier, si $\sigma_{M} \in \Sigma(\langle \widehat{\vartheta} \rangle,\widetilde{\pi}_{0}(\phi_{M}))$ alors  $\sigma:=\iota \circ \sigma_{M} \in  \Sigma(\langle \widehat{\vartheta} \rangle,\widetilde{\pi}_{0}(\phi))$. On obtient de la sorte une application 
\[ \iota_{M}^{\G} : \LparamBm{\Ild}{M} \longrightarrow \Lpbm{\Ild}\]
qui à $(\phi_M,\sigma_M)$ associe $(\phi,\sigma)$.

\begin{Lem}
	\label{lemPhiSigmaM}
	Soit $t \in \textsf{M}^{*}$ un représentant de la classe de conjugaison rationnelle semi-simple associée à $(\phi_{M},\sigma_{M})$, par \ref{proBijPhiSigmaRatio}. Alors la classe de conjugaison associée à $(\phi,\sigma)$ est celle de $t$ vu comme un élément de $\gpfini*{\G}$.
\end{Lem}

\begin{proof}
	Commençons par conjuguer $(\phi_M,\sigma_M)$ pour que $\phi_M$ soit à image dans $\gpalg*{T}$. Prenons $f\widehat{\vartheta}$ un relèvement de $\sigma_M(\widehat{\vartheta})$ avec $f \in N(\gpalg*{T},\gpalg*{M})$. Notons $s\in \textsf{\textbf{T}}^{*}$ l'élément correspondant à $\phi_M$ et  $\bar{w}\in \WfM{M}/\WsM{s}{M}$ ($\WfM{M}$ étant le groupe de Weyl de $\gpalg*{M}$ relativement à $\widehat{\textbf{T}}$) la réduction de $f$, de sorte que $(s,\bar{w})$ soit un représentant de l'image de $(\phi_{M},\sigma_{M})$ par la bijection $\LparamBm{\Ild}{M} \overset{\sim}{\rightarrow} \widetilde{X}_{\ld}/\WfM{M}$. Alors $\iota(f)\widehat{\vartheta}$ est un relèvement de $\sigma$ et $\iota(f) \in N(\gpalg*{T},\gpalg*{\G})$. Donc $(s,\bar{w})$, où l'on voit $\bar{w}$ comme un élément de $\Wf/\Ws{s}$, est le couple correspondant à $(\phi,\sigma)$, d'où le résultat.
\end{proof}

\begin{The}
	\label{theInductionParaboliquephisigma}
	Soit $\gpalg{P}$ un sous-groupe parabolique de $\gpalg{\G}$ ayant pour facteur de Levi $\gpalg{M}$.
	\begin{enumerate}
		\item Soit $(\phi,\sigma) \in \Lpbm{\Ild}$. Alors
		\[ \rp[ \rep[\ld][(\phi,\sigma)]{\G} ] \subseteq \prod_{\substack{(\phi_{M},\sigma_M) \\ \iota_{M}^{\G}(\phi_M,\sigma_M)=(\phi,\sigma)} } \rep[\ld][(\phi_{M},\sigma_M)]{M}\]
		où $\rp$ désigne la restriction parabolique.
		
		\item Soit $(\phi_M,\sigma_M) \in \LparamBm{\Ild}{M}$ et notons $(\phi,\sigma):=\iota_{M}^{\G}(\phi_M,\sigma_M)$. Alors
		\[ \ip[\rep[\ld][(\phi_{M},\sigma_M)]{M} ] \subseteq \rep[\ld][(\phi,\sigma)]{\G}\]
		où $\ip$ désigne l'induction parabolique.
	\end{enumerate}
	
\end{The}

\begin{proof}
	Cela découle du lemme \ref{lemPhiSigmaM} et des propositions \ref{proInducParaSyst}, \ref{proResParaSyst}.
\end{proof}

\begin{The}
	\label{theEquivalenceparaphisigma}
	Si $C_{\gpalg*{\G}}(\phi) \subseteq \iota(\gpalg*{M})$ le foncteur d'induction parabolique $\ip$ réalise une équivalence de catégories entre $\rep[\ld][(\phi_{M},\sigma_M)]{M}$ et $\rep[\ld][(\phi,\sigma)]{\G}$.
\end{The}

\begin{proof}
	Cela découle du théorème \ref{theInductionParaboliquephisigma} et du théorème 4.4.6 dans \cite{lanard}.
\end{proof}

\subsection{Équivalence inertielle sur les paramètres de Weil}
\label{secEqinert}

Nous avons obtenu au théorème \ref{theDecompoPhiSigma} une décomposition de $\rep[\ld][\phi]{G}$ en rajoutant à $\phi$ un paramètre $\sigma$. Ici, nous allons paramétrer d'une autre manière cette décomposition. À la place de rajouter un paramètre $\sigma$ à $\phi$, on va considérer les relèvements $\lambda \in \Lp{\weil}$  de $\phi$ et les regrouper grâce à une relation d'équivalence. Pour $\ld=\Ql$, on va retrouver la relation d'équivalence inertielle introduite par Haines dans \cite{haines}. Nous introduirons également une relation de "$\lprime$-équivalence inertielle", qui imite celle de Haines, pour obtenir le cas $\ld=\Zl$. Cette nouvelle interprétation de la décomposition de $\rep[\ld][0]{\G}$, nous montrera que dans le cas où $\G$ est un groupe classique non-ramifié ($p\neq 2$) et $\ld=\Ql$ (cas où l'on a la correspondance de Langlands locale), alors la décomposition du théorème \ref{theDecompoPhiSigma} est la décomposition de $\rep[\ld][0]{\G}$ en "blocs stables".

\sautintro

Commençons par remarquer la chose suivante : nous avons identifié les paramètres de l'inertie modérés $\Lpm{\iner}[\Ql]$ aux classes de conjugaison géométriques $\fr$-stables dans $\gpfinialg*{\G}$ et $\Lpm{\inerl}[\Ql]$ à celles d'ordre premier à $\lprime$. Maintenant, si l'on considère des paramètres inertiels modérés mais à valeur dans le dual sur $\Fl$, $\phi : \iner \rightarrow \Ldual{\G}[\Fl]$, la même démonstration montre que l'on peut les identifier aux classes de conjugaison géométriques $\fr$-stables dans $\gpfinialg*{\G}$ d'ordre premier à $\lprime$. On a une identification naturelle $\Lpm{\iner}[\Fl] \simeq \Lpm{\inerl}[\Ql]$.

De même, notons $\Lpbm{\iner}[\Fl]$ l'ensemble des couples $(\phi,\sigma)$ à $\gpalg*{\G}$-conjugaison près, où $\phi \in \Lpm{\iner}[\Fl]$ et $\sigma \in \Sigma( \langle \widehat{\vartheta} \rangle ,\widetilde{\pi}_{0}(\phi))$. On a également une bijection $\Lpbm{\iner}[\Fl] \simeq \Lpbm{\inerl}[\Ql]$.

\bigskip

Fixons des données $\gpalg*{T}_{0}\subseteq \gpalg*{B}_{0} \subseteq \gpalg*{\G}$, composées d'un tore maximal et d'un Borel, stables sous l'action du groupe de Galois, qui nous permettent de définir la notion de paraboliques standards et de Levis standards de $\Ldual{\G}$, comme dans \cite{borel} paragraphes 3.3 et 3.4.  Pour $\mathcal{M}$ un Levi standard de $\Ldual{\G}$, on note $\{\mathcal{M}\}$ l'ensemble des sous-groupes de Levi standards qui sont $\gpalg*{\G}$-conjugués à $\mathcal{M}$. Remarquons que si $\mathcal{M}$ un Levi de $\Ldual{\G}$, alors $\mathcal{M}^{\circ}$ est un Levi de $\gpalg*{\G}$.

Soit $\lambda : \weil \rightarrow \Ldual{\G}$ un morphisme admissible. Alors $\lambda$ donne lieu à une unique classe de Levi standards $\{\mathcal{M}_{\lambda}\}$ telle qu'il existe $\lambda^{+}$ dans $(\lambda)_{\gpalg*{\G}}$, la classe de $\gpalg*{\G}$-conjugaison de $\lambda$, dont l'image est contenue minimalement dans $\mathcal{M}_{\lambda}$, pour un $\mathcal{M}_{\lambda}$ dans cette classe. 

Haines définit dans \cite{haines} définition 5.3.3, une notion d'équivalence inertielle pour des paramètres $\lambda : \weil \to \Ldual{G}[\Ql]$. Nous rappelons cette dernière et l'étendons en une notion de $\lprime$-équivalence inertielle pour des paramètres $\lambda : \weil \to \Ldual{G}[\Fl]$.

\begin{Def}
\label{defEqInert}
Soient $\lambda_{1},\lambda_{2} : \weil \rightarrow \Ldual{\G}[\Ql]$ (resp. $\lambda_{1},\lambda_{2} : \weil \rightarrow \Ldual{\G}[\Fl]$) deux paramètres admissibles. On dit que $\lambda_{1}$ et $\lambda_{2}$ sont \textit{inertiellement équivalents} (resp. $\lprime$-\textit{inertiellement équivalents}) si
\begin{enumerate}
\item $\{\mathcal{M}_{\lambda_{1}}\}=\{\mathcal{M}_{\lambda_{2}}\}$
\item il existe $\mathcal{M} \in \{\mathcal{M}_{\lambda_{1}}\}$, $\lambda_{1}^{+} \in (\lambda_{1})_{\gpalg*{\G}}$ et $\lambda_{2}^{+} \in (\lambda_{2})_{\gpalg*{\G}}$ dont les images sont minimalement contenues dans $\mathcal{M}$, et $z \in H^{1}(\langle \widehat{\vartheta} \rangle, (Z(\mathcal{M}^{\circ})^{\iner})^{\circ})$ vérifiant
\[ (z\lambda_{1}^{+})_{\mathcal{M}^{\circ}}=(\lambda_{2}^{+})_{\mathcal{M}^{\circ}}\]
\end{enumerate}
\end{Def}

On définit $\mathcal{B}^{st}$ (resp. $\mathcal{B}^{st}_{(\lprime)}$) comme l'ensemble des paramètres $\lambda : \weil \rightarrow \Ldual{\G}[\Ql]$ (resp. $\lambda : \weil \rightarrow \Ldual{\G}[\Fl]$) à équivalence inertielle près (resp. $\lprime$-équivalence inertielle près). À l'usuelle, on unifie les notations en posant $\mathcal{B}^{st}_{\ld}$ qui vaut $\mathcal{B}^{st}$ si $\ld = \Ql$  et $\mathcal{B}^{st}_{(\lprime)}$ si $\ld = \Zl$. On défini également $\mathcal{B}_{m,\ld}^{st}$, le sous-ensemble de $\mathcal{B}_{\ld}^{st}$ des éléments modérés.

\medskip

Le but de cette section est de construire une bijection naturelle
\[\mathcal{B}^{st}_{m,\ld} \overset{\sim}{\rightarrow} \Lpbm{\iner}[\ldb],\]
où $\ldb$ vaut $\Ql$ si $\ld=\Ql$ et $\Fl$ si $\ld=\Zl$.

Commençons par définir une application $\Lpm{\weil}[\ldb] \rightarrow \Lpbm{\iner}[\ldb]$. Soit $\varphi : \weil \to \Ldual{G}[\ldb]$. Posons $\phi:=\varphi_{|\iner} : \iner \to \Ldual{G}[\ldb]$ un paramètre modéré. Il reste à construire une section $\sigma \in \Sigma( \langle \widehat{\vartheta} \rangle ,\widetilde{\pi}_{0}(\phi))$. L'image de $\varphi$ est contenue dans $\widetilde{C}_{\gpalg*{\G}}(\phi)$, donc $\varphi : \weil \rightarrow \widetilde{C}_{\gpalg*{\G}}(\phi)$. En composant avec la projection $\widetilde{C}_{\gpalg*{\G}}(\phi) \twoheadrightarrow \widetilde{\pi}_{0}(\phi)$ on obtient une application $\weil \rightarrow \widetilde{\pi}_{0}(\phi)$. Puisque $\phi$ est modéré, on $\varphi(\iner)=\phi(\iner) \subseteq C_{\gpalg*{\G}}(\phi)^{\circ}$ et cette application passe au quotient et nous fournit $\weil/\iner \rightarrow \widetilde{\pi}_{0}(\phi)$, ou encore une section $\sigma : \langle \widehat{\vartheta} \rangle \rightarrow \widetilde{\pi}_{0}(\phi)$. Dit autrement, la section $\sigma$ est caractérisée par $\sigma(\widehat{\vartheta})=\varphi(\text{Frob}) \in \widetilde{\pi}_{0}(\phi)$. On obtient de la sorte une application $\Lpm{\weil}[\ldb] \rightarrow \Lpbm{\iner}[\ldb]$.

\begin{Lem}
\label{lemSurjParamWeil}
L'application $\Lpm{\weil}[\ldb] \rightarrow \Lpbm{\iner}[\ldb]$ est surjective.
\end{Lem}

\begin{proof}
Soit $(\phi,\sigma) \in \Lpbm{\iner}[\ldb]$. Choisissons $\varphi \in \Lp{\weil}[\ldb]$ un relèvement de $\phi$ arbitraire. Notons $\varphi(\text{Frob})=:f\widehat{\vartheta}$ et considérons $uf\widehat{\vartheta}$, un relèvement de $\sigma(\widehat{\vartheta})$ dans $\widetilde{C}_{\gpalg*{\G}}(\phi)$, avec $u \in C_{\gpalg*{\G}}(\phi)$. Pour $x\in \iner$ comme $\varphi$ est un morphisme on a $\Ad(f) \circ \widehat{\vartheta}  (\phi(x)^{q})=\phi(x)$. Ainsi $\Ad(uf) \circ \widehat{\vartheta}  (\phi(x)^{q})=\phi(x)$ et l'application $\varphi' : \weil=\iner \rtimes \langle \text{Frob} \rangle \rightarrow \Ldual{\G}$ définie par $\varphi'(x)=\phi(x)$ pour $x \in \iner$ et $\varphi'(\text{Frob})=uf\widehat{\vartheta}$ est également un morphisme. Ainsi $\varphi' \in \Lp{\weil}[\ldb]$ fournit un antécédent de $(\phi,\sigma)$.
\end{proof}

\begin{Pro}
\label{proBijStablePhiTilde}
L'application $\Lpm{\weil}[\ldb] \rightarrow \Lpbm{\iner}[\ldb]$ passe au quotient et fournit une bijection $\mathcal{B}^{st}_{m,\ld} \overset{\sim}{\rightarrow} \Lpbm{\iner}[\ldb]$.
\end{Pro}

\begin{proof}
Commençons par montrer que l'application passe au quotient. Soient $\varphi_1,\varphi_2 \in \Lpm{\weil}[\ldb]$ qui sont inertiellement équivalents. Quitte à les conjuguer, on peut alors supposer qu'il existe un Levi standard $\mathcal{M}$ contenant minimalement l'image de $\varphi_1$ et $\varphi_2$ et un $z \in H^{1}(\langle \widehat{\vartheta} \rangle, Z(\mathcal{M}^{\circ})^{\circ})$ tel que $\varphi_1=z\varphi_2$. Notons $(\phi_1,\sigma_1)$ (resp. $(\phi_2,\sigma_2)$) l'image de $\varphi_1$ (resp. $\varphi_2$) par $\Lpm{\weil}[\ldb] \rightarrow \Lpbm{\iner}[\ldb]$. Les paramètres $\varphi_1$ et $\varphi_2$ ont même restriction à $\iner$ donc $\phi_1=\phi_2$. Notons $\phi:=\phi_1=\phi_2$. Pour $i \in \{1,2\}$, $\sigma_i$ est déterminé par l'image de $\varphi_i(\text{Frob}) \in \widetilde{\pi}_{0}(\phi)$. Mais comme $Z(\mathcal{M}^{\circ})^{\circ} \subseteq C_{\gpalg*{\G}}(\phi)^{\circ}$, on a $\varphi_1(\text{Frob}) = \varphi_2(\text{Frob})$ dans $\widetilde{\pi}_{0}(\phi)$ et $\sigma_1=\sigma_2$ d'où le résultat.

On obtient donc une application $\mathcal{B}^{st}_{m,\ld} \rightarrow \Lpbm{\iner}[\ldb]$. Celle-ci est surjective d'après le lemme \ref{lemSurjParamWeil}. La preuve de l'injectivité est analogue à la preuve de la proposition C.0.2 de \cite{lanard} en remplaçant $\cent{\phi}{\gpalg*{\G}}$ par $\cent*{\phi}{\gpalg*{\G}}$.
\end{proof}

Cette bijection nous permet de réinterpréter la décomposition du théorème \ref{theDecompoPhiSigma} :

\begin{Pro}
On a la décomposition suivante
\[ \rep[\ld][0]{\G}=\prod_{ [\varphi] \in \mathcal{B}_{m,\ld}^{st}} \rep[\ld][[\varphi]]{\G} \]
\end{Pro}

Notons que les facteurs $\rep[\ld][[\varphi]]{\G}$ sont construits sans supposer l'existence de la correspondance de Langlands. Dans le cas d'un groupe classique, où la correspondance de Langlands locale est connue, alors ces facteurs sont bien compatibles avec cette dernière (ceci est démontré dans la section \ref{secCompatibiliteLanglandsStable}). On obtient alors :

\begin{The}
\label{theblocsstables}
Soit $\G$ un groupe classique non-ramifié, $\ld=\Ql$ et $p \neq 2$. Alors la décomposition
\[\rep[\Ql][0]{\G}=\prod_{ [\varphi] \in \mathcal{B}^{st}_{m}} \rep[\Ql][[\varphi]]{\G} =\prod_{(\phi,\sigma) \in \Lpbm{\iner}} \rep[\Ql][(\phi,\sigma)]{\G}\]
est la décomposition de $\rep[\ld][0]{\G}$ en "blocs stables". C'est à dire que ces facteurs correspondent à des idempotents primitifs du centre de Bernstein stable.
\end{The}

\subsection{Condition de relevance}
\label{secRelevance}

Nous avons vu que lorsque $\G$ est quasi-déployé, les sous-catégories $\rep[\ld][[\varphi]]{\G} = \rep[\ld][(\phi,\sigma)]{\G}$ sont non vides. Ce n'est plus le cas lorsque l'on retire cette hypothèse. Dans cette partie, nous souhaitons montrer que $\rep[\ld][[\varphi]]{\G}$ est non vide si et seulement si $[\varphi]$ est relevant (dans un sens précisé en dessous).

\begin{Def}
On dit qu'un paramétre de Langlands $\varphi' \in \Lparametre{G}[\ldb]$ est relevant si, lorsque l'image de $\varphi'$ est contenue dans un Levi de $\Ldual{G}[\ldb]$ alors ce dernier est relevant (au sens de \cite{borel} 3.4).

On dit que $[\varphi]$ est relevant s'il existe  $\varphi' \in \Lparametre{G}[\ldb]$ relevant tel que $\varphi'_{|\weil} \in [\varphi]$.
\end{Def}

Les preuves nécéssaires pour démontrer les résultats souhaités de cette partie sont similaires à celles de la partie 4.3 de \cite{lanard}. Nous nous appuierons donc sur ces dernières, et on ne se concentrera que sur les légères modifications.

\begin{Lem}
\label{lemFactoVarphi}
Soit $\varphi \in \Lp{\weil}[\ldb]$. Définissons $\mathcal{M}_{\varphi}:=C_{{}^{L}\mathbf{G}}(Z(C_{\widehat{\mathbf{G}}}(\phi)^{\circ})^{\varphi(W_k),\circ})$. Alors toute extension $\varphi' \in \Lparametre{G}[\ldb]$ de $\varphi$ à $\weil'$ se factorise par $\mathcal{M}_{\varphi}$. De plus il existe un $\varphi' \in \Lparametre{G}[\ldb]$ ne se factorisant par aucun sous Levi propre de $\mathcal{M}_{\varphi}$ et tel que $\varphi'_{|\weil} \in [\varphi]$. 
\end{Lem}

\begin{proof}
La première partie est montrée dans le lemme 4.3.2 de \cite{lanard}. Il ne reste qu'à modifier légèrement la preuve de la deuxième partie de ce même lemme pour obtenir le résultat souhaité. Le lemme  4.3.2 de \cite{lanard} construit à partir de $\varphi$ un paramètre $\varphi'$. Ce dernier ne se factorise par aucun sous Levi propre de $\mathcal{M}_{\varphi}$ et vérifie que $\varphi'_{|\iner} \sim \phi:= \varphi_{|\iner}$. Pour vérifier que $\varphi'_{|\weil} \in [\varphi]$, il reste à montrer, d'après la section \ref{secEqinert}, que $\varphi'(\text{Frob})$ et $\varphi(\text{Frob})$ différent par un élément de $\cent*{\phi}{\gpalg*{\G}}$. Or par définition de $\varphi'$, $\varphi'(\text{Frob})$ et $\varphi(\text{Frob})$ différent par un élément dans l'image du morphisme principal de $\sl{2}$ , $\sl{2} \to \cent*{\varphi}{\gpalg*{G}}$, d'où le résultat.
\end{proof}

\begin{The}
La sous-catégorie $\rep[\ld][[\varphi]]{\G}$ est non vide si et seulement si $[\varphi]$ est relevant.
\end{The}

\begin{proof}
Soit $t \in \ss*{\gpfini*{G}}$ une classe de conjugaison semi-simple rationnelle correspondant à $[\varphi]$. Alors $\rep[\ld][[\varphi]]{\G}$ est non vide si et seulement si $t$ est dans l'image de l'application $\pairesSt/{\sim_r} \to  \ss*{\gpfini*{G}}$ de la proposition \ref{prosimconj} si et seulement s'il existe $\gpalg{S}$ un tore non-ramifié de $\G$ en dualité avec $\gpfinialg*{S}$ tel que $t \in \gpfinialg*{S}$. Le reste de la preuve est alors quasiment identique à la preuve de la proposition 4.3.5 dans \cite{lanard}, en utilisant le nouveau lemme \ref{lemFactoVarphi}.
\end{proof}

\subsection{Classes de conjugaison rationnelles pour les groupes classiques}

\label{secConjClassique}

Le but de cette section est de préparer quelques résultats qui seront nécessaires pour la section \ref{secCompatibiliteLanglandsStable}, qui vise à montrer la compatibilité de la décomposition associée à $\sim_r$ avec la correspondance de Langlands locale. Par le théorème 4.5.6 dans \cite{lanard}, nous avons déjà la compatibilité pour la relation d'équivalence $\sim_\infty$. La relation d'équivalence $\sim_\infty$ correspond à des classes de conjugaison géométriques dans $\gpfinialg*{\G}$, alors que $\sim_r$ correspond à des classes de conjugaison rationnelles. Cette partie consiste donc à étudier les classes de conjugaison rationnelles dans une classe de conjugaison géométrique pour un groupe classique non-ramifié. Lorsque $\G$ est à centre connexe, il n'y a aucune différence, on s'intéressera donc aux groupes $\sp{2n}$, $\so{2n}$ et $\so*{2n}$ (groupe spécial orthogonal quasi-déployé, non-déployé, associé à une extension quadratique non-ramifiée) de duaux respectifs (sur le corps fini) $\so{2n+1}$, $\so{2n}$ et $\so*{2n}$. On ne s'intéressera également qu'au cas $\ld=\Ql$.

\sautintro

Prenons $s \in (\ss{\gpfinialg*{\G}})^{\fr}$ une classe de conjugaison géométrique semi-simple $\fr$-stable. Celle-ci contient au plus deux classes de conjugaison rationnelles. De plus, elle en contient exactement deux si et seulement si le centralisateur de $s$ n'est pas connexe, ce qui se produit si et seulement si $1$ et $-1$ sont tous les deux valeurs propres de $s$.

Nous souhaitons identifier les classes de conjugaison rationnelles dans une classe de conjugaison géométrique. Pour cela, nous allons définir une application $\widetilde{X}/\Wf \to \mathbb{Z}/2\mathbb{Z}$. Dans le cas où l'on s'est placé, nous avons $X \simeq \mathbb{Z}^{n}$ et le groupe de Weyl $\Wf$ s'identifie à $(\mathbb{Z}/2\mathbb{Z})^{n} \rtimes S_{n}$, dans le cas où $\gpfinialg*{\G} = \so{2n+1}$ et au noyau de l'application $(\mathbb{Z}/2\mathbb{Z})^{n} \rtimes S_{n} \rightarrow \mathbb{Z}/2\mathbb{Z}$, $((\epsilon_1,\cdots,\epsilon_n),\sigma) \mapsto \sum_{i=1}^{n} \epsilon_i$, lorsque $\gpfinialg*{\G} = \so{2n}$ ou $\so*{2n}$. Nous savons que le Frobenius agit sur $X$ par $\widehat{\vartheta} \circ \psi$ où $\psi$ est l'élévation à la puissance $q$. Ici $\widehat{\vartheta}$ est trivial pour $\so{2n+1}$ et $\so{2n}$, et est un élément d'ordre deux pour $\so*{2n}$. Pour unifier ces notations on notera $v\fr$ le Frobenius, où $v \in \Wf$ est trivial pour $\so{2n+1}$ et $\so{2n}$, et $v=((0,\cdots,0,1),id)$ pour $\so*{2n}$.

Soit $(s,w)$ où $s\in X \otimes \resalg^{\times}$ et $w \in \Wf$. On définit $f(s,w)=\sum_{i \in I_{s}} \epsilon_i \in \ZnZ{2}$ où $s=(s_1,\cdots,s_n)$, $w=((\epsilon_1,\cdots,\epsilon_n),\sigma)$ et $I_{s}=\{i\in\{1,\cdots,n\}, s_i=-1 \}$. Pour simplifier les écritures, on abrégera la notation $w=((\epsilon_1,\cdots,\epsilon_n),\sigma)$ en $w=(\epsilon_i,\sigma)$.

Rappelons que l'on note $\Ws{s}$ le groupe de Weyl de $C_{\gpfinialg*{\G}}(s)^{\circ}$ comme dans la section \ref{secClassesRatioX}.

\begin{Lem}
	\label{lemf1}
	Soit $w' \in \Ws{s}$, alors $f(s,w)=f(s,w'w)$.
\end{Lem}

\begin{proof}
Écrivons $w=(\epsilon_i,\sigma)$ et $w'=(\lambda_i,\tau)$. Notons $a_{+}$ (resp. $a_{-}$) le nombre de $1$ (resp. $-1$) dans $s=(s_1,\cdots,s_n)$. Alors $C_{\gpfinialg*{\G}}(s)^{\circ} \simeq \so{2a_{+}+1} \times \so{2a_{-}} \times \prod_{i=1}^{r} \gl{n_i}$ si $\gpfinialg*{\G}=\so{2n+1}$ et $C_{\gpfinialg*{\G}}(s)^{\circ} \simeq \so{2a_{+}} \times \so{2a_{-}} \times \prod_{i=1}^{r} \gl{n_i}$ si $\gpfinialg*{\G}=\so{2n}$ ou $\so*{2n}$. En particulier, on a dans tous les cas $\sum_{i \in I_{s}} \lambda_i = 0$.

Maintenant, $w'w=(\lambda_i,\tau) \cdot (\epsilon_i,\sigma) = (\lambda_i + \epsilon_{\tau^{-1}(i)},\tau\sigma)$. Notons que $s=w'\cdot s$ de sorte que $\tau$ permute $I_{s}$. Ainsi, $\sum_{i\in I_{s}} \lambda_i + \epsilon_{\tau^{-1}(i)} = \sum_{i\in I_{s}} \lambda_i + \sum_{i\in I_{s}} \epsilon_{\tau^{-1}(i)} = 0 + \sum_{i\in I_{s}} \epsilon_{i}$, d'où le résultat.
\end{proof}

\begin{Lem}
	\label{lemf2}
Soit $u \in \Wf$, alors $f(s,w)=f(u\cdot s, uwu^{-1})$.
\end{Lem}

\begin{proof}
Notons $w=(\epsilon_i,\sigma)$ et $u=(\lambda_i,\tau)$. Nous avons alors que $I_{u\cdot s} = \tau(I_{s})$. Calculons $uwu^{-1}$. On a $uwu^{-1}=(\lambda_i,\tau) \cdot (\epsilon_i,\sigma) \cdot (\lambda_{\tau(i)},\tau^{-1}) = (\lambda_i + \epsilon_{\tau^{-1}(i)}+\lambda_{\tau(\sigma^{-1}(\tau^{-1}(i)))},\tau\sigma\tau^{-1})$. Donc $f(u\cdot s, uwu^{-1})=\sum_{i \in I_{u \cdot s}} \lambda_i + \epsilon_{\tau^{-1}(i)}+\lambda_{\tau(\sigma^{-1}(\tau^{-1}(i)))} = \sum_{i \in I_{u \cdot s}} \lambda_i + \sum_{i \in I_{s}} \epsilon_{i} + \sum_{i \in I_{u \cdot s}} \lambda_i = \sum_{i \in I_{s}} \epsilon_{i} = f(s,w)$.
\end{proof}

Soit $(s,\bar{w	}) \in \widetilde{X}$ alors avec notre notation ($v\fr$) pour le Frobenius nous avons que pour $u \in \Wf$, $u \cdot (s,\bar{w})=(u \cdot s, uwvu^{-1}v^{-1})$. Ainsi les lemmes \ref{lemf1} et \ref{lemf2} montrent que l'on peut définir une application $\widetilde{X}/\Wf \to \ZnZ{2}$, par $(s,\bar{w}) \mapsto f(s,wv)$.

Notons que si le centralisateur de $s$ n'est pas connexe, alors pour $w' \in \Ws*{s} \setminus \Ws{s}$, les même calculs que pour le lemme \ref{lemf1} montrent que $f(s,w'w)=1+f(s,w)$. Ainsi, en notant $(\widetilde{X}/\Wf)_{s}$ la fibre de $s$ par l'application $\widetilde{X}/\Wf \to ((X\otimes_{\mathbb{Z}} \resalg^{\times})/\Wf)^{\fr}$, l'application $(\widetilde{X}/\Wf)_{s} \overset{\sim}{\to} \ZnZ{2}$ est une bijection. On notera ainsi $s_{[0]}$ et $s_{[1]}$ les deux classes de conjugaison rationnelles  contenues dans la classe de conjugaison géométrique de $s$, correspondant respectivement aux images réciproques de $0$ et $1$.

\bigskip

Nous aurons besoin par la suite d'un représentant d'une des classes rationnelles tel que les composantes de $\bar{w}$ sur les $\pm1$ soient triviales. Soit $w \in \Wf$ tel que $s=wv\fr(s)$. Écrivons $w=(\epsilon_i,\sigma)$. Alors $\sigma$ préserve les ensembles $I_s=\{ i \in \{1, \cdots ,n\}, s_i=-1\}$ et $I_s'=\{ i \in \{1, \cdots ,n\}, s_i=1\}$, ainsi $\sigma$ s'écrit comme un produit de permutations à supports disjoints $\sigma=\sigma_1 \times \sigma_{-1} \times \tau$, où $\sigma_1$ est à support dans $I_s'$ et $\sigma_{-1}$ à support dans $I_s$.  De plus $\ZnZ{2}$ agit trivialement sur les $1$ et $-1$ donc si l'on définit $\nu_i$ qui vaut $0$ si $i \in I_s \cup I_s'$ et $\epsilon_i$ sinon et que l'on pose $w'=(\nu_i,\tau) \in \Wf$ on a toujours $s=w'v\fr(s)$. Le couple $(s,\bar{w'}) \in \widetilde{X}$ correspond alors à un élément $\sd \in \gpfini*{\G}$. Notons le fait suivant, comme pour tout $i \in I_s$, $\nu_i=0$ on a que $f(s,w'v)=1$ si et seulement si $v$ est non trivial et $\tau(n) \in I_s$ si et seulement si $v$ et non trivial et $n \in I_s$. En particulier, même si $\sd$ dépend du choix de $w$, sa classe de conjugaison est bien définie indépendamment des choix effectués.

\bigskip

Soit $x \in \bts$. Dans le cas des groupes classiques nous connaissons la forme des quotients réductifs des sous-groupes parahoriques maximaux. On obtient dans les cas qui nous intéressent ici (voir par exemple \cite{StevensLust} section 2)

\[\begin{array}{|c|c|c|}
\hline \G & \quotred{\G}{x} & \quotred*{\G}{x} \\
\hline \sp{2n} & \sp{2n_1} \times \sp{2n_2} & \so{2n_1+1} \times \so{2n_2+1}\\
\hline \so{2n} & \so{2n_1} \times \so{2n_2} &  \so{2n_1} \times \so{2n_2}\\
\hline \so*{2n} &  \so{2n_1} \times \so*{2n_2} &  \so{2n_1} \times \so*{2n_2}  \\
\hline
\end{array}\]

Dans chaque cas, le groupe $\quotred{\G}{x}$ se décompose en un produit de deux groupes classiques que l'on nommera  $\quotred{\G}{x} \simeq \quotred{\G}{x,1} \times \quotred{\G}{x,2}$.

Soit $t \in \ss{\quotred*{\G}{x}}$. La classe de conjugaison $t$ correspond ainsi à $(t_1,t_2)$ où $t_i \in \ss{\quotred*{\G}{x,i}}$.  Dans la section \ref{secSimpler} nous avons obtenu la classe de $\sim_r$-équivalence de $(x,t)$ grâce à une application $\ss{\quotred*{\G}{x}} \to (\gpfini*{\G})_{ss}$. Cette application est obtenue à partir de $\widetilde{X}/\Wx{x} \to \widetilde{X}/\Wf$ et est compatible avec l'application $\widetilde{X}/\Wf \to \ZnZ{2}$ de sorte que l'on ait 

\begin{Lem}
\label{lemtsd}
Si $t=(t_{1,[i]},t_{2,[j]})$ alors sa classe de $\sim_r$-équivalence est $s_{[i+j]}$, où $s$ correspond à la classe de $\sim_{\infty}$-équivalence de $t$.
\end{Lem}

Nous avons besoin d'un dernier résultat concernant les représentations cuspidales irréductibles dans une série de Deligne-Lusztig.

\begin{Lem}
\label{lemCuspidaleUnipotente}
Soit $\gpfinialg{\G}$ un groupe réductif connexe défini sur $\res$. Soit $s \in \gpfini*{\G}_{ss}$. Alors si $\dl{\gpfini{\G}}{s}$ contient une représentation cuspidale irréductible, $\dl{(C_{\gpfinialg*{\G}}(s)^{\circ})^{\fr}}{1}$ en contient également une.
\end{Lem}

\begin{proof}
Il suffit d'appliquer la proposition 1.10 (i) de \cite{CabanesEnguehardBlocks} avec $e=1$.
\end{proof}

Soit $s \in \gpfini*{\G}$. Notons $a_{-}$ la dimension de l'espace propre associé à la valeur propre $-1$. Supposons que  $\widetilde{\mathcal{E}}(\gpfini{\G},s)$ contienne une représentation cuspidale irréductible. Alors comme $\widetilde{\mathcal{E}}(\gpfini{\G},s) = \dl{\gpfini{\G}}{s_{[0]}} \sqcup \dl{\gpfini{\G}}{s_{[1]}}$, l'une des deux séries $\dl{\gpfini{\G}}{s_{[0]}}$, $\dl{\gpfini{\G}}{s_{[1]}}$ contient une représentation cuspidale irréductible.

Si $\gpfini{\G} \neq \so*{2n}$ ou $\gpfini{\G}=\so*{2n}$ et $n \notin I_s$, alors la classe de conjugaison de $\sd$ est $s_{[0]}$. De plus $(\cent*{\sd}{\gpfinialg*{\G}})^{\fr}=\so{a_{-}}[\res] \times \gpfini{G}'$ où $\gpfinialg{\G}'$ est un groupe réductif connexe défini sur $\res$, et si l'on prend $s'_{t}$ un représentant de $s_{[1]}$ alors $(\cent*{s'_t}{\gpfinialg*{\G}})^{\fr}=\so*{a_{-}}[\res] \times \gpfini{G}''$. Le lemme \ref{lemCuspidaleUnipotente} nous dit alors que si $\dl{\gpfini{\G}}{s_{[0]}}$ (resp. $\dl{\gpfini{\G}}{s_{[1]}}$) contient une représentation cuspidale irréductible alors $\mathcal{E}(\so{a_{-}},1)$ (resp. $\mathcal{E}(\so*{a_{-}},1)$) en contient également une. Or nous savons que $\so{a_{-}}$ (resp. $\so*{a_{-}}$) contient une cuspidale irréductible unipotente si et seulement s'il existe un entier $m_{-}$ pair (resp. impair) tel que $a_{-}=2m_{-}^2$ (\cite{lusztig}, appendice "Tables of Unipotent Representations").

Maintenant si $\gpfini{\G}=\so*{2n}$ et $n \in I_s$, alors la classe de conjugaison de $\sd$ est $s_{[1]}$. Cette fois $(\cent*{\sd}{\gpfinialg*{\G}})^{\fr}=\so*{a_{-}}[\res] \times \gpfini{G}'$ et $(\cent*{s'_t}{\gpfinialg*{\G}})^{\fr}=\so{a_{-}}[\res] \times \gpfini{G}''$. Et de même, l'une des deux séries $\mathcal{E}(\so{a_{-}},1)$, $\mathcal{E}(\so*{a_{-}},1)$ contient une représentation cuspidale irréductible. Par conséquent on a encore que $a_{-}=2m_{-}^2$.

En regroupant les deux discussions précédentes, on vient de démontrer (où l'on note $[m_{-}]$ la classe d'équivalence de $m_{-}$ dans $\ZnZ{2}$)

\begin{Lem}
\label{lemCuspidalemPair}
Si $\widetilde{\mathcal{E}}(\gpfini{\G},s)$ contienne une représentation cuspidale irréductible alors cette dernière est dans $\dl{\gpfini{\G}}{s_{[m_{-}]}}$.
\end{Lem}

\subsection{Compatibilité à la correspondance de Langlands locale}

\label{secCompatibiliteLanglandsStable}

Nous supposons dans cette partie que $\G$ est un groupe classique non-ramifié et $\ld = \Ql$, de sorte que l'on ait la correspondance de Langlands locale. Nous souhaitons vérifier la compatibilité à la correspondance de Langlands de la décomposition $\rep[\Ql][0]{\G}=\prod_{[\varphi] \in \mathcal{B}^{st}_{\Ql}} \rep[\Ql][[\varphi]]{\G}$. De façon plus précise, nous souhaitons montrer que si $\pi \in \rep[\Ql][[\varphi]]{\G}$ est une représentation irréductible alors $\varphi_{\pi|\weil} \in [\varphi]$, où $\varphi_{\pi}$ est le paramètre de Langlands associé à $\pi$.

Par le théorème 4.5.6 dans \cite{lanard},  nous avons le résultat pour les paramètres inertiels. Ainsi, le cas où $C_{\gpalg*{\G}}(\phi)$ est connexe est déjà traité. Par conséquent, il ne reste qu'à traiter le cas où $\G$ est à centre non connexe. On supposera donc dans cette partie que $\G$ est l'un des groupes $\sp{2n}[\kk]$, $\so{2n}[\kk]$ ou $\so*{2n}[\kk]$. Dans le but d'utiliser les résultats de \cite{StevensLust} pour calculer les paramètres de Langlands, on supposera également que $p \neq 2$.

\sautintro

Nos groupes classiques viennent avec un plongement $\iota : \Ldual{G} \hookrightarrow{}^{L}\gl{N}$. Pour $\sp{2n}$ et $\so{2n}$ de duaux respectifs $\so{2n+1}$ et $\so{2n}$ on plonge le $\so{N}$ correspondant dans $\gl{N}$. Analysons plus précisément $\G=\so*{2n}$. Son dual est $\Ldual{\G}=\so{2n} \rtimes \langle \widehat{\vartheta} \rangle$, où $ \widehat{\vartheta}$ agit sur $\so{2n}$ par conjugaison par la matrice $\widehat{w}$ ci-dessous

\[\widehat{w} = \begin{pmatrix}
   1 &  &&&&&\\
   & \ddots &&&&& \\
   && 1&&&& \\
   &&&
      \begin{bmatrix}
       0&1\\
   	   1&0
   	   \end{bmatrix}
   &&&\\
   &&&&1&&\\
   &&&&&\ddots&\\
   &&&&&&1
\end{pmatrix}.\] 
On peut alors définir $\iota$ par $\iota(M,1)=(M,1)$ et $\iota(1,\widehat{\vartheta})=(\widehat{w},\widehat{\vartheta})$. Pour résumer, on a

\[\begin{array}{|c|c|c|}
\hline \G & \Ldual{G} & {}^{L}\gl{N} \\
\hline \sp{2n} & \so{2n+1} \times \langle \widehat{\vartheta} \rangle & \gl{2n+1} \times \langle \widehat{\vartheta} \rangle\\
\hline \so{2n} & \so{2n} \times \langle \widehat{\vartheta} \rangle & \gl{2n} \times \langle \widehat{\vartheta} \rangle \\
\hline \so*{2n} & \so{2n} \rtimes \langle \widehat{\vartheta} \rangle & \gl{2n} \times \langle \widehat{\vartheta} \rangle \\
\hline
\end{array}\]

Soit $\pi$ une représentation cuspidale irréductible. Notons $\mathcal{A}(\kk)$ l'ensemble des classes d'équivalence de représentations irréductibles cuspidales auto-duales d'un certain $\gl{m}[\kk]$. Alors pour $\rho \in \mathcal{A}(\kk)$, il existe au plus un nombre réel positif $s_{\pi}(\rho)$ tel que l'induite parabolique $i(\rho |det(\cdot)|_{\kk}^{s_{\pi}(\rho)} \otimes \pi)$ soit réductible. On définit alors l'ensemble de Jordan par $Jord(\pi)=\{(\rho,m)\in \mathcal{A}(\kk) \times \mathbb{N}^{*}, 2s_{\pi}(\rho)-(m+1) \in 2\mathbb{N}\}$. Considérons ici la version suivante du groupe de Weil-Deligne : $\weil'\simeq \weil \times \sl{2}[\Ql]$. On peut alors retrouver $\varphi_{\pi}$ à partir de l'ensemble de Jordan grâce au théorème suivant

\begin{The}[\cite{moeglin}]
On a
\[\iota \circ \varphi_{\pi}=\bigoplus_{(\rho,m) \in Jord(\pi)} \varphi_{\rho} \otimes st_{m}\]
où $\varphi_{\rho}$ est la représentation irréductible de $\weil$ correspondant à $\rho$ via la correspondance de Langlands locale pour $\gl{n}$ et $st_{m}$ est la représentation irréductible $m$-dimensionnelle de $\sl{2}[\Ql]$.
\end{The}

Il existe $x \in \bts$ et $t \in \ss{\quotred*{\G}{x}}$ tels que $e_{t,\Ql}^{\quotred{\G}{x}} \pi^{\radpara{\G}{x}} \neq 0$. Nous avons vu à la section \ref{secConjClassique} que le groupe $\quotred{\G}{x}$ se décompose en un produit de deux groupes $\quotred{\G}{x} \simeq \quotred{\G}{x,1} \times \quotred{\G}{x,2}$. Ainsi $t$ correspond via cet isomorphisme à $(t_1,t_2)$ où $t_i \in \ss{\quotred*{\G}{x,i}}$. Nous avons également défini des entiers $a^{(1)}_{-}$ et $a^{(2)}_{-}$ correspondant respectivement à la dimension de l'espace propre associé à la valeur propre $-1$ pour $t_1$ et $t_2$. Comme $\dl{\quotred{\G}{x,i}}{t_i}$ contient une représentation irréductible cuspidale, il existe des entiers $m^{(1)}_{-}$ et $m^{(2)}_{-}$ tels que $a^{(i)}_{-}=2(m^{(i)}_{-})^2$.

Soit $s$ la classe de conjugaison géométrique associée à $\phi=\varphi_{\pi|\iner}$. Supposons que $\cent{s}{\gpfinialg*{\G}}$ n'est pas connexe, alors $s$ contient, d'après la section \ref{secConjClassique}, deux classes de conjugaison rationnelles $s_{[0]}$ et $s_{[1]}$. La bijection entre $\ss{\gpfini*{\G}}$ et  $\mathcal{B}_{\Ql}^{st}$ permet de faire correspondre à $s_{[0]}$ et $s_{[1]}$ deux classes d'équivalence de $\mathcal{B}_{\Ql}^{st}$ : $\varphi_{[0]}$ et $\varphi_{[1]}$.

\begin{Lem}
\label{lemPhiDiag}
On a $\varphi_{\pi|\weil} \in \varphi_{[m_{-}^{(1)}+m_{-}^{(2)}]}$.
\end{Lem}

\begin{proof}
Prenons $ [\varphi_t] \in \mathcal{B}_{\Ql}^{st}$ correspondant à $\sd$. Nous allons examiner quand est-ce que $\varphi_{\pi|\weil} \in [\varphi_t]$. Par construction, nous avons que si $\varphi \in \cd$ alors $\varphi_{|\iner} \sim \phi$. Or $\varphi_{\pi|\iner} =\phi$ donc il ne nous reste qu'à examiner l'image de $\varphi_\pi (\text{Frob})$ dans $\widetilde{\pi}_{0}(\phi)$. On a construit $\cd$ à partir d'une paire $(s,\bar{w})$, où $w=(\epsilon_i,\sigma) \in \Wf$ vérifie que $\epsilon_i=0$ si $s_i= \pm 1$. Notons $\mathds{1}$ le caractère trivial de $\gl{1}[\kk]$, $\omega_0$ le caractère non-ramifié d'ordre deux et $\omega_1$,$\omega_2$ les caractères ramifiés d'ordre deux ($\omega_2=\omega_1\omega_0$). Ainsi $\varphi_{\pi|\weil} \in \cd$ si et seulement s'il existe $\eta \in \cent*{\phi}{\gpalg*{\G}}$ tel que $\eta \varphi_{\pi}$ ait une composante triviale sur $\mathds{1},\omega_0,\omega_1,\omega_2$. Or $\cent*{\phi}{\gpalg*{\G}} = \so{a_{+}} \times \so{a_{-}} \times \prod_{i=1}^{k} \gl{n_i}$, donc $\varphi_{\pi|\weil} \in \cd$ si et seulement si la composante de $\varphi_{\pi|\weil}$ pour $\rho = \mathds{1},\omega_{0}$ est dans $\so{a_{+}}$ et celle pour $\rho = \omega_{1},\omega_{2}$ est dans $\so{a_{-}}$. En remarquant que $\varphi_{\pi|\weil}$ est à valeur dans $\so{n}$, on obtient que $\varphi_{\pi} \in \cd$ si et seulement si la composante de $\varphi_{\pi|\weil}$ pour $\rho = \omega_{1},\omega_{2}$ est dans $\so{a_{-}}$.

Nous savons que $\iota \circ \varphi_{\pi}=\bigoplus_{(\rho,m) \in Jord(\pi)} \varphi_{\rho} \otimes st_{m}$ donc 
\[\iota \circ \varphi_{\pi|\weil}=\bigoplus_{(\rho,m) \in Jord(\pi)} \varphi_{\rho} \otimes (\nu^{(1-m)/2}\oplus \cdots \oplus \nu^{(m-1)/2})\]
où $\nu$ est trivial sur $\iner$ et $\nu(\text{Frob})=q$ (lorsque l'on écrit $\varphi_{\pi}=\bigoplus_{(\rho,m) \in Jord(\pi)} \varphi_{\rho} \otimes st_{m}$, la version du groupe de Weil-Deligne considérée est $\weil'=\weil \times \sl{2}[\Ql]$ et donc la restriction à $\weil$ devient maintenant $\varphi_{\pi|\weil}(w)=\varphi_{\pi}(w,d_w)$ où $d_w=diag(|w|^{1/2},|w|^{-1/2})$). Maintenant 
\[\{\pm s_{\pi}(\omega_1),\pm s_{\pi}(\omega_2)\} = \{ \pm (m_{-}^{(1)}+m_{-}^{(2)}+1), \pm (m_{-}^{(1)}-m_{-}^{(2)})\}\]
d'après \cite{StevensLust} section 8, et donc $s_{\pi}(\omega_1)$ et $s_{\pi}(\omega_2)$ sont de même parité. En particulier on obtient que la composante de $\iota \circ \varphi_{\pi|\weil}$ pour $\rho = \omega_{1},\omega_{2}$ est dans $\so{a_{-}}$ si et seulement si $s_{\pi}(\omega_1)$ est pair si et seulement si $m_{-}^{(1)}+m_{-}^{(2)}$ est pair.

Si $\G \neq \so*{2n}$ ou $\G = \so*{2n}$ et $n \notin I_s$, alors $[\varphi_t]=\varphi_{[0]}$ et $\iota$ ne change pas la composante  de $\varphi_\pi$ en $\rho = \omega_{1},\omega_{2}$. Donc $\varphi_{\pi|\weil} \in \varphi_{[0]}$ si et seulement si $m_{-}^{(1)}+m_{-}^{(2)}$ est pair, d'où le résultat.

Si $\G = \so*{2n}$ et $n \in I_s$, alors $[\varphi_t]=\varphi_{[1]}$ et $\iota$ multiplie la composante  de $\varphi_\pi$ en $\rho = \omega_{1},\omega_{2}$ par $\widehat{w}$ (de déterminant $-1$). Donc $\varphi_{\pi|\weil} \in \varphi_{[1]}$ si et seulement si $m_{-}^{(1)}+m_{-}^{(2)}$ est impair, d'où le résultat.
\end{proof}

\begin{Pro}
\label{proCompatibiliteLanglands}
Soit $\pi \in \rep[\Ql][[\varphi]]{\G}$ une représentation cuspidale irréductible. Alors $\varphi_{\pi|\weil} \in [\varphi]$.
\end{Pro}

\begin{proof}

On peut supposer le centralisateur non connexe, le cas connexe est traité dans \cite{lanard} théorème 4.5.6.

Soit $\pi \in \rep[\Ql][[\varphi]]{\G}$ une représentation cuspidale irréductible. Il existe $x \in \bts$ et $t \in \ss{\quotred*{\G}{x}}$ tels que $e_{t,\Ql}^{\quotred{\G}{x}} \pi^{\radpara{\G}{x}} \neq 0$. La classe $t$ s'écrit $(t_1,t_2)$. Comme $\dl{\quotred{\G}{x,i}}{t_i}$ contient une représentation irréductible cuspidale, le lemme \ref{lemCuspidalemPair} nous dit que $t_i=t_{[m^{(i)}_{-}]}$. Notons $s$ la classe de conjugaison rationnelle correspondant à $[\varphi]$, alors le lemme \ref{lemtsd} nous dit alors que $s=s_{[m_{-}^{(1)}+m_{-}^{(2)}]}$. On conclut alors par le lemme \ref{lemPhiDiag} qui nous dit que $\varphi_{\pi|\weil} \in \varphi_{[m_{-}^{(1)}+m_{-}^{(2)}]}$.
\end{proof}

\begin{The}
\label{thecompatibilitelanglandsphisigma}
Soit $\G$ un groupe classique non-ramifié, $\ld=\Ql$ et $p \neq 2$. Alors la décomposition $\rep[\Ql][0]{\G}=\prod_{[\varphi] \in \mathcal{B}^{st}_{\Ql}} \rep[\Ql][[\varphi]]{\G}$ est compatible à la correspondance de Langlands locale. C'est-à-dire que si $\pi \in \rep[\Ql][[\varphi]]{\G}$ est une représentation irréductible, alors $\varphi_{\pi|\weil} \in [\varphi]$, où $\varphi_{\pi}$ est le paramètre de Langlands associé à $\pi$.
\end{The}

\begin{proof}
Il existe un Levi $M$ et une représentation irréductible cuspidale $\tau$ telle que $\pi$ soit une sous-représentation de $\ip (\tau)$. La proposition \ref{proCompatibiliteLanglands} nous donne le résultat pour $\tau$. Mais comme la correspondance de Langlands locale pour les groupes classiques est compatible à l'induction parabolique (pour les paramètres de Weil) et qu'il en est de même pour la décomposition $\rep[\Ql][0]{\G}=\prod_{[\varphi] \in \mathcal{B}^{st}_{\Ql}} \rep[\Ql][[\varphi]]{\G}$ par le théorème \ref{theInductionParaboliquephisigma} , on obtient le résultat pour $\pi$.
\end{proof}

\section{La relation d'équivalence \texorpdfstring{$\sim_{e}$}{~e}}

\label{secsime}

Nous venons d'étudier les relations d'équivalence $\sim_\infty$ et $\sim_r$ qui nous ont fourni des décompositions de $\rep[\ld][0]{G}$. Passons maintenant à l'étude de la dernière relation : $\sim_e$. Nous savons que cette relation permet de construire les systèmes cohérents minimaux et donc d'obtenir la décomposition de $\rep[\ld][0]{G}$ la plus fine que l'on puisse obtenir avec notre méthode. Pour paramétrer les classes de $\sim_e$-équivalence, nous allons avoir besoin de rajouter aux paires $(\phi,\sigma)$ une certaine donnée cohomologique $\alpha$. On pourra résumer l'ensemble des correspondances de cet article par le diagramme suivant
\[ \xymatrix{
\sim_{e} \ar@{=>}[r] \ar@{<->}[d]& \sim_{r} \ar@{=>}[r] \ar@{<->}[d]& \sim_{\infty} \ar@{<->}[d]\\
(\phi,\sigma,\alpha) & (\phi,\sigma) & \phi}\]

On suppose toujours que $\gpalg{G}$ se déploie sur une extension non-ramifiée de $\kk$. Cependant, à partir de la section \ref{secSthetaQuasidep}, les résultats ne s'appliqueront qu'aux formes intérieures pures des groupes non-ramifiés. On supposera donc, à partir de ce moment, que $\gpalg{G}$ est non-ramifié et on prendra $\omega \in H^1(\kk,\gpalg{G})$, qui correspond à $G_{\omega}$, une forme intérieure pure de $\G$. L'isomorphisme de Kottwitz nous permet d'identifier $\omega$ à un élément $\omega \in \Irr[\pi_{0}(Z(\gpalg*{\G})^{\widehat{\vartheta}})]$. On associera par la suite à $(\phi,\sigma) \in \Lpbm{\Ild}$, une application
\[\hps : \Irr[\pi_{0}(Z(C_{\gpalg*{\G}}(\phi)^{\circ})^{\sigma(\widehat{\vartheta})})]/{\pi_{0}(\phi)^{\sigma(\widehat{\vartheta})}} \to \Irr[\pi_{0}(Z(\gpalg*{\G})^{\widehat{\vartheta}})].\]
La décomposition associée à $\sim_e$ s'écrira alors
\[\rep[\ld][(\phi,\sigma)]{G_{\omega}} = \prod_{\alpha \in \hps^{-1}(\omega)} \rep[\ld][(\phi,\sigma,\alpha)]{G_{\omega}} \]

\subsection{Paramétrisation des classes de \texorpdfstring{$\sim_{e}$}{e}-équivalence dans une classe de \texorpdfstring{$\sim_{r}$}{r}-équivalence}

\label{secParamsime}

Nous avons décrit dans la section \ref{secSimr} les classes de $\sim_r$-équivalence. Pour paramétrer les classes de $\sim_e$-équivalence nous cherchons donc à décrire $\clr/{\sim_e}$, où $\clr$ est la classe de $\sim_r$ équivalence d'une paire $(\gpalg{S},\theta) \in \pairesSt$. De façons plus précise, nous allons construire une bijection $\clr/{\sim_{e}} \overset{\sim}{\longrightarrow} \ker[\tilde{H}^{1}(\fr,\WSte/\WSta) \rightarrow H^{1}(\fr,\We/\Wa)]$ (on rappelle la définition de $H^1$ dans l'annexe \ref{secCohomologie}). 

\sautintro

 Prenons donc $(\mathbf{S},\theta) \in \pairesSt$ et notons $\clr$ sa classe de $\sim_r$-équivalence. Définissons $\tilde{H}^{1}(\fr,\NSte)$ l'image de $H^{1}(\fr,\NSte)$ dans $H^{1}(\fr,\NSte*)$. Remarquons que $\tilde{H}^{1}(\fr,\NSte)$ n'est autre que $Z^{1}(\fr,\NSte)/{\sim}$, où $n\sim n'$ s'il existe $g \in \NSte*$ tel que $n=gn'\fr(g)^{-1}$.

\begin{Pro}
\label{proCphiT}
On a une bijection
\[\clr/{\sim_{\G}} \overset{\sim}{\longrightarrow} \ker[\tilde{H}^{1}(\fr,\NSte) \rightarrow H^{1}(\fr,\G^{nr})].\]
\end{Pro}

\begin{proof}
Soit $(\mathbf{S}',\theta') \in \clr$. Alors il existe $g \in \G^{nr}$ et $m \in \mathbb{N}^{*}$ tels que  ${}^{g} (\mathbf{S}^{\fr^{m}})=\mathbf{S}'^{\fr^{m}}$, $g \theta\langle m \rangle = \theta'\langle m \rangle$ et $g^{-1}\fr(g) \in \NSte$. L'élément $g^{-1}\fr(g)$ définit un élément de $\ker[\tilde{H}^{1}(\fr,\NSte) \rightarrow H^{1}(\fr,\G^{nr})]$. Montrons que cette construction est indépendante du choix de $g$. Soit $g'$ un autre élément de $\G^{nr}$ vérifiant les mêmes conditions que $g$. Alors $g^{-1}g' \in \NSte*$ et donc $[g^{-1}\fr(g)]=[g'^{-1}\fr(g')]$ dans $\tilde{H}^{1}(\fr,\NSte)$. La construction précédente ne dépend pas de la classe de $\G$-conjugaison choisie et on obtient une application
\[\clr/{\sim_{\G}}  \rightarrow \ker[\tilde{H}^{1}(\fr,\NSte) \rightarrow H^{1}(\fr,\G^{nr})].\]

Cette application est injective. Prenons $(\mathbf{S}_{1},\theta_{1})$ et $(\mathbf{S}_{2},\theta_{2})$ deux éléments de $\clr$. Soient $g_{1},g_{2} \in \G^{nr}$ associés à $(\mathbf{S}_{1},\theta_{1})$ et $(\mathbf{S}_{2},\theta_{2})$ et vérifiant que $[g_{1}^{-1}\fr(g_{1})]=[g_{2}^{-1}\fr(g_{2})]$ dans $\tilde{H}^{1}(\fr,\NSte)$. Il existe alors $h \in \NSte*$ tel que $g_{1}^{-1}\fr(g_{1})=hg_{2}^{-1}\fr(g_{2})\fr(h)^{-1}$. L'élément $g_{1}hg_{2}^{-1}$ est donc dans $\G$. Ainsi, quitte à conjuguer $(\mathbf{S}_{1},\theta_{1})$ par un élément de $\G$, on peut supposer que $g_{2}=g_{1}h$. Comme $h \in \NSte*$, $\mathbf{S}_{1}=\mathbf{S}_{2}$ et $\theta_{1}\langle m \rangle = \theta_{2} \langle m \rangle$. Mais comme l'application $\Tr_{\fr^m/\fr}$ est injective, $\theta_{1}=\theta_{2}$.

Elle est également surjective par le lemme \ref{lemgStheta}, d'où le résultat.
\end{proof}

\begin{Lem}
\label{lemH1NW}
L'application $N^{nr}\twoheadrightarrow \We$ induit un isomorphisme $H^{1}(\fr,N^{nr}) \overset{\sim}{\rightarrow} H^{1}(\fr,\We)$ (et de même en remplaçant respectivement $N^{nr}$ par $\NSte*$, $\NSte$ et $\We$ par $\WSte*$, $\WSte$).
\end{Lem}

\begin{proof}
La preuve est analogue à celle du lemme 2.3.4 de \cite{DebackerReeder}.
\end{proof}

Notons $\tilde{H}^{1}(\fr,\WSte)$ l'image de $H^{1}(\fr,\WSte)$ dans $H^{1}(\fr,\WSte*)$.

\begin{Cor}
\label{corCphiH1}
On a une bijection
\[ \clr/{\sim_{\G}} \overset{\sim}{\longrightarrow} \ker[\tilde{H}^{1}(\fr,\WSte) \rightarrow H^{1}(\fr,\We/\Wa)].\]
\end{Cor}

\begin{proof}
D'après la proposition \ref{proCommuteNW} on a le diagramme commutatif

\[ \xymatrix{
H^{1}(\fr,\NSte) \ar@{->}[r] \ar@{->}[d]^{\sim}& H^{1}(\fr,\NSte*) \ar@{->}[r] \ar@{->}[d]^{\sim}& H^{1}(\fr,N^{nr}) \ar@{->}[r] \ar@{->}[d]^{\sim}& H^{1}(\fr,\G) \ar@{->}[d]^{\sim}\\
H^{1}(\fr,\WSte) \ar@{->}[r] & H^{1}(\fr,\WSte*) \ar@{->}[r] & H^{1}(\fr,\We) \ar@{->}[r] & H^{1}(\fr,\We/\Wa)} \]
d'où le résultat par la proposition \ref{proCphiT}.
\end{proof}

De même l'image de $H^{1}(\fr,\WSte/\WSta)$ dans $H^{1}(\fr,\WSte*/\WSta)$ sera notée $\tilde{H}^{1}(\fr,\WSte/\WSta)$.

\begin{Pro}
\label{proSimeSimr}
On a une bijection
\[\clr/{\sim_{e}} \overset{\sim}{\longrightarrow} \ker[\tilde{H}^{1}(\fr,\WSte/\WSta) \rightarrow H^{1}(\fr,\We/\Wa)].\]
Et l'application, qui à un élément de $\clr$ associe sa classe de $\sim_{e}$-équivalence, est donnée par l'application naturelle $\tilde{H}^{1}(\fr,\WSte) \rightarrow \tilde{H}^{1}(\fr,\WSte/\WSta)$.
\end{Pro}

\begin{proof}

Commençons par montrer que l'application
\[H^{1}(\fr,\WSte) \rightarrow H^{1}(\fr,\WSte/\WSta)\]
est surjective. Notons que si $(\mathbf{S}',\theta')$ est $\sim_{r}$-équivalente à $(\mathbf{S},\theta)$ (les deux paires étant reliées par un $g$ vérifiant $g^{-1}\fr(g) \in \NSte$), alors comme dans l'annexe \ref{secCohomologie} l'application $x \mapsto gx\fr(g)^{-1}$ induit des bijections faisant commuter le diagramme
\begin{equation}
\tag{*}
\label{diagram}
\xymatrix{
H^{1}(\fr,\WSte) \ar@{->}[r] \ar@{->}[d]^{\sim} & H^{1}(\fr,\WSte/\WSta) \ar@{->}[d]^{\sim}\\
H^{1}(\fr,\WStecomplet{\gpalg{S}'}{\theta'}) \ar@{->}[r] & H^{1}(\fr,\WStecomplet{\gpalg{S}'}{\theta'}/\WStacomplet{\gpalg{S}'}{\theta'})}
\end{equation}

Il nous suffit donc de montrer la surjectivité de
\[H^{1}(\fr,\WStecomplet{\gpalg{S}'}{\theta'}) \rightarrow H^{1}(\fr,\WStecomplet{\gpalg{S}'}{\theta'}/\WStacomplet{\gpalg{S}'}{\theta'}).\]
D'après le lemme \ref{lemChambreStableSime}, il existe une paire $(\mathbf{S}',\theta')$ $\sim_{r}$-équivalente à $(\mathbf{S},\theta)$ telle que $\mathcal{A}_{(\mathbf{S}',\theta')}$ contienne $C_{(\mathbf{S}',\theta')}$ une chambre $\fr$-stable. Notons $\Omega_{(\mathbf{S}',\theta')}:=\{w \in \WStecomplet{\gpalg{S}'}{\theta'}, w\cdot C_{(\mathbf{S}',\theta')}=C_{(\mathbf{S}',\theta')}\}$. L'isomorphisme $\Omega_{(\mathbf{S}',\theta')} \overset{\sim}{\rightarrow} \WStecomplet{\gpalg{S}'}{\theta'}/\WStacomplet{\gpalg{S}'}{\theta'}$ nous fournit alors une section $\fr$-équivariante à la suite exacte
\[1 \rightarrow \WStacomplet{\gpalg{S}'}{\theta'} \rightarrow \WStecomplet{\gpalg{S}'}{\theta'} \rightarrow \WStecomplet{\gpalg{S}'}{\theta'}/\WStacomplet{\gpalg{S}'}{\theta'} \rightarrow 1.\]
En particulier l'application voulue est bien surjective.

En utilisant le corollaire \ref{corCphiH1}, on obtient une surjection
\[\clr/{\sim_{\G}} \twoheadrightarrow \ker[\tilde{H}^{1}(\fr,\WSte/\WSta) \rightarrow H^{1}(\fr,\We/\Wa)].\]
Il nous reste à montrer que cette application passe au quotient pour la $\sim_{e}$-équivalence et que l'application quotientée est injective. C'est à dire, il faut montrer que si $(\mathbf{S}_{1},\theta_{1}),(\mathbf{S}_{2},\theta_{2}) \in \clr$ alors $(\mathbf{S}_{1},\theta_{1}) \sim_e (\mathbf{S}_{2},\theta_{2})$ si et seulement si $(\mathbf{S}_{1},\theta_{1})$ et $(\mathbf{S}_{2},\theta_{2})$ ont même image dans $\tilde{H}^{1}(\fr,\WSte/\WSta)$. Or le diagramme commutatif \eqref{diagram} (appliqué avec $(\mathbf{S}_{1},\theta_{1})$ à la place de $(\gpalg{S}',\theta')$) permet de nous ramener au cas où $(\mathbf{S}_{1},\theta_{1})=(\mathbf{S},\theta)$. Supposons donc que $(\mathbf{S}_{2},\theta_{2}) \sim_{e}(\mathbf{S},\theta)$, il existe alors $g$ reliant les deux paires tel que $g^{-1}\fr(g) \in \NSta$ et donc l'image de  $(\mathbf{S}_{2},\theta_{2})$ est triviale dans $\tilde{H}^{1}(\fr,\WSte/\WSta)$. Réciproquement, supposons que l'image de  $(\mathbf{S}_{2},\theta_{2})$ est triviale dans $\tilde{H}^{1}(\fr,\WSte/\WSta)$ c'est à dire que si $g$ relie $(\mathbf{S}_{2},\theta_{2})$ et $(\mathbf{S},\theta)$, et que $\bar{w}$ est l'image de $g^{-1}\fr(g)$ dans $\WSte/\WSta$ alors $[\bar{w}]=[1]$ dans $\tilde{H}^{1}(\fr,\WSte/\WSta)$. Ainsi quitte à multiplier $g$ par un élément de $\NSte*$ on peut supposer que $\bar{w}=1$, c'est à dire que $g^{-1}\fr(g) \in \NSta$ et donc que $(\mathbf{S}_{2},\theta_{2}) \sim_{e}(\mathbf{S},\theta)$, ce qui achève la preuve de la proposition.
\end{proof}

Notons que cette proposition dépend du choix d'un point base $(\gpalg{S},\theta)$. Cependant; si $(\gpalg{S}',\theta') \in \pairesSt$ est une autre paire telle que $(\gpalg{S},\theta) \sim_e (\gpalg{S}',\theta')$. Prenons $g \in G^{nr}$ reliant  $(\gpalg{S},\theta)$ et $(\gpalg{S}',\theta')$ tel que $g^{-1}\fr(g) \in \NSta$, nous avons alors le diagramme commutatif suivant
\[
\xymatrix{
\clr/{\sim_{e}} \ar@{->}[r]^-{\sim} \ar@{=}[d] & \ker[\tilde{H}^{1}(\fr,\WSte/\WSta) \rightarrow H^{1}(\fr,\We(\gpalg{S})/\Wa(\gpalg{S}))] \ar@{->}[d]^{x \mapsto gx\fr(g)^{-1}}\\
[\gpalg{S}',\theta']_r/{\sim_{e}} \ar@{->}[r]^-{\sim} & \ker[\tilde{H}^{1}(\fr,\WStecomplet{\gpalg{S}'}{\theta'}/\WStacomplet{\gpalg{S}'}{\theta'}) \rightarrow H^{1}(\fr,\We(\gpalg{S}')/\Wa(\gpalg{S}'))]}
\]

\subsection{Le cas quasi-déployé}

\label{secSthetaQuasidep}

On suppose dans cette partie que $\gpalg{G}$ est non-ramifié (c'est-à-dire que l'on suppose en plus que $\gpalg{G}$ est quasi-déployé). Nous souhaitons réinterpréter sur le dual les groupes de cohomologies apparus dans la section \ref{secParamsime} dans le cas où le groupe est non-ramifié. En particulier, nous allons construire un isomorphisme
\[ \Irr[\pi_{0}(Z(C_{\gpalg*{\G}}(\phi)^{\circ})^{\sigma(\widehat{\vartheta})})]/{\pi_{0}(\phi)^{\sigma(\widehat{\vartheta})}} \overset{\sim}{\longrightarrow} \tilde{H}^{1}(\fr,\WSte/\WSta).\]
\sautintro

Soit $\phi : \Ild \to \Ldual{\G}$ un paramètre inertiel. La conjugaison par n'importe quelle section ensembliste de la suite exacte $C_{\gpalg*{\G}}(\phi)^{\circ} \hookrightarrow \widetilde{C}_{\gpalg*{\G}}(\phi) \twoheadrightarrow \widetilde{\pi}_{0}(\phi)$ donne une action "extérieure" bien définie
\[ \widetilde{\pi}_{0}(\phi) \longrightarrow \Out(C_{\gpalg*{\G}}(\phi)^{\circ}).\]
Alors n'importe quelle section continue $\sigma : \langle \widehat{\vartheta} \rangle \rightarrow \widetilde{\pi}_{0}(\phi)$ nous donne par composition avec l'action précédente, un morphisme de $\langle \widehat{\vartheta} \rangle$ dans $\Out(C_{\gpalg*{\G}}(\phi)^{\circ})$, donc sur la donnée radicielle de $C_{\gpalg*{\G}}(\phi)^{\circ}$.

\bigskip

L'inclusion $Z(\gpalg*{\G})^{\widehat{\vartheta}} \subseteq Z(C_{\gpalg*{\G}}(\phi)^{\circ})^{\widetilde{\pi}_{0}(\phi)} \subseteq Z(C_{\gpalg*{\G}}(\phi)^{\circ})^{\sigma(\widehat{\vartheta})}$ nous fournit alors une application $\Irr[\pi_{0}(Z(C_{\gpalg*{\G}}(\phi)^{\circ})^{\sigma(\widehat{\vartheta})})] \to \Irr[\pi_{0}(Z(\gpalg*{\G})^{\widehat{\vartheta}})]$. Le groupe $\pi_{0}(\phi)^{\sigma(\widehat{\vartheta})}$ agit sur $\Irr[\pi_{0}(Z(C_{\gpalg*{\G}}(\phi)^{\circ})^{\sigma(\widehat{\vartheta})})]$ et l'application précédente se factorise par le quotient car
\[Z(C_{\gpalg*{\G}}(\phi)^{\circ})^{\widetilde{\pi}_{0}(\phi)} \subseteq Z(C_{\gpalg*{\G}}(\phi)^{\circ})^{\sigma(\widehat{\vartheta})\pi_{0}(\phi)^{\sigma(\widehat{\vartheta})}} \subseteq Z(C_{\gpalg*{\G}}(\phi)^{\circ})^{\sigma(\widehat{\vartheta})}.\]
Ceci nous permet de définir une application $\hps$ de la façon suivante

\begin{Def} On nomme $\hps$ l'application
\[\hps : \Irr[\pi_{0}(Z(C_{\gpalg*{\G}}(\phi)^{\circ})^{\sigma(\widehat{\vartheta})})]/{\pi_{0}(\phi)^{\sigma(\widehat{\vartheta})}} \to \Irr[\pi_{0}(Z(\gpalg*{\G})^{\widehat{\vartheta}})].\]
\end{Def}

La classe de $\sim_r$-équivalence du couple $(\phi,\sigma)$ est déterminée par une classe de conjugaison rationnelle de $t$ dans $\gpfini*{\G}$. De plus par construction, nous obtenons le lemme suivant.

\begin{Lem}
\label{lemdonneradicielle}
Soit $t \in \ss*{\gpfini*{G}}$ l'image de $(\phi,\sigma)$ par l'application de la proposition \ref{proBijPhiSigmaRatio}. Alors la donnée radicielle de $C_{\gpalg*{\G}}(\phi)^{\circ}$ munie de l'action de $\widehat{\vartheta}$ est égale à celle de $C_{\gpfinialg*{G}}(t)^{\circ}$ également munie de l'action du Frobenius.
\end{Lem}

Prenons $\textsf{\textbf{S}}^{*}$ un tore maximalement déployé dans $C_{\gpfinialg*{\G}}(t)^{\circ}$. Fixons un sommet hyperspécial $o$ dans $\bt$ ainsi qu'un isomorphisme entre $\quotredalg*{\G}{o}$ et $\gpfinialg*{\G}$. Ainsi $\gpfinialg*{S}$ devient un tore maximal $\fr$-stable de $\quotredalg*{\G}{o}$ et nous pouvons choisir $\gpfinialg{S}$ un tore maximal $\fr$-stable de $\quotredalg{\G}{o}$ en dualité avec $\gpfinialg*{S}$ sur $\res$. Ce dernier se relève en un tore non-ramifié $\mathbf{S}$ de $\gpalg{\G}$. La bijection $(\gpfinialg*{S})^{\fr} \tosim X^{*}(\textbf{S})/(\fr_{\mathbf{S}} -1)X^{*}(\textbf{S})$ (voir section \ref{secintroStheta}) envoie $t$ sur un élément que l'on nomme $\theta$. Nous obtenons de la sorte une paire $(\gpalg{S},\theta) \in \pairesSt$. Notons que cette paire est bien définie à $\sim_e$-équivalence près (une fois le sommet hyperspecial $o$ fixé) puisque toutes les paires obtenues par le procédé précédent induisent un système cohérent minimal qui contient le couple $(o,t)$ (ces systèmes sont donc tous égaux) et sont donc $\sim_e$-équivalentes. Notons également que par construction, $(\mathbf{S},\theta)$ est un représentant de la $\sim_r$-classe que l'on avait choisie.

Notons $X= X_{*}(\gpalg{S})$, $X^{a}(\mathbf{S},\theta):=X \cap \WSta$ et $\bar{X}(\mathbf{S},\theta):=X/X^{a}(\mathbf{S},\theta)$.

\begin{Lem}
\label{lemracineeng}
Le sous-groupe de $X$ engendré par les racines de $C_{\gpalg*{\G}}(\phi)^{\circ}$ est $X^{a}(\mathbf{S},\theta)$.
\end{Lem}

\begin{proof}
Par le lemme \ref{lemdonneradicielle}, la donnée radicielle de $C_{\gpalg*{\G}}(\phi)^{\circ}$ est égale à celle de $C_{\gpfinialg*{G}}(t)^{\circ}$. Or le lemme \ref{lemIdentifWeyl} construit un isomorphisme entre $\WStx{o}$ et $\Wf(\gpfinialg*{S},C_{\quotredalg*{\G}{\sigma}}(t)^{\circ})$ et sa preuve montre qu'il préserve également les systèmes de racines. Le sommet $o$ étant hyperspécial, nous avons un isomorphisme $\WStx{o} \tosim \WStf$, ce qui nous donne le résultat.
\end{proof}

\begin{Pro}
\label{procommirrW}

Nous avons une bijection
\[ \Irr[\pi_{0}(Z(C_{\gpalg*{\G}}(\phi)^{\circ})^{\sigma(\widehat{\vartheta})})]/{\pi_{0}(\phi)^{\sigma(\widehat{\vartheta})}} \overset{\sim}{\longrightarrow} \tilde{H}^{1}(\fr,\WSte/\WSta),\]
rendant commutatif le diagramme suivant
\[ \xymatrix{
\Irr[\pi_{0}(Z(C_{\gpalg*{\G}}(\phi)^{\circ})^{\sigma(\widehat{\vartheta})})]/{\pi_{0}(\phi)^{\sigma(\widehat{\vartheta})}} \ar@{->}[r]^-{\hps} \ar@{->}[d]^{\sim}& \Irr[\pi_{0}(Z(\gpalg*{\G})^{\widehat{\vartheta}})] \ar@{->}[d]^{\sim}\\
\tilde{H}^{1}(\fr,\WSte/\WSta)\ar@{->}[r] & H^{1}(\fr,\We/\Wa)}\]
où l'isomorphisme vertical de droite est celui de l'annexe \ref{secCohomologie}.
\end{Pro}

\begin{proof}
Par le lemme \ref{lemracineeng}, on peut identifier par restriction des caractères, comme dans la section 2.5 de \cite{DebackerReeder}, $\bar{X}(\mathbf{S},\theta)=\Hom(Z(C_{\gpalg*{\G}}(\phi)),\mathbb{C}^{\times})$. De plus, les actions de Frobenius étant compatibles par le lemme \ref{lemdonneradicielle}, $\bar{X}(\mathbf{S},\theta)/(1-\vartheta)\bar{X}(\mathbf{S},\theta) = \Hom(Z(C_{\gpalg*{\G}}(\phi))^{\sigma(\widehat{\vartheta})}),\mathbb{C}^{\times})$ et les éléments de $\bar{X}(\mathbf{S},\theta)/(1-\vartheta)\bar{X}(\mathbf{S},\theta)$ s'annulant sur la composante identité de $Z(C_{\gpalg*{\G}}(\phi))^{\sigma(\widehat{\vartheta})})$ sont exactement les éléments de torsion. D'où une identification
\[[\bar{X}(\mathbf{S},\theta)/(1-\vartheta)\bar{X}(\mathbf{S},\theta)]_{\text{tor}} = \Irr[\pi_{0}(Z(C_{\gpalg*{\G}}(\phi)^{\circ})^{\sigma(\widehat{\vartheta})})].\]

Et on obtient un diagramme commutatif
\[ \xymatrix{
\Irr[\pi_{0}(Z(C_{\gpalg*{\G}}(\phi)^{\circ})^{\sigma(\widehat{\vartheta})})] \ar@{->}[r] \ar@{->}[d]^{\sim}& \Irr[\pi_{0}(Z(\gpalg*{\G})^{\widehat{\vartheta}})] \ar@{->}[d]^{\sim}\\
[\bar{X}(\mathbf{S},\theta)/(1-\vartheta)\bar{X}(\mathbf{S},\theta)]_{\text{tor}} \ar@{->}[r] & [\bar{X}/(1-\vartheta)\bar{X}]_{\text{tor}}}.\]

Maintenant le Corollaire 2.4.2 de \cite{DebackerReeder} nous fournit un isomorphisme $[\bar{X}/(1-\vartheta)\bar{X}]_{\text{tor}} \overset{\sim}{\longrightarrow} H^{1}(\fr,\We/\Wa)$. De la même façon nous obtenons un isomorphisme $\bar{X}(\mathbf{S},\theta)/(1-\vartheta)\bar{X}(\mathbf{S},\theta)]_{\text{tor}} \overset{\sim}{\longrightarrow} H^{1}(\fr,\WSte/\WSta)$, et donc un diagramme commutatif

\[ \xymatrix{
\Irr[\pi_{0}(Z(C_{\gpalg*{\G}}(\phi)^{\circ})^{\sigma(\widehat{\vartheta})})] \ar@{->}[r] \ar@{->}[d]^{\sim}& \Irr[\pi_{0}(Z(\gpalg*{\G})^{\widehat{\vartheta}})] \ar@{->}[d]^{\sim}\\
[\bar{X}(\mathbf{S},\theta)/(1-\vartheta)\bar{X}(\mathbf{S},\theta)]_{\text{tor}} \ar@{->}[r] \ar@{->}[d]^{\sim}& [\bar{X}/(1-\vartheta)\bar{X}]_{\text{tor}}\ar@{->}[d]^{\sim}\\
H^{1}(\fr,\WSte/\WSta)\ar@{->}[r] & H^{1}(\fr,\We/\Wa)}.\]

Maintenant, en remarquant de plus que  que les flèches verticales de gauches du diagramme précédent sont compatibles avec l'action de $\pi_{0}(\phi)^{\sigma(\widehat{\vartheta})}$ et que 
\[\WSte*/\WSte \simeq \WStf*/\WStf \simeq \pi_{0}(C_{\gpfinialg*{\G}}(t)),\]
nous pouvons passer au quotient et l'on obtient le résultat.

\end{proof}

Les propositions \ref{procommirrW} et \ref{proSimeSimr} démontrent la proposition suivante

\begin{Pro}
\label{probijrnr}
Nous avons une bijection
\[\mathcal{C}^{r}(\phi,\sigma)/{\sim_{e}} \overset{\sim}{\longrightarrow} \ker(\hps),\]
où $\mathcal{C}^{r}(\phi,\sigma)$ est la classe de $\sim_r$-équivalence image réciproque de $(\phi,\sigma)$ par l'injection $\pairesSt/{\sim_r} \hookrightarrow \Lpbm{\Ild}$.
\end{Pro}

On déduit de la proposition \ref{probijrnr} et du théorème \ref{theBijSime} le théorème suivant

\begin{The}
\label{thedecomponrps}
Soient $\gpalg{G}$ un groupe réductif connexe non-ramifié.

Soient $\phi : \Ild \to \Ldual{G}$ un paramètre inertiel et  $\sigma \in \overline{\Sigma}( \langle \widehat{\vartheta} \rangle ,\widetilde{\pi}_{0}(\phi))$. Nous avons la décomposition suivante de la catégorie $\rep[\ld][(\phi,\sigma)]{\G}$ 
\[\rep[\ld][(\phi,\sigma)]{G} = \prod_{\alpha \in \ker(\hps)} \rep[\ld][(\phi,\sigma,\alpha)]{G} \]

Toutes ces catégories sont construites à partir de systèmes de classes de conjugaison 0-cohérents minimaux.
\end{The}

\begin{Rem} \label{remdependanceo}La bijection $\mathcal{C}^{r}(\phi,\sigma)/{\sim_{e}} \overset{\sim}{\longrightarrow} \ker(\hps)$ (de la proposition \ref{probijrnr}) semble canonique, mais ne l'est pas (par conséquent la décomposition du théorème \ref{thedecomponrps} ne l'est pas non plus). Elle dépend du choix du sommet hyperspécial $o$ qui permet de sélectionner un représentant $(\gpalg{S},\theta)$ de $\mathcal{C}^{r}(\phi,\sigma)$. 

De plus, de façon générale, il n'est pas possible de la rendre canonique. En effet, prenons l'exemple de $\sl{2}$ et de la classe de conjugaison rationnelle $\left[ \begin{smallmatrix} -1&0\\ 0&1 \end{smallmatrix} \right]$ dans $\pgl{2}[\res]$. Cette dernière, plus le choix d'un sommet hyperspécial donne un système de classes de conjugaison cohérent minimal (en mettant des systèmes de conjugaison vide sur les sommets non-conjugués à ce sommet). Comme il y a deux ensembles de sommets hyperspéciaux, non-conjugués entre eux, on obtient deux systèmes minimaux, et il est impossible de les différencier.

Il peut arriver dans certains cas, que malgré le fait que l'on doive faire le choix d'un sommet hyperspécial, la sous-catégorie construite n'en dépende finalement pas. Prenons l'exemple de $G=\sp{4}(\kk)$ et de $(\phi,\sigma)$ correspondant à la classe de conjugaison rationnelle de $s=diag(1,-1,-1,-1,-1) \in \so{5}(\res)$. Nommons $x_1$ et $x_2$ les deux sommets hyperspéciaux d'une chambre, et $y$ le dernier sommet de cette chambre (non-hyperspécial). Déterminons le système de classes de conjugaison correspondant à cette classe de $\sim_r$-équivalence. Par la section \ref{secSimpler}, il suffit de regarder l'image réciproque de $s$ par les applications $\ss{\quotred*{G}{x_1}} \to \ss{\gpfini*{G}}$,  $\ss{\quotred*{G}{x_2}} \to \ss{\gpfini*{G}}$ et  $\ss{\quotred*{G}{y}} \to \ss{\gpfini*{G}}$. Les sommets $x_1$ et $x_2$ sont aisés, puisque $\quotred*{G}{x_1} \simeq \quotred*{G}{x_2} \simeq \so{5}(\res)$, donc l'image réciproque est la classe de conjugaison de $s$. Pour $y$, nous avons que $\quotred*{G}{y} \simeq \so{3}(\res) \times \so{3}(\res)$ et un calcul montre qu'il y a deux éléments dans l'image réciproque de $s$ : la classe de conjugaison de $s'=(diag(1,-1,-1),diag(1,-1,-1))$ et celle de $t=(t',t')$, où $t'=\begin{pmatrix} 
-1 & 0 & 0\\
0 & 0 & 1\\
0 & 1 & 0 
\end{pmatrix}$. Les éléments $s$ et $s'$ provienne du tore diagonal et sont donc dans le même système minimal. La classe de conjugaison de $t$ est quant à elle elliptique, donc engendre son propre système minimal. Ainsi, nous trouvons que $\rep[\ld][(\phi,\sigma)]{G}$ se décompose en deux sous-catégories. Maintenant, que l'on choisisse $x_1$ ou $x_2$, on obtient le système cohérent minimal contenant $(x_1,s)$ et $(x_2,s)$. Dans cet exemple, la décomposition de $\rep[\ld][(\phi,\sigma)]{G}$ est canonique.

\bigskip

Nous nous attendons également à une compatibilité du théorème \ref{thedecomponrps} à la correspondance de Langlands locale enrichie (à un paramètre de Langlands enrichi $(\varphi,\eta)$ où $\eta \in \Irr(\pi_{0}(C_{\gpalg*{G}}(\varphi)))$, on associe un $\alpha$ en restreignant $\eta$ à $\pi_{0}(Z(C_{\gpalg*{\G}}(\phi)^{\circ})^{\sigma(\widehat{\vartheta})})$). Cette dernière n'est pas non plus canonique mais dépend du choix d'une donnée de Whittaker.
\end{Rem}

\subsection{Décomposition pour une forme intérieure pure}
En suivant les idées de Vogan, supposons toujours que $\gpalg{G}$ est non-ramifié et nous allons traiter toutes ses formes intérieures pures.

\sautintro

Soit $u \in Z^{1}(\fr,G^{nr})$. Ce dernier permet de définir un Frobenius twisté $\fr_u:=\Ad(u) \circ \fr$ donnant lieu à une forme intérieure pure $(G^{nr})^{\fr_u}$ de $\G$. Pour $g \in G^{nr}$, la conjugaison par $g$ induit un isomorphisme
\[ Ad(g) : (G^{nr})^{\fr_u} \longrightarrow (G^{nr})^{\fr_{g*u}} \]
où $g*u=gu\fr(g)^{-1}$. Ainsi les formes intérieures pures sont paramétrées par les $\omega \in H^1(\fr,G^{nr})$, mais à un $\omega$ fixé nous n'avons pas de Frobenius twisté canonique associé. Pour remédier à cela nous allons considérer dans cette partie non pas des paires $(\gpalg{S},\theta)$ mais des triplets $(\gpalg{S},\theta,u)$ avec $u \in Z^{1}(\fr,G^{nr})$.

Considérons les triplets $(\gpalg{S},\theta,n)$ où $n \in N(\gpalg{S})$ donne un cocyle dans $Z^{1}(\fr,G^{nr})$, $\gpalg{S}$ est un $\knr$-tore déployé maximal $\fr_n$-stable de $\gpalg{\G}$ (on rappelle que $\fr_n=\Ad(n) \circ \fr$) et $\theta \in X^{*}(\gpalg{S})/(\fr_{n,\gpalg{S}} -1)X^{*}(\gpalg{S})$, avec $\fr_{n,\gpalg{S}}:=w\vartheta_{\mathbf{S}}\circ \psi$, où $w$ est l'image de $n$ dans $\Wf[S]$. Alors dans ce cas, la paire $(\gpalg{S},\theta)$ est un élément de $\pairesSt$ comme dans la section \ref{secPairesStheta}, mais pour le groupe $(G^{nr})^{\fr_n}$. Notons cet ensemble $\pairesSt((G^{nr})^{\fr_n})$.

Deux triplets $(\gpalg{S}_1,\theta_1,n_1)$ et $(\gpalg{S}_2,\theta_2,n_2)$ sont dit $G^{nr}$-conjugués s'il existe $g \in G^{nr}$ tel que ${}^{g} \gpalg{S}_1=\gpalg{S}_2$, $g\theta_1=\theta_2$ et $g*n_1=n_2$. Prenons $(\gpalg{S},\theta,n)$ un triplet. La paire $(\gpalg{S},\theta)$ appartient alors à $\pairesSt((G^{nr})^{\fr_n})$. Nommons $\omega$ l'image de $n$ par $Z^{1}(\fr,G^{nr}) \to H^{1}(\fr,G^{nr})$. Si $m \in Z^{1}(\fr,G^{nr})$ est un autre cocycle d'image $\omega$, alors il existe $g \in G^{nr}$ tel que $g*n=m$. Alors la paire $({}^{g}\gpalg{S},g\theta)$ appartient à $\pairesSt((G^{nr})^{\fr_m})$. Si $g'\in G^{nr}$ est un autre élément tel que $g'*n=m$, alors $g(g')^{-1} =\fr_m( g(g')^{-1})$, de sorte que $(\gpalg{S},\theta,n)$ définit une paire $(\gpalg{S}_{\omega},\theta_{\omega}) \in \pairesSt(G_{\omega})$ bien définie à $G_{\omega}$-conjugaison près, où $\gpalg{G}_{\omega}$ est une forme intérieure pure de $\gpalg{G}$ associée à $\omega$. Cette application induit une bijection
\[ \{(\gpalg{S},\theta,n)\}/{\sim_{G^{nr}}} \tosim \{ (\omega,(\gpalg{S}_{\omega},\theta_{\omega})), \omega \in H^{1}(\fr,G^{nr}), (\gpalg{S}_{\omega},\theta_{\omega}) \in \pairesSt(G_{\omega})/{\sim_{G_{\omega}}}\}.\]
On peut étendre, de la même façon, la relation d'équivalence $\sim_e$ aux triplets de la manière suivante

\begin{Def}
\label{defEqeinner}
Deux triplets $(\gpalg{S}_1,\theta_1,n_1)$ et $(\gpalg{S}_2,\theta_2,n_2)$ sont dit $\sim_r$-équivalents (resp. $\sim_e$-équivalents) si et seulement s'il existe $m\in \mathbb{N}^{*}$ , $g \in \G^{nr}$ et $h \in G^{nr}$ tels que
\begin{enumerate}
		\item $g*n_2=n_1$
		\item ${}^{h} \mathbf{S}_{1}^{\fr_{n_1}^{m}}={}^{g}\mathbf{S}_{2}^{\fr_{n_2}^{m}}$
		\item $h \theta_{1}\langle m \rangle = g\theta_{2}\langle m \rangle$
		\item $h^{-1}\fr_{n_1}(h) \in \NStecomplet{\gpalg{S}_1}{\theta_1}$ (resp. $h^{-1}\fr_{n_1}(h) \in \NStacomplet{\gpalg{S}_1}{\theta_1} $)
	\end{enumerate}
\end{Def}
Notons que dans la définition précédente, $\theta_{1}\langle m \rangle$ est relatif à $\fr_{n_1}$ et $\theta_{2}\langle m \rangle$ à $\fr_{n_2}$.

Soient $(\gpalg{S}_1,\theta_1,n_1)$ et $(\gpalg{S}_2,\theta_2,n_2)$ deux triplets. Notons respectivement $(\gpalg{S}_{1,\omega_1},\theta_{1,\omega_1}) \in \pairesSt(G_{\omega_1})/{\sim_{G_{\omega_1}}}$ et $(\gpalg{S}_{2,\omega_2},\theta_{2,\omega_2}) \in \pairesSt(G_{\omega_2})/{\sim_{G_{\omega_2}}}$ les paires associées par l'application précédente. Alors  $(\gpalg{S}_1,\theta_1,n_1) \sim_r (\gpalg{S}_2,\theta_2,n_2)$ si et seulement si $\omega_1=\omega_2$ et $(\gpalg{S}_{1,\omega_1},\theta_{1,\omega_1})  \sim_r (\gpalg{S}_{2,\omega_2},\theta_{2,\omega_2})$ (et de même pour $\sim_e$). En effet, fixons par exemple $n_1$ et prenons $G_{\omega_1}=(G^{nr})^{\fr_{n_1}}$. Alors $\omega_1=\omega_2$ si et seulement s'il existe $g\in G^{nr}$ tel que $g*n_2=n_1$. Dans ce cas on peut prendre $(\gpalg{S}_{2,\omega_2},\theta_{2,\omega_2})=({}^{g}\gpalg{S}_2,g\theta_2)$ et $(\gpalg{S}_{1,\omega_1},\theta_{1,\omega_1}) = (\gpalg{S}_1,\theta_1)$ et la condition pour que ces deux paires soient $\sim_r$-équivalentes (resp. $\sim_e$-équivalentes) est exactement la définition \ref{defEqeinner}.  On obtient donc en particulier une bijection
\[ \{(\gpalg{S},\theta,n)\}/{\sim_e} \tosim \{ (\omega,(\gpalg{S}_{\omega},\theta_{\omega})), \omega \in H^{1}(\fr,G^{nr}), (\gpalg{S}_{\omega},\theta_{\omega}) \in \pairesSt(G_{\omega})/{\sim_e}\}.\]
Une classe de $\sim_e$-équivalence de triplets $(\gpalg{S},\theta,n)$ correspond donc à un système d'idempotents cohérent minimal de $G_{\omega}$.

\bigskip

Considérons un triplet $(\gpalg{S},\theta,e)$ ($e$ désigne l'élément neutre de $G^{nr}$) et $n \in \NSte$ donnant un cocyle dans $Z^{1}(\fr,G^{nr})$. Le tore $\gpalg{S}$ est $\fr$-stable et comme $n \in N(\gpalg{S})$ il est également $\fr_n$-stable. La preuve du lemme \ref{lemgStheta} montre qu'il existe $\theta_n \in X^{*}(\gpalg{S})/(\fr_{n,\gpalg{S}} -1)X^{*}(\gpalg{S})$ tel que $\theta \langle m \rangle = \theta_n \langle m \rangle$ pour un certain $m \in \mathbb{N}^{*}$. Notons que ce $\theta_n$ est unique. En effet, s'il existe $\theta'_n \in X^{*}(\gpalg{S})/(\fr_{n,\gpalg{S}} -1)X^{*}(\gpalg{S})$ tel que $\theta \langle m' \rangle = \theta'_n \langle m' \rangle$, alors $\theta_n \langle mm' \rangle = \theta'_n \langle mm' \rangle$ et $\theta_n=\theta'_n$ par injectivité de  $\Tr_{\fr^{mm'}/\fr}$ (lemme \ref{lemTrace}). On vient donc de construire une application qui à $(\gpalg{S},\theta,e)$ associe le triplet $(\gpalg{S},\theta_n,n)$.

Comme $\theta \langle m \rangle = \theta_n \langle m \rangle$, on a en particulier que $\WSte = \WStecomplet{\gpalg{S}}{\theta_n}$ et $\WSta = \WStacomplet{\gpalg{S}}{\theta_n}$. On obtient de la sorte

\begin{Cor}
\label{corH1formepure}
La multiplication par $n^{-1}$ (à droite) nous donne le diagramme commutatif suivant
\[ \xymatrix{
\tilde{H}^{1}(\fr,\WSte/\WSta) \ar@{->}[r] \ar@{->}[d]^{\times n^{-1}}& H^{1}(\fr,\We/\Wa) \ar@{->}[d]^{\times n^{-1}}\\
\tilde{H}^{1}(\fr_n, \WStecomplet{\gpalg{S}}{\theta_n} / \WStacomplet{\gpalg{S}}{\theta_n})\ar@{->}[r] & H^{1}(\fr_n,\We/\Wa)}.\]
\end{Cor}

Revenons à nos classes d'équivalence pour $\sim_e$. Rappelons que $\hps$ est l'application $\hps : \Irr[\pi_{0}(Z(C_{\gpalg*{\G}}(\phi)^{\circ})^{\sigma(\widehat{\vartheta})})]/{\pi_{0}(\phi)^{\sigma(\widehat{\vartheta})}} \to \Irr[\pi_{0}(Z(\gpalg*{\G})^{\widehat{\vartheta}})]$. Grâce à l'isomorphisme de Kottwitz $H^{1}(\fr,\G^{nr}) \tosim \Irr[\pi_{0}(Z(\gpalg*{\G})^{\widehat{\vartheta}})]$ nous pouvons identifier un élément $\omega \in H^{1}(\fr,\G^{nr})$ à $\omega \in \Irr[\pi_{0}(Z(\gpalg*{\G})^{\widehat{\vartheta}})]$.

\begin{Pro}
\label{probijclasser}
Soient $\omega \in \Irr[\pi_{0}(Z(\gpalg*{\G})^{\widehat{\vartheta}})]$. Nous avons alors une bijection
\[\mathcal{C}^{r}(\phi,\sigma,\omega)/{\sim_{e}} \overset{\sim}{\longrightarrow} \hps^{-1}(\omega),\]
où $\mathcal{C}^{r}(\phi,\sigma,\omega)$ est la classe de $\sim_r$-équivalence image réciproque de $(\phi,\sigma)$ par l'injection $\pairesSt(G_{\omega})/{\sim_r} \hookrightarrow \Lpbm{\Ild}$, pour $G_{\omega}$ la forme intérieure pure de $G$ associée à $\omega$.

De plus, cette bijection ne dépend que du choix du sommet hyperspécial $o$ fait en \ref{secSthetaQuasidep}.
\end{Pro}

\begin{proof}
Considérons $(\gpalg{S},\theta)$ une paire construite à partir de $(\phi,\sigma)$ et $o$ comme dans la section \ref{secSthetaQuasidep}. La proposition \ref{procommirrW} montre que l'on peut identifier $\hps$ à l'application $\tilde{H}^{1}(\fr,\WSte/\WSta) \rightarrow H^{1}(\fr,\We/\Wa)$. On voit donc $\omega$ comme un élément de $H^{1}(\fr,\We/\Wa)$ dans l'image de $\tilde{H}^{1}(\fr,\WSte/\WSta)$. Le lemme \ref{lemH1NW} ainsi que la preuve de la proposition \ref{proSimeSimr} montrent que l'application $H^{1}(\fr, \NSte) \rightarrow H^{1}(\fr,\WSte/\WSta)$ est surjective. En particulier, il existe $n \in \NSte$ donnant un cocycle de $Z^{1}(\fr, \NSte)$ et ayant pour image $\omega$ par l'application $Z^{1}(\fr, \NSte) \to \tilde{H}^{1}(\fr,\WSte/\WSta) \to H^{1}(\fr,\We/\Wa)$.

Nous savons alors construire un triplet $(\gpalg{S},\theta_{n},n)$ tel que $\mathcal{C}^{r}(\phi,\sigma,\omega) = [\gpalg{S},\theta_{n},n]_{r}$ (où $[\gpalg{S},\theta_{n},n]_{r}$ désigne la classe de $\sim_r$-équivalence de $(\gpalg{S},\theta_{n},n)$). La proposition \ref{proSimeSimr} nous donne une bijection entre $[\gpalg{S},\theta_{n},n]_{r}/{\sim_{e}}$ et $\ker[\tilde{H}^{1}(\fr_n,\WStecomplet{\gpalg{S}}{\theta_n}/\WStacomplet{\gpalg{S}}{\theta_n}) \rightarrow H^{1}(\fr_n,\We/\Wa)]$ et donc par le corollaire \ref{corH1formepure} une bijection
\[ [\gpalg{S},\theta_{n},n]_{r}/{\sim_{e}} \tosim \hps^{-1}(\omega).\]

Il nous reste à vérifier que cette bijection est indépendante du choix de $n$. Prenons $m \in Z^{1}(\fr, \NSte)$ un autre relèvement de $\omega$. Nous allons montrer que l'image de $(\gpalg{S},\theta_{m},m)$ par la bijection précédente est l'image de $m$ dans $\tilde{H}^{1}(\fr,\WSte/\WSta)$ ce qui achèvera la preuve. Comme les triplets $(\gpalg{S},\theta_{n},n)$ et $(\gpalg{S},\theta_{m},m)$ sont $\sim_r$-équivalents, il existe $k\in \mathbb{N}^{*}$ , $g \in \G^{nr}$ et $h \in G^{nr}$ tels que $g*m=n$, ${}^{h} (\mathbf{S}^{\fr_{n}^{k}})={}^{g}\mathbf{S}^{\fr_{m}^{k}}$, $h \theta_{n}\langle k \rangle = g\theta_{m}\langle k \rangle$ et $h^{-1}\fr_{n}(h) \in \NStecomplet{\gpalg{S}}{\theta_n}$. Alors par construction de la bijection $[\gpalg{S},\theta_{n},n]_{r}/{\sim_{e}} \tosim \hps^{-1}(\omega)$, l'image de $(\gpalg{S},\theta_{m},m)$ est l'image de $h^{-1}n\fr(h)$ dans $\tilde{H}^{1}(\fr,\WSte/\WSta)$. Nous savons qu'il existe $k_n,k_m \in \mathbb{N}^{*}$ tels que $\theta_n \langle k_n \rangle = \theta \langle k_n \rangle$ et $\theta_m \langle k_m \rangle = \theta \langle k_m \rangle$. Ainsi, il existe $k' \in \mathbb{N}^{*}$ tel que $\fr_{n}^{k'} = \fr_{m}^{k'}=\fr^{k'}$, ${}^{h} (\mathbf{S}^{\fr^{k'}})={}^{g}\mathbf{S}^{\fr^{k'}}$ et $h \theta \langle k' \rangle = g\theta \langle k' \rangle$, de sorte que $g^{-1}h \in \NSte*$. Nous pouvons alors conclure que $h^{-1}n\fr(h)=(h^{-1}g)m\fr(g^{-1}h)$ et $m$ ont même image dans $\tilde{H}^{1}(\fr,\WSte/\WSta)$.
\end{proof}

On déduit de la proposition \ref{probijclasser} précédente et du théorème \ref{theBijSime} le théorème suivant.

\begin{The}
\label{thmDecompoRepPhi}
Soient $\gpalg{G}$ un groupe réductif connexe non-ramifié. Soit $\omega \in \Irr[\pi_{0}(Z(\gpalg*{\G})^{\widehat{\vartheta}})]$, à qui correspond $\gpalg{G}_{\omega}$, une forme intérieure pure de $\gpalg{G}$. 

Soient $\phi : \Ild \to \Ldual{G}$ un paramètre inertiel et  $\sigma \in \overline{\Sigma}( \langle \widehat{\vartheta} \rangle ,\widetilde{\pi}_{0}(\phi))$. Nous avons la décomposition suivante de la catégorie $\rep[\ld][(\phi,\sigma)]{\G_\omega}$ 
\[\rep[\ld][(\phi,\sigma)]{G_{\omega}} = \prod_{\alpha \in \hps^{-1}(\omega)} \rep[\ld][(\phi,\sigma,\alpha)]{G_{\omega}} \]

De plus toutes ces catégories sont construites à partir de systèmes de classes de conjugaison 0-cohérents minimaux.
\end{The}

Notons que cette décomposition n'est pas canonique mais dépend du choix d'un sommet hyperspécial (voir remarque \ref{remdependanceo}).

\subsection{Lien entre \texorpdfstring{$\Zl$}{Zl} et \texorpdfstring{$\Ql$}{Ql}}

Explicitons le lien entre les catégories sur $\Ql$ et sur $\Zl$.

\sautintro

Soit $(\phi,\sigma) \in \Lpbm{\iner}$ et $(\phi',\sigma') \in \Lpbm{\inerl}$ l'image de $(\phi,\sigma)$ par l'application $\Lpbm{\iner} \to \Lpbm{\inerl}$ de la section \ref{seclienzlqlr}. Alors $Z(C_{\gpalg*{\G}}(\phi')^{\circ})^{\sigma'(\widehat{\vartheta})} \subseteq Z(C_{\gpalg*{\G}}(\phi)^{\circ})^{\sigma(\widehat{\vartheta})}$ et l'on obtient une application 
\[\Irr[\pi_{0}(Z(C_{\gpalg*{\G}}(\phi)^{\circ})^{\sigma(\widehat{\vartheta})})] \to \Irr[\pi_{0}(Z(C_{\gpalg*{\G}}(\phi')^{\circ})^{\sigma'(\widehat{\vartheta})})]\]
induisant une application sur les quotients.

Il est alors aisé d'obtenir la proposition suivante.

\begin{Pro}
Soit $\omega \in \Irr[\pi_{0}(Z(\gpalg*{\G})^{\widehat{\vartheta}})]$. Nous avons 
\[\rep[\Zl][(\phi',\sigma',\alpha')]{G_{\omega}} \cap \rep[\Ql]{G_{\omega}} = \prod_{(\phi,\sigma,\alpha)} \rep[\Ql][(\phi,\sigma,\alpha)]{G_{\omega}},\]
où le produit est pris sur les $(\phi,\sigma) \in \Lpbm{\iner}$ s'envoyant sur $(\phi',\sigma')$ par $\Lpbm{\iner} \to \Lpbm{\inerl}$ et $\alpha \in \hps^{-1}(\omega)$ s'envoyant sur $\alpha'$ par l'application précédente.
\end{Pro}

\subsection{Compatibilité à l'induction et à la restriction parabolique}

\label{secCompaInductionParaboliqueAlpha}

Vérifions ici la compatibilité de la décomposition du théorème \ref{thmDecompoRepPhi} aux foncteurs d'induction et de restriction parabolique.

\sautintro

Soit $\gpalg{P}$ un $\kk$-sous-groupe parabolique de $\gpalg{\G}$ de quotient de Levi $\gpalg{M}$ défini sur $\kk$. Considérons $\gpalg*{M}$ un dual de $\gpalg{M}$ sur $\Ql$ muni d'un plongement $\iota:\Ldual{M} \hookrightarrow \Ldual{\G}$. 
Soient $\phi_{M} \in \Lparamm{I_k^{\ld}}{M}$ et $\sigma_{M} \in \Sigma(\langle \widehat{\vartheta} \rangle,\widetilde{\pi}_{0}(\phi_{M}))$. Nous avons vu à la section \ref{secCompaInductionParaboliquePhiSigma} que nous pouvions associer à $\phi_M$ et $\sigma_M$ le paramètre $\phi:=\iota \circ \phi_{M} \in \Lpm{\Ild}$ et $\sigma:=\iota \circ \sigma_{M} \in  \Sigma(\langle \widehat{\vartheta} \rangle,\widetilde{\pi}_{0}(\phi))$.

Nous avons que $\iota(Z(C_{\gpalg*{M}}(\phi_M)^{\circ}) \supseteq Z(C_{\gpalg*{\G}}(\phi)^{\circ})$ donc $\iota$ induit une application
\[\Irr[\pi_{0}(Z(C_{\gpalg*{M}}(\phi_M)^{\circ})^{\sigma_M(\widehat{\vartheta})})] \rightarrow \Irr[\pi_{0}(Z(C_{\gpalg*{\G}}(\phi)^{\circ})^{\sigma(\widehat{\vartheta})})],\]
qui passe au quotient, et rendant le diagramme suivant commutatif

\[ \xymatrix{
\Irr[\pi_{0}(Z(C_{\gpalg*{M}}(\phi_M)^{\circ})^{\sigma_M(\widehat{\vartheta})})]/{\pi_{0}(\phi_M)^{\sigma_M(\widehat{\vartheta})}}  \ar@{->}[r] \ar@{->}[d]& \Irr[\pi_{0}(Z(C_{\gpalg*{\G}}(\phi)^{\circ})^{\sigma(\widehat{\vartheta})})]/{\pi_{0}(\phi)^{\sigma(\widehat{\vartheta})}}  \ar@{->}[d]\\
 \Irr[\pi_{0}(Z(\gpalg*{M})^{\widehat{\vartheta}})]  \ar@{->}[r] &   \Irr[\pi_{0}(Z(\gpalg*{\G})^{\widehat{\vartheta}})] }.\]

Soit $\omega_M \in \Irr[\pi_{0}(Z(\gpalg*{M})^{\widehat{\vartheta}})]$ et $\omega \in \Irr[\pi_{0}(Z(\gpalg*{\G})^{\widehat{\vartheta}})]$ son image par $\Irr[\pi_{0}(Z(\gpalg*{M})^{\widehat{\vartheta}})]  \rightarrow   \Irr[\pi_{0}(Z(\gpalg*{\G})^{\widehat{\vartheta}})] $. Ainsi $\alpha_M \in \psi_{(\phi_M,\sigma_M)}^{-1}(\omega_M) $ est envoyé sur  $\alpha \in \hps^{-1}(\omega)$. On notera par la suite $\iota_{M}^{\G}$, comme dans la section \ref{secCompaInductionParaboliquePhiSigma} l'application qui à $(\phi_M,\sigma_M,\alpha_M)$ associe $(\phi,\sigma,\alpha)$. Notons que comme $\omega$ est l'image de $\omega_M$, alors $M_{\omega_M}$ est un Levi de $G_\omega$ et l'on a un plongement de $\bte[M_{\omega_M}]$ dans $\bte[G_\omega]$.

\bigskip

Réinterprétons l'application $\alpha_M  \rightarrow \alpha$ en terme de groupes de Weyl. Nommons $t \in \ss*{\gpfini*{M}}$ la classe de conjugaison semi-simple associée à $(\phi_M,\sigma_M)$. Fixons $o$ un sommet hyperspécial de $\bte[M]$ donnant un sommet hyperspécial de $\bte[\G]$ (que l'on nomme encore $o$). La section \ref{secSthetaQuasidep} construit alors à partir de $(\phi_M,\sigma_M)$ une paire $(\gpalg{S},\theta)$ où $\gpalg{S}$ est un tore maximal non-ramifié de $\gpalg{M}$. Ce dernier permet de définir les groupes $\WSteM{M}$,$\WStaM{M}$, $\WeM{M}$, $\WaM{M}$ (on rajoute des indices $M$ pour signifier que ces groupes sont définis vis à vis de $M$). Le tore $\gpalg{S}$ peut également être vu comme un tore maximal non-ramifié de $\gpalg{\G}$ et donc définit également les groupes $\WSte$,$\WSta$, $\We$, $\Wa$. Les inclusions $\WSteM{M} \subseteq \WSte$ et $\WStaM{M} \subseteq \WSta$  permettent de définir une application
\[ \tilde{H}^{1}(\fr,\WSteM{M}/\WStaM{M})  \longrightarrow\tilde{H}^{1}(\fr,\WSte/\WSta) .\]
Cette application correspond à 
\[\Irr[\pi_{0}(Z(C_{\gpalg*{M}}(\phi_M)^{\circ})^{\sigma_M(\widehat{\vartheta})})]/{\pi_{0}(\phi_M)^{\sigma_M(\widehat{\vartheta})}} \longrightarrow \Irr[\pi_{0}(Z(C_{\gpalg*{\G}}(\phi)^{\circ})^{\sigma(\widehat{\vartheta})})]/{\pi_{0}(\phi)^{\sigma(\widehat{\vartheta})}}\]
via la proposition \ref{procommirrW}.

\begin{Lem}
\label{lemLeviClasse}
Notons $N$ une forme intérieure pure de $M$ associée à $\omega_M$. Soient $x$ un sommet dans l'immeuble de Bruhat-Tits étendu de $N$ et $s \in \ss*{\quotred*{N}{x}}$. Nommons $(\phi_M,\sigma_M,\alpha_M)$ les données associées à $s$. Notons également $(\phi,\sigma,\alpha)$ les données associées à $s$ vu comme un élément de $\ss*{\quotred*{\G}{x}}$ (où $x$ est vu ici comme un élément de $\bte[G_\omega]$). Alors $(\phi,\sigma,\alpha) = \iota_{M}^{\G} (\phi_M,\sigma_M,\alpha_M)$.
\end{Lem}

\begin{proof}
Cela découle aisément de la définition de l'application $\iota_{M}^{\G}$, puisque à $s$ on associe une paire $(\mathbf{S}',\theta')$, avec $\mathbf{S}'$ un tore maximal non-ramifié de $N$, et lorsque $s$ est vu comme un élément de $\ss*{\quotred*{\G}{x}}$, on peut lui associer la même pair $(\mathbf{S}',\theta')$ mais en voyant cette fois-ci $\mathbf{S}'$ comme un tore maximal non-ramifié de $G_{\omega}$.
\end{proof}

\begin{The}
\label{theInductionParaboliquealpha}
Soit $\gpalg{P}$ un sous-groupe parabolique de $\gpalg{\G}$ ayant pour facteur de Levi $\gpalg{M}$.
\begin{enumerate}
\item 
\[ \rp[ \rep[\ld][(\phi,\sigma,\alpha)]{\G_{\omega}} ] \subseteq \prod_{\substack{(\phi_{M},\sigma_M,\alpha_M) \\\iota_{M}^{\G} (\phi_M,\sigma_M,\alpha_M) =(\phi,\sigma,\alpha)}} \rep[\ld][(\phi_{M},\sigma_M,\alpha_M)]{M_{\omega_M}}\]
\item 
\[ \ip[ \rep[\ld][(\phi_{M},\sigma_M,\alpha_M)]{M_{\omega_M}} ] \subseteq \rep[\ld][\iota_{M}^{\G} (\phi_M,\sigma_M,\alpha_M)]{G_\omega}\]
\end{enumerate}

\end{The}

\begin{proof}
Cela découle du lemme \ref{lemLeviClasse} et des propositions \ref{proInducParaSyst}, \ref{proResParaSyst}.
\end{proof}

\begin{The}
\label{theEquivalencephisigmaalpha}
Si $C_{\gpalg*{\G}}(\phi) \subseteq \iota(\gpalg*{M})$ le foncteur d'induction parabolique $\ip$ réalise une équivalence de catégories entre $\rep[\ld][(\phi_{M},\sigma_M,\alpha_M)]{M_{\omega_M}}$ et $\rep[\ld][\iota_{M}^{\G} (\phi_M,\sigma_M,\alpha_M)]{G_\omega}$.
\end{The}

\begin{proof}
Cela découle du théorème \ref{theInductionParaboliquealpha} et du théorème 4.4.6 dans \cite{lanard}.
\end{proof}

\subsection{Compatibilité à la construction de DeBacker-Reeder}

Soit $\varphi$ un paramètre de Langlands modéré elliptique et en position général comme dans \cite{DebackerReeder} (elliptique signifie que l'image de $\varphi$ n'est pas contenue dans un sous-groupe de Levi propre de $\Ldual{\G}$ et en position générale que le centralisateur de $\varphi(\iner)$ dans $\gpalg*{G}$ est un tore). Soit $\omega \in H^{1}(\kk,\gpalg{G})$. L'isomorphisme de Kottwitz permet de définir une application $\Irr(\cent{\varphi}{\gpalg*{G}}) \rightarrow H^{1}(\kk,\gpalg{G})$ et l'on note $\Irr(\cent{\varphi}{\gpalg*{G}},\omega)$ la fibre sur $\omega$ de cette application. DeBacker et Reeder construisent alors dans \cite{DebackerReeder} un $L$-paquet $\Pi(\varphi,\omega)$ pour $G_{\omega}$, paramétré par $\Irr(\cent{\varphi}{\gpalg*{G}},\omega)$. Nous souhaitons vérifier dans cette partie, que cette construction est compatible à la décomposition du théorème \ref{thmDecompoRepPhi}.

\sautintro

Fixons $o$ un sommet hyperspécial que l'on choisira identique pour la construction de DeBacker-Reeder et la notre, ainsi que $\gpalg{T}$ un tore maximalement déployé de $\G$.

Rappelons brièvement la construction de DeBacker et Reeder. Quitte à conjuguer $\varphi$, on peut supposer que $\varphi(\iner) \subseteq \gpalg*{T}$ et que $\varphi(\text{Frob})=f\widehat{\vartheta}$, avec $f \in N(\gpalg*{T})$. Notons $w$, l'image de $f$ dans $\Wf$. On peut former à partir de ce $\varphi$ (\cite{DebackerReeder} section 4.3) un caractère $\chi$ de niveau zéro de $\gpalg{T}^{\fr}$. Soit $\lambda \in X_w$, où $X_w$ est la préimage dans $X$ de $[X/(1-w\vartheta)X]_{\text{tor}}$. L'évaluation sur une uniformisante $\varpi$, identifie $\lambda$ à $t_\lambda \in \We$, l'opérateur de translation par $\lambda$. DeBacker et Reeder construisent alors, dans la section 2.7, $\dot{w} \in N$, $\dot{w}_\lambda \in N$ et $u_\lambda \in Z^{1}(\fr,N)$ tels que $\dot{w}$ relève $w$, $\dot{w}_\lambda = p_{\lambda}^{-1}\fr_\lambda(p_\lambda)$ et $t_\lambda \dot{w} = \dot{w}_\lambda	u_\lambda$, où $\fr_{\lambda}= \Ad(u_\lambda) \circ \fr$. On pose $\gpalg{T}_\lambda = {}^{p_\lambda}\gpalg{T}$ et $\chi_\lambda = p_\lambda \chi \in \Irr(\gpalg{T}_{\lambda}^{\fr_\lambda})$. Grâce à la théorie de Deligne-Lusztig, ils forment à partir de $(\gpalg{T}_\lambda,\chi_\lambda)$ une représentation $\pi_\lambda$ de $(G^{nr})^{\fr_\lambda}$.

\begin{The}
\label{thecompadebackerreeder}
La construction de DeBacker-Reeder (\cite{DebackerReeder}) est compatible au théorème \ref{thmDecompoRepPhi}. C'est à dire que $\pi_\lambda$ appartient à $\rep[\Ql][(\phi,\sigma,\alpha)]{G_{\omega}}$ où $\phi=\varphi_{|\iner}$, $\sigma(\widehat{\vartheta})=\varphi(\text{Frob})$ et $\alpha$ correspond à $\lambda$ via la bijection $[X/(1-w\vartheta)X]_{\text{tor}} \overset{\sim}{\longrightarrow} \Irr[\pi_{0}(C_{\gpalg*{\G}}(\varphi))]$ (voir \cite{DebackerReeder} 4.1).
\end{The}

\begin{proof}
Dans cet article, à partir de $\varphi$ nous construisons non pas $(\gpalg{T},\chi)$ mais une paire $(\gpalg{S},\theta)$, où $\theta \in \Irr({}^{0}(\gpalg{S}^{\fr})/(\gpalg{S}^{\fr})^{+})$. Le lien entre les deux est le suivant. Écrivons, comme dans \cite{DebackerReeder} section 2.6, $\dot{w}=p_{0}^{-1} \fr(p_0)$, avec $p_0 \in \para{\G}{o}$. Alors $\gpalg{S}={}^{p_0}\gpalg{T}$. Comme $p_0\chi \in \Irr(\gpalg{S}^{\fr})$ est de niveau zéro, on le voit comme un caractère de $\gpalg{S}^{\fr}/(\gpalg{S}^{\fr})^{+}$, et on a $\theta = (p_0\chi)_{|{}^{0}(\gpalg{S}^{\fr})}$.

Notons qu'avec les hypothèses faites sur $\varphi$, on a que $\cent{\phi}{\gpalg*{G}}=\gpalg*{T}$ et $\sigma(\widehat{\vartheta})=w\widehat{\vartheta}$, de sorte que $C_{\gpfinialg*{G}}(t)^{\circ}=\gpfinialg*{S}$. En particulier $\WSte*=\WSte=X_{*}(\gpalg{S})$, $\WSta=1$ et $\NSte=\gpalg{S}$. L'élément $p_0 \lambda p_{0}^{-1}$ nous donne bien un élément de $H^{1}(\fr,\WSte/\WSta)$, et on peut appliquer le théorème \ref{thmDecompoRepPhi} pour obtenir un système cohérent. Posons $n:=p_0 t_\lambda p_{0}^{-1} \in \NSte$ qui relève $p_0 \lambda p_{0}^{-1}$. La preuve du théorème \ref{thmDecompoRepPhi} construit alors le triplet $(\gpalg{S},\theta_n,n)$. Mais comme $n \in \gpalg{S}$, $\fr=\fr_{n}$ sur $\gpalg{S}$ et donc $\theta_n = \theta$. On obtient un triplet $(\gpalg{S},\theta,n)$ qui donne un système cohérent. Posons $g:=p_\lambda p_{0}^{-1}$. Alors 
\begin{align*}
g*n &= g n \fr(g)^{-1} = (p_\lambda p_{0}^{-1})(p_0 t_\lambda p_{0}^{-1})(\fr(p_{0}p_\lambda^{-1}))\\
&=p_\lambda t_\lambda \dot{w} \fr(p_\lambda)^{-1} = p_\lambda \dot{w}_\lambda	u_\lambda \fr(p_\lambda)^{-1}\\
&=p_\lambda p_{\lambda}^{-1}\fr_\lambda(p_\lambda)	u_\lambda \fr(p_\lambda)^{-1} = p_\lambda p_{\lambda}^{-1}u_\lambda \fr(p_\lambda) u_\lambda^{-1} u_\lambda \fr(p_\lambda)^{-1}\\
&=u_\lambda
\end{align*}

La définition \ref{defEqeinner} montre (avec $h=1$) que $(\gpalg{S},\theta,n) \sim_e ({}^g \gpalg{S},g\theta,u)$. Or ${}^g \gpalg{S}={}^{p_\lambda p_{0}^{-1}} \gpalg{S}={}^{p_\lambda} \gpalg{T}=\gpalg{T}_\lambda$ et de même $g\theta = (\chi_\lambda)_{|{}^{0}(\gpalg{T}_{\lambda}^{\fr_{\lambda}})}$. Les deux constructions sont donc compatibles.
\end{proof}

\subsection{Lien avec \texorpdfstring{\cite{datFunctoriality}}{[Dat17a]}}

Dat prédit dans \cite{datFunctoriality}, une décomposition de la catégorie $\rep[\ld][\phi]{G}$ indexée par $\sigma$ et des données cohomologiques liées à un groupe réductif non-ramifié $\Gps$. Rappelons ici la construction de ces groupes $\Gps$, ainsi que le lien avec le théorème \ref{thmDecompoRepPhi}. Nous verrons en particulier que ce dernier peut être reformulé avec les groupes $\Gps$, corroborant, dans le cas du niveau zéro, les prédictions de \cite{datFunctoriality}.

\sautintro

Soit $\phi : \Ild \to \Ldual{\G}$ un paramètre inertiel. Appelons alors $\mathbf{G}_{\phi}^{split,\circ}$ un groupe réductif sur $\kk$ dual de $C_{\gpalg*{\G}}(\phi)^{\circ}$. Nous avons vu dans la section \ref{secSthetaQuasidep} que l'on a une action $\widetilde{\pi}_{0}(\phi) \to \Out(C_{\gpalg*{\G}}(\phi)^{\circ})$. Le choix d'un épinglage de $\mathbf{G}_{\phi}^{split,\circ}$ nous donne une section $\Out(\mathbf{G}_{\phi}^{split,\circ}) \longrightarrow \Aut(\mathbf{G}_{\phi}^{split,\circ})$, et donc un morphisme
\[ \widetilde{\pi}_{0}(\phi) \longrightarrow \Out(C_{\gpalg*{\G}}(\phi)^{\circ})=\Out(\mathbf{G}_{\phi}^{split,\circ}) \longrightarrow \Aut(\mathbf{G}_{\phi}^{split,\circ}).\]
On peut alors former un $\kk$-groupe réductif (non-connexe)
\[ \mathbf{G}_{\phi}^{split}:=\mathbf{G}_{\phi}^{split,\circ} \rtimes \pi_{0}(\phi)\]
muni d'une action de $\widetilde{\pi}_{0}(\phi)$ donnée par
\[ \theta : \widetilde{\pi}_{0}(\phi) \longrightarrow \Aut(\mathbf{G}_{\phi}^{split})\]
(où $\widetilde{\pi}_{0}(\phi)$ agit sur $\pi_{0}(\phi)$ par conjugaison). Alors n'importe quelle section continue $\sigma : \langle \widehat{\vartheta} \rangle \rightarrow \widetilde{\pi}_{0}(\phi)$ nous donne une $\kk$-forme $\Gps$ de $\mathbf{G}_{\phi}^{split}$ telle que l'action de $\langle \widehat{\vartheta} \rangle$ sur $\Gps(\knr)=\mathbf{G}_{\phi}^{split}(\knr)$ soit l'action naturelle twistée par $\theta \circ \sigma$. La composante neutre $\Gpsc$ de $\Gps$ est alors un groupe réductif connexe non-ramifié et l'on a
\[\Gps(\kk)=\Gpsc(\kk) \rtimes \pi_{0}(\phi)^{\sigma(\widehat{\vartheta})}.\]

Notons également que si $c \in \pi_{0}(\phi)$, la conjugaison par $(1,c)$ dans $\mathbf{G}_{\phi}^{split}(\kalg)$ induit un $\kk$-isomorphisme $\Gps \overset{\sim}{\longrightarrow} \Gphisigma{\phi}{\sigma^{c}}$, de sorte que la classe d'isomorphisme de $\Gps$ sur $\kk$ ne dépend que de l'image de $\sigma$ dans $\overline{\Sigma}(\langle \widehat{\vartheta} \rangle,\widetilde{\pi}_{0}(\phi))$.

\bigskip

On définit l'ensemble $\tilde{H}^{1}(\fr,\Gpscnr):=Im[H^{1}(\fr,\Gpscnr) \rightarrow H^{1}(\fr,\Gpsnr)]$. Alors $\tilde{H}^{1}(\fr,\Gpscnr)$ n'est autre que $H^{1}(\fr,\Gpscnr)$ quotienté par l'action de $\pi_{0}(\phi)^{\sigma(\widehat{\vartheta})}$. 

Grâce à l'isomorphisme de Kottwitz, nous avons des bijections
\[H^{1}(\fr,\G^{nr}) \overset{\sim}{\rightarrow} \Irr[\pi_{0}(Z(\gpalg*{\G})^{\widehat{\vartheta}})] \text{ et }H^{1}(\fr,\Gpscnr) \overset{\sim}{\rightarrow} \Irr[\pi_{0}(Z(C_{\gpalg*{\G}}(\phi)^{\circ})^{\sigma(\widehat{\vartheta})})].\]
Ainsi, nous avons également une bijection
\[ \Irr[\pi_{0}(Z(C_{\gpalg*{\G}}(\phi)^{\circ})^{\sigma(\widehat{\vartheta})})]/{\pi_{0}(\phi)^{\sigma(\widehat{\vartheta})}} \overset{\sim}{\longrightarrow} \tilde{H}^{1}(\fr,\Gpscnr).\]

L'application $\hps$ nous fournit alors une application $\tilde{H}^{1}(\fr,\Gpscnr) \rightarrow H^{1}(\fr,\G^{nr})$, rendent le diagramme suivant commutatif
\[ \xymatrix{
\Irr[\pi_{0}(Z(C_{\gpalg*{\G}}(\phi)^{\circ})^{\sigma(\widehat{\vartheta})})]/{\pi_{0}(\phi)^{\sigma(\widehat{\vartheta})}} \ar@{->}[r]^-{\hps} \ar@{->}[d]^{\sim}& \Irr[\pi_{0}(Z(\gpalg*{\G})^{\widehat{\vartheta}})] \ar@{->}[d]^{\sim}\\
\tilde{H}^{1}(\fr,\Gpscnr)\ar@{->}[r] & H^{1}(\fr,\G^{nr})}\]

Nous obtenons de la sorte la proposition suivante

\begin{Pro}
\label{proDecompoDat}
Soient $\gpalg{G}$ un groupe réductif connexe non-ramifié et $\omega \in H^{1}(\fr,\G^{nr})$.

La décomposition du théorème \ref{thmDecompoRepPhi} peut être réinterpréter en une décomposition indexée par $(\phi,\sigma,\alpha)$ où $\phi : \Ild \to \Ldual{G}$ est un paramètre inertiel, $\sigma \in \overline{\Sigma}( \langle \widehat{\vartheta} \rangle ,\widetilde{\pi}_{0}(\phi))$ et $\alpha$ est dans l'image réciproque de $\omega$ par l'application $\tilde{H}^{1}(\fr,\Gpscnr) \rightarrow H^{1}(\fr,\G^{nr})$.
\end{Pro}

Dans la section 2.1.2 de \cite{datFunctoriality}, Dat prédit une telle décomposition lorsque $C_{\gpalg*{\G}}(\phi)^{\circ}$ est connexe, puis propose la construction des $\Gps$ dans la section 2.1.4 pour traiter du cas non-connexe. Ainsi, la proposition \ref{proDecompoDat} corrobore ces spéculations dans le cas du niveau zéro.

\bigskip

L'intérêt d'exprimer cette décomposition à l'aide des groupes $\Gps$ est le suivant. Dans \cite{datFunctoriality}, il est conjecturé que la fonctorialité de Langlands pourrait s'étendre, de manière fonctorielle, en un transfert sur les représentations non-irréductibles, et sous certaines hypothèses fournir des équivalences de catégories. Par exemple, lorsque $C_{\gpalg*{\G}}(\phi)$ est connexe, le plongement ${}^{L}(\gpalg{G}_{\phi,1})_{\alpha} \hookrightarrow \Ldual{G}_{\omega}$ devrait induire une équivalence de catégories entre $\rep[\ld][1]{(G_{\phi,1})_{\alpha}}$ et $\rep[\ld][(\phi,1,\alpha)]{G_\omega}$.

\appendix

\section{Cohomologie des groupes \texorpdfstring{$p$}{p}-adiques}

\label{secCohomologie}

Rappelons dans cette annexe, quelques éléments sur la cohomologie des groupes $p$-adiques dont nous avons besoin dans cet article.

\sautintro

Soit $U$ un groupe muni d'un endomorphisme $\fr$ tel que tout élément est fixé par une puissance de $\fr$. Munissons $U$ de la topologie discrète. Alors le groupe $\hat{\mathbb{Z}}$ des entiers profinis, de générateur topologique $\fr$, agit continûment sur $U$ et on note
\[ H^{1}(\fr,U):=H^{1}(\hat{\mathbb{Z}},U)\]
la cohomologie continue de $U$.

\bigskip

Pour un entier $d\geq 1$ et un $g \in U$ on note $N_{d}(g):=g\fr(g)\cdots \fr^{d-1}(g) \in U$. On peut alors voir les cocycles continus comme les éléments de
\[Z^{1}(\fr,U):=\{u \in U | N_{d}(u)=1 \text{ pour un certain } d\geq 1\}.\]
Alors $H^{1}(\fr,U)$ est le quotient de $Z^{1}(\fr,U)$ sous la $U$-action : $g*u=gu\fr(g)^{-1}$.

\bigskip

Dans le cas qui nous intéresse d'un groupe réductif connexe $\gpalg{\G}$ non-ramifié nous retrouvons la cohomologie Galoisienne. Notons $\fr$ l'automorphisme de $G^{nr}$ donné par l'action d'un Frobenius inverse. Alors la surjection naturelle $\Gal(\kalg/\kk)\rightarrow \Gal(\kalg/\kk) / \iner$ induit un isomorphisme
\[ H^{1}(\fr,G^{nr}) \simeq H^{1}(\kk,\gpalg{\G})\]
(voir par exemple \cite{DebackerReeder} section 2.2).

\bigskip

Prenons $\mathbf{T}$ un tore maximal $\knr$-déployé défini sur $\kk$ et maximalement déployé dans $\gpalg{\G}$. On note $N^{nr}$ le normalisateur de $T^{nr}$ dans $G^{nr}$, $\We:=N^{nr}/{}^{0}T^{nr}$ le groupe de Weyl étendu de $T^{nr}$ dans $G^{nr}$, où ${}^{0}T^{nr}$ est le sous-groupe borné maximal de $T^{nr}$, et $\Wa$ le groupe de Weyl affine, qui est le sous-groupe de $\We$ engendré par les réflexions des murs des chambres de $\mathcal{A}$. L'appartement $\mathcal{A}$ contient une chambre $C$ $\fr$-stable car $\mathbf{T}$ est maximalement déployé. Notons alors $N^{nr}_{C}:=\{n \in N^{nr} | n \cdot C=C\}$ et $\Omega_{C} := \{w \in \We | w \cdot C=C\}$. Notons que $\Omega_{C}$ est l'image de $N^{nr}_{C}$ dans $\We$ et que comme $\Wa$ agit simplement transitivement sur l'ensemble des chambres on a un isomorphisme $\Omega_{C} \tosim \We/\Wa$.

\begin{Pro}[\cite{DebackerReeder} Corollaire 2.4.3]
Les applications $N^{nr}_{C} \twoheadrightarrow \Omega_{C}$ et $N^{nr}_{C} \hookrightarrow G^{nr}$ induisent des isomorphismes sur les $H^{1}(\fr,\cdot)$ et donc un isomorphisme
\[ H^{1}(\fr,G^{nr}) \simeq H^{1}(\fr,\Omega_{C}).\]
\end{Pro}

L'isomorphisme $\Omega_{C} \overset{\sim}{\rightarrow} \We/\Wa$ nous donne donc un isomorphisme $H^{1}(\fr,G^{nr})\simeq H^{1}(\fr,\We/\Wa)$, valable pour l'instant pour un tore maximalement déployé $\mathbf{T}$.

\medskip

Notons que l'on peut simplifier l'isomorphisme $ H^{1}(\fr,G^{nr}) \simeq H^{1}(\fr,\Omega_{C})$ grâce à l'annexe de \cite{PappasRapoport} qui construit une surjection naturelle $G^{nr} \to \Omega_C$. Ceci donne une version géométrique du morphisme de Kottwitz que l'on peut construire de la façon suivante.

Soit $\phi : G_{sc} \to G$ le revêtement simplement connexe de $G$. Nous rappelons que $G_{sc}$ agit transitivement sur l'ensemble des chambres. Soit $g \in G^{nr}$. Prenons $h \in G_{sc}$ tel que $\phi(h)g$ stabilise $C$, qui est unique au fixateur de $C$ dans $G_{sc}$ près. Alors $\phi(h)g$ agit sur $\bar{C}$, la clôture de $C$, par une application affine. La construction même de l'immeuble de Bruhat-Tits (voir \cite[§7.4]{BTI}) nous dit que cette action sur $\bar{C}$ peut-être réalisée par un élément $n\in N^{nr}$. Puisque $n$ stabilise $C$, son image dans $\We$ est dans $\Omega_{C}$. Ceci nous fournit une application $G^{nr} \to \Omega_C$ qui peut être vue comme une version géométrique du morphisme de Kottwitz. Pour voir que c'est un morphisme de groupe naturel, on peut noter que $\phi(G_{sc}^{nr})$ est distingué dans $G^{nr}$. Donc $G^{nr}$ agir naturellement sur l'espace quotient $\bt(G,\knr)/{\phi(G_{sc}^{nr})}$, qui s'identifie au domaine fondamentale $\bar{C}$. Alors, l'action de $G^{nr}$ sur $\bar{C}$ nous donne un morphisme canonique $G^{nr} \to \Omega_C$.

\begin{Lem}
\label{lemNWH1}

Le diagramme suivant est commutatif

\[ \xymatrix{
H^{1}(\fr,N^{nr}) \ar@{->}[r] \ar@{->}[dr]& H^{1}(\fr,G^{nr}) \ar@{->}[d]^{\sim}\\
&H^{1}(\fr,\We/\Wa)} \]
où les flèches provenant de $H^{1}(\fr,N^{nr})$ correspondent aux applications $N^{nr}\hookrightarrow G^{nr}$ et $N^{nr}\twoheadrightarrow \We/\Wa$.

\end{Lem}

\begin{proof}

Remarquons que la version géométrique du morphisme de Kottwitz précédente nous fournit un diagramme commutatif :
\[ \xymatrix{
N^{nr} \ar@{->}[r] \ar@{->}[d]& G^{nr} \ar@{->}[d]\\
\We/\Wa \ar@{->}[r]& \Omega_C}. \]
En effet prenons un $g\in N^{nr}$. Notons que $\Wa$ est l'image de $\phi(G_{sc}) \cap N^{nr}$ dans $\We$. Alors pour ce $g$ nous pouvons prendre un $h \in G_{sc} \cap \phi^{-1}(N^{nr})$ et le morphisme de Kottwitz géométrique se factorise par $N^{nr} \to \We \to \We/\Wa \to \Omega_C$.

On obtient alors le diagramme désiré sur les $H^{1}(\fr,\cdot)$.
\end{proof}

Nous souhaitons étendre les résultats précédents à n'importe quel tore non-ramifié.

\bigskip

Soit donc $\mathbf{T}'$ un tore maximal $\knr$-déployé défini sur $\kk$ et $N'$, $\We'$ ,$\Wap$ définis comme avant pour $\mathbf{T}$. Il existe un $g \in G^{nr}$ tel que $T'^{nr}={}^{g}T^{nr}$ et on a alors $g\fr(g)^{-1} \in N'^{nr}$. Notons que si $n\in N$ alors $gn\fr(g)^{-1} =(gng^{-1})(g\fr(g)^{-1})\in N'^{nr}$. De plus cette application, $n \mapsto gn\fr(g)^{-1}$,  induit une application sur les $Z^{1}(\fr,\cdot)$ et envoie des cocycles cohomologues sur des cocycles cohomologues, donc induit une bijection $H^{1}(\fr,N^{nr})\overset{\sim}{\rightarrow} H^{1}(\fr,N'^{nr})$ de sorte que
\[ \xymatrix{
H^{1}(\fr,N^{nr}) \ar@{->}[rr]^{\sim} \ar@{->}[dr] &  & H^{1}(\fr,N'^{nr}) \ar@{->}[dl]\\
& H^{1}(\fr,G^{nr}) &} \]
commute.
\bigskip

Notons $w'\in \We'$ l'image de $g\fr(g)^{-1}$ par $N'^{nr} \twoheadrightarrow \We'$. Alors $n \mapsto gn\fr(g)^{-1}$ induit une application $\We \rightarrow \We'$, $v \mapsto gv\fr(g)^{-1}=(gvg^{-1})w$. Cette application passe au quotient et induit une bijection $\We/\Wa \overset{\sim}{\longrightarrow} \We'/\Wap$, en effet si $v_{1}=v_{2}v^{0}$ avec $v_{1},v_{2} \in \We$ et $v^{0} \in \Wa$, $(gv_{1}g^{-1})w=(gv_{2}v^{0}g^{-1})w=[(gv_{2}g^{-1})w](w^{-1}gv^{0}g^{-1}w)$ et comme $gv^{0}g^{-1} \in \Wap$ qui est normal dans $\We'$, $(w^{-1}gv^{0}g^{-1}w) \in \Wap$. Cette application passe aux cocycles continus et induit une bijection
\[ H^{1}(\fr,\We/\Wa) \overset{\sim}{\rightarrow} H^{1}(\fr,\We'/\Wap)\]
qui est indépendante du choix de $g$, et on a par construction le diagramme commutatif suivant

\[ \xymatrix{
H^{1}(\fr,N^{nr}) \ar@{->}[r]^{\sim} \ar@{->}[d] & H^{1}(\fr,N'^{nr}) \ar@{->}[d]\\
H^{1}(\fr,\We/\Wa) \ar@{->}[r]^{\sim} & H^{1}(\fr,\We'/\Wap)}\]

\bigskip

On vient donc de construire un isomorphisme, valable pour tout tore maximal non-ramifié $\mathbf{T}'$,
\[ H^{1}(\fr,G^{nr}) \simeq H^{1}(\fr,\We'/\Wap)\]
qui vérifie grâce au lemme \ref{lemNWH1} la proposition suivante

\begin{Pro}
\label{proCommuteNW}
On a une bijection $H^{1}(\fr,G^{nr}) \tosim H^{1}(\fr,\We'/\Wap)$ rendant le diagramme commutatif suivant
\[ \xymatrix{
H^{1}(\fr,N'^{nr}) \ar@{->}[r] \ar@{->}[dr]& H^{1}(\fr,G^{nr}) \ar@{->}[d]^{\sim}\\
&H^{1}(\fr,\We'/\Wap)} \]
\end{Pro}

Le résultat précédent étant valable pour tout tore non-ramifié, et pour éviter la multiplication des ' dans les formules, on ne suppose plus que $\mathbf{T}$ est maximalement déployé. Expliquons le lien entre l'isomorphisme $H^{1}(\fr,G^{nr}) \simeq H^{1}(\fr,\We/\Wa)$ et l'isomorphisme de Kottwitz. Notons $X=X_{*}(\mathbf{T})$. L'évaluation en $\varpi$, un élément de $\kk$ de valuation 1, permet d'identifier $\lambda \in X$ avec l'opérateur $t_{\lambda} \in \We$ de translation par $\lambda$. Notons $X^{a}:=X \cap \Wa$ et $\bar{X}:=X/X^{a}$. Alors l'application $X \rightarrow \We$, $\lambda \mapsto t_{\lambda}$ induit un isomorphisme $\bar{X}\overset{\sim}{\rightarrow} \We/\Wa$.

\begin{Lem}[\cite{DebackerReeder} Corollaire 2.4.2]
\label{lemdriso}
On a un isomorphisme
\[ [\bar{X}/(1-\vartheta)\bar{X}]_{\text{tor}} \overset{\sim}{\longrightarrow} H^{1}(\fr,\We/\Wa)\]
où la notation $[\cdot]_{\text{tor}}$ désigne les éléments de torsion.
\end{Lem}

D'après la section 2.5 de \cite{DebackerReeder}, on peut identifier $\bar{X}=\Hom(Z(\gpalg*{G}),\mathbb{C}^{\times})$ par restriction des caractères. De plus $\bar{X}/(1-\vartheta)\bar{X} = \Hom(Z(\gpalg*{G})^{\widehat{\vartheta}},\mathbb{C}^{\times})$ et les éléments de $\bar{X}/(1-\vartheta)\bar{X}$ s'annulant sur la composante identité de $Z(\gpalg*{G})^{\widehat{\vartheta}}$ sont exactement les éléments de torsion. D'où une identification
\[[\bar{X}/(1-\vartheta)\bar{X}]_{\text{tor}} = \Irr[\pi_{0}(Z(\gpalg*{G})^{\widehat{\vartheta}})].\]
Nous retrouvons la forme usuelle de l'isomorphisme de Kottwitz.
\section{Quelques isomorphismes sur les tores et leurs duaux}

\label{secIsoTores}

Dans cette section tous les groupes sont définis sur $\res$. Nous rappelons quelques isomorphismes sur les tores sur $\res$ qui nous sont utiles dans cet article. Nous suivons pour cela \cite{cabanes_enguehard} section 8.2.

\sautintro

Soit $\gpfinialg{T}$ un tore sur $\res$ et notons $X(\gpfinialg{T})$ ses caractères et $Y(\gpfinialg{T})$ ses co-caractères. On fixe un système compatible de racines de l'unité : $\iota : (\mathbb{Q}/\mathbb{Z})_{p'} \overset{\backsim}{\longrightarrow} \resalg^{\times}$ et $\iota': (\mathbb{Q}/\mathbb{Z})_{p'} \rightarrow \Ql^{\times}$ et l'on note
\[\kappa = \iota' \circ \iota^{-1} : \resalg^{\times} \to \Ql^{\times}.\]
Posons également $D=(\mathbb{Q}/\mathbb{Z})_{p'}$. L'application $\iota$ permet de définir un isomorphisme
\[\begin{array}{ccc}
 Y(\gpfinialg{T}) \otimes_{\mathbb{Z}} D & \to & \gpfinialg{T} \\
 \eta \otimes a & \mapsto & \eta(\iota(a)) \\
\end{array}\]
Via cet isomorphisme $\gpfinialg{T}^{\fr}$ peut être décrit comme le noyau de l'endomorphisme $\fr-1$ de $Y(\gpfinialg{T}) \otimes_{\mathbb{Z}} D$ :
\begin{align}
\label{suiteexacteTF}
1 \longrightarrow \gpfinialg{T}^{\fr} \longrightarrow Y(\gpfinialg{T}) \otimes_{\mathbb{Z}} D \overset{\fr-1}{\longrightarrow} Y(\gpfinialg{T}) \otimes_{\mathbb{Z}} D \longrightarrow 1.
\end{align}
Soit $\mathbb{Q}_{p'}$ le sous-groupe additif de $\mathbb{Q}$ des rationnels $r/s$, avec $r,s \in \mathbb{Z}$ et $s$ n'est pas divisible par $p$. Alors le lemme du serpent appliqué à
\[ \xymatrix{
		&&&1	\ar@{->}[d]&\\
		&&&\gpfinialg{T}^{\fr} \ar@{->}[d]&\\
		0  \ar@{->}[r]& Y(\gpfinialg{T}) \ar@{->}[r] \ar@{->}[d]^-{\fr-1}& Y(\gpfinialg{T}) \otimes_{\mathbb{Z}} \mathbb{Q}_{p'} \ar@{->}[r] \ar@{->}[d]^-{\fr-1}& Y(\gpfinialg{T}) \otimes_{\mathbb{Z}} D \ar@{->}[r] \ar@{->}[d]^-{\fr-1}& 0\\
		0  \ar@{->}[r]& Y(\gpfinialg{T}) \ar@{->}[r] \ar@{->}[d]& Y(\gpfinialg{T}) \otimes_{\mathbb{Z}} \mathbb{Q}_{p'} \ar@{->}[r]& Y(\gpfinialg{T}) \otimes_{\mathbb{Z}} D \ar@{->}[r]& 0\\
		&Y(\gpfinialg{T})/(\fr-1)Y(\gpfinialg{T}) \ar@{->}[d]\\
		&0
}\]
donne un isomorphisme entre $\gpfinialg{T}^{\fr}$ et le conoyau de $(\fr-1)$ sur $Y(\gpfinialg{T})$. De façon plus explicite, si $d$ est un entier tel que $\fr^d$ se déploie, alors en posant $\zeta=\iota(1/(q^d-1))$ on a
\[1 \longrightarrow Y(\gpfinialg{T}) \overset{\fr-1}{\longrightarrow} Y(\gpfinialg{T}) \overset{N_1}{\longrightarrow} \gpfinialg{T}^{\fr} \longrightarrow 1.\]
où $N_1(\eta)=N_{\fr^d/\fr}(\eta(\zeta))$.

En appliquant le foncteur $\Hom(\cdot,D)$ à (\ref{suiteexacteTF}) on obtient également la suite exacte
\[1 \longrightarrow X(\gpfinialg{T}) \overset{\fr-1}{\longrightarrow} X(\gpfinialg{T}) \overset{\Res}{\longrightarrow} \Irr(\gpfinialg{T}^{\fr}) \longrightarrow 1\]
où $\Res$ est la restriction de $\gpfinialg{T}$ à $\gpfinialg{T}^{\fr}$ (composée avec $\kappa$).

\bigskip

Prenons maintenant $\gpfinialg{G}$ un groupe réductif connexe sur $\res$ et $\gpfinialg{T}$ un $\res$-tore maximal de $\gpfinialg{G}$. Nous pouvons alors leur associer une donnée radicielle $(X,Y,\phi,\phi^{\vee})$. L'endomorphisme de Frobenius $\fr$ agit sur l'ensemble des sous-groupes connexes unipotents 1-dimensionnels de $\gpfinialg{G}$ qui sont normalisés par $\gpfinialg{T}$ et donc induit une permutation $f$ de $\phi$ définie par $F(X_{\alpha})=X_{f\alpha}$ pour $\alpha \in X$. Nous avons également une action de $\fr$ sur $X$ et sa transposé, notée $\fr^{\vee}$, agit sur $Y$. On obtient alors des données
\[ \fr : X \to X,\ \fr^{\vee} : Y \to Y,\ q : \phi \to \{p^n\}_{n\in \mathbb{N}},\ f:\phi \to \phi\]
\[\fr(f\alpha)=q(\alpha)\alpha,\ \fr^{\vee}(\alpha^{\vee})=q(\alpha)(f\alpha)^{\vee},\ \alpha \in \phi\]
Réciproquement, si $\gpfinialg{G}$ définit une donnée radicielle $(X,Y,\phi,\phi^{\vee})$ alors n'importe quel quadruplet $(\fr,\fr^{\vee},f,q)$ vérifiant les conditions précédentes peut être réalisé par une isogénie $\fr : \gpfinialg{G} \to \gpfinialg{G}$.

Prenons maintenant $\gpfinialg*{G}$ un dual de $\gpfinialg{G}$ (autour des tores $\gpfinialg*{T}$ et $\gpfinialg{T}$). Alors $(\fr^{\vee},\fr,f^{-1},q\circ f^{-1})$ définit un quadruplet qui est réalisé par une isogénie $\fr^{*} : \gpfinialg*{G} \to \gpfinialg*{G}$. Nous dirons alors que $(\gpfinialg{G},\fr)$ et $(\gpfinialg*{G},\fr^{*})$ sont en dualités sur $\res$.

La discussion précédente montre que l'on a un isomorphisme

\[\begin{array}{ccc}
 (\gpfinialg*{T})^{\fr^{*}} & \tosim & \Irr(\gpfinialg{T}^{\fr}) \\
 s & \mapsto & \theta \\
\end{array}\]
où $\theta$ et $s$ sont reliés par la relation (quel que soit $\eta \in Y(\gpfinialg{T})=X(\gpfinialg*{T})$)
\[ \theta(N_{\fr^d/\fr}(\eta(\zeta))) = \kappa(\eta(s)).\]

Cette bijection est donnée à partir d'un choix des tores $\gpfinialg{T}$ et $\gpfinialg*{T}$ mais peut devenir indépendant du choix du tore de la façon suivante

\begin{Pro}[\cite{cabanes_enguehard} Proposition 8.21]
\label{proBijPairesSt}
L'application précédente fournit une bijection entre les classes de $\gpfinialg{G}^{\fr}$-conjugaison de paires $(\gpfinialg{T},\theta)$, où $\theta \in \Irr(\gpfinialg{T}^{\fr})$, et les classes de $(\gpfinialg*{G})^{\fr^{*}}$-conjugaison de paires $(\gpfinialg*{T},s)$, où $s \in (\gpfinialg*{T})^{\fr^{*}}$.

\end{Pro}

Notons que l'on peut réinterpréter cette bijection de la façon suivante. Fixons $\gpfinialg{S}$ et $\gpfinialg*{S}$ deux tores maximaux $\fr$-stables en dualités sur $\res$, qui serviront de tores de référence. Prenons une paire $(\gpfinialg{T},\theta)$. Il existe alors $g \in \gpfinialg{G}$ tel que $\gpfinialg{T}={}^{g}\gpfinialg{S}$. Comme $\gpfinialg{T}$ et $\gpfinialg{S}$ sont $\fr$-stables, l'élément $g^{-1}\fr{g}$ normalise $\gpfinialg{S}$ et l'on note $w$ son image dans $\Wf$. Définissons également $\theta_w:=g^{-1}\theta \in \Irr(\gpfinialg{S}^{w\fr})$. La paire $(w,\theta_w)$ est bien définie à $\fr$-conjugaison près et sa classe de $\fr$-conjugaison caractérise la classe de $\gpfinialg{G}^{\fr}$-conjugaison de $(\gpfinialg{T},\theta)$.

Soit $t_w \in (\gpfinialg*{S})^{w\fr^{*}}$ image de $\theta_w$ par la bijection $\Irr(\gpfinialg{S}^{w\fr}) \tosim   (\gpfinialg*{S})^{w\fr^{*}}$. Identifions $\Wf$ au groupe de Weyl de $\gpfinialg*{G}$ relatif à $\gpfinialg*{S}$. Prenons maintenant $g^{*} \in \gpfinialg*{G}$ tel que  $(g^{*})^{-1}\fr(g^{*})$ soit un relèvement de $w$ ($g^{*}$ existe bien par surjectivité de l'application de Lang). Définissons $(\gpfinialg*{T},s)$ par $(\gpfinialg*{T},s):=({}^{g^{*}}\gpfinialg*{S},\Ad(g^{*})(t_w))$. Alors la classe de conjugaison de la paire $(\gpfinialg*{T},s)$ est celle en dualité avec la classe de $(\gpfinialg{T},\theta)$ comme dans la proposition \ref{proBijPairesSt}.

\bibliographystyle{hep}
\bibliography{biblio}
\end{document}